\documentclass[a4paper,11pt]{article}

\usepackage[utf8]{inputenc}
\usepackage[T1]{fontenc}
\usepackage[english]{babel}
\usepackage{csquotes}
\usepackage{authblk}
\usepackage{titling}
\usepackage{titlesec}
\usepackage{graphicx}
\usepackage{float}

\usepackage{mathtools}
\usepackage{amssymb}
\usepackage{amsthm}
\usepackage{amsmath}
\allowdisplaybreaks
\usepackage{cases}
\usepackage{mathtools}
\usepackage{booktabs}
\usepackage{xcolor}
\newcommand{\noeoc}{\textcolor{black}{\textbf{\textemdash}}}

\usepackage{booktabs}
\usepackage{graphicx}
\usepackage{caption}
\usepackage{xcolor}

\providecommand{\noeoc}{%
    \textcolor{black}{\textbf{\textemdash}}%
}

\usepackage{caption}
\usepackage{subcaption}

\usepackage{xcolor}
\usepackage{listings}

\definecolor{codegreen}{rgb}{0,0.6,0}
\definecolor{codegray}{rgb}{0.5,0.5,0.5}
\definecolor{codepurple}{rgb}{0.58,0,0.82}
\definecolor{backcolour}{rgb}{0.95,0.95,0.92}

\lstdefinestyle{mystyle}{
    backgroundcolor=\color{backcolour},   
    commentstyle=\color{codegreen},
    keywordstyle=\color{magenta},
    numberstyle=\tiny\color{codegray},
    stringstyle=\color{codepurple},
    basicstyle=\ttfamily\footnotesize,
    breaklines=true,
    breakatwhitespace=false,
    captionpos=b,
    keepspaces=true,
    numbersep=5pt,
    showspaces=false,
    showstringspaces=false,
    showtabs=false,
    tabsize=2
}
\usepackage{tabularx}

\usepackage[
    backend=biber,
    style=phys,
    sorting=nyt,
    biblabel=brackets,
    maxbibnames=99,
    articletitle=true,
    doi=false
]{biblatex}
\DeclareFieldFormat[article]{title}{#1}

\usepackage[
    a4paper,
    left=1.8cm,
    right=1.8cm,
    top=1.5cm,
    bottom=1.4cm
]{geometry}

\usepackage{enumitem}

\newtheorem{theorem}{Theorem}[section]

\theoremstyle{definition}

\theoremstyle{remark}
\newtheorem{remark}[theorem]{Remark}

\theoremstyle{proposition}
\newtheorem{proposition}[theorem]{Proposition}
\newtheorem{corollary}[theorem]{Corollary}
\newtheorem{lemma}[theorem]{Lemma}

\theoremstyle{assumption}
\newtheorem{assumption}{Assumption}[section]

\newcommand{\f}[1]{\textbf{\textit{#1}}}
\newcommand{\curl}{\textbf{curl}}

\newcommand{\norm}[1]{\lVert#1\rVert}

\DeclareMathOperator*{\argmin}{arg\, min}
\newcommand{\citegen}[3][]{%
    \if\relax\detokenize{#1}\relax
        \cite{#2}%
    \else
        \cite[, #1~#3]{#2}%
    \fi
}

\titleformat{\section}[runin]
  {\normalfont\bfseries}
  {\thesection.}
  {0.5em}
  {}
  [.]

\titlespacing*{\section}
  {0pt}
  {1.0em}
  {0.6em}

  \titleformat{\subsection}[runin]
  {\normalfont\bfseries}
  {\thesubsection.}
  {0.5em}
  {}
  [.]

\titlespacing*{\subsection}
  {0pt}
  {0.9em}
  {0.6em}
\usepackage[colorlinks=true, linkcolor=black, urlcolor=black, citecolor=black]{hyperref}
\usepackage{cleveref}
\title{%
  \bfseries\Large
  \MakeUppercase{Finite Element Analysis of Maxwell's Equations with
  Time-Varying Coefficients and Nonlinear E-H Coupling: Convergence and Exponential Decay}\thanks{This work is supported  by   DFG research grants   YO 159/4-1  (project number: 455353159) and YO 159/5-1 (project number: 513566305)} %
}

\author{ Jan Renner\thanks{Universit\"at Duisburg-Essen, Fakult\"at f\"ur Mathematik, Thea-Leymann-Str. 9, D-45127 Essen,
Germany. Email: jan.renner@uni-due.de, irwin.yousept@uni-due.de }  \phantom{} and \phantom{}    Irwin Yousept$^{\dagger}$\thanks{Corresponding author. }}
\date{}

\begin{document}
 
\maketitle

\vspace{-2.1cm}
\begin{abstract}
\vspace{-0.2cm}
This paper analyzes  finite element approximations for Maxwell's equations with time-dependent coefficients and nonlinear E-H coupling     in both the total current density and   the  Silver--M\"uller-type boundary condition. The nonlinear coupling and the time-varying material parameters prevent the use of standard analytical tools, including the discrete compactness property  and the Minty--Browder argument, making the numerical analysis  particularly challenging. As the main novelty, we establish two primary contributions: Uniform convergence on every finite time interval $[0,T]$ and exponential stability on the infinite time horizon $[0,\infty)$. The uniform convergence is established via a  Cauchy-type argument within a time-dependent Hilbert space framework that entirely bypasses Minty's trick and the discrete compactness argument. Furthermore, under a positive homogeneity assumption on the nonlinearities, the exponential stability is analyzed by means of a nonlinear sup-max   problem arising from ratios of consecutive discrete energies. With this strategy, we ultimately prove   a fully discrete unconditional exponential decay result under a specific sufficient condition. In the absence of this condition, a concrete counterexample shows that the  energy decay fails to hold. Numerical experiments are provided to validate the theoretical findings. 

\vspace{0.1cm}
\noindent \textbf{Keywords:} {Finite element analysis, non-autonomous Maxwell equations, nonlinear E-H coupling, uniform convergence, fully discrete unconditional exponential decay.}
 
\end{abstract}

\setlength{\abovedisplayskip}{10pt}
\setlength{\belowdisplayskip}{10pt}

\section{Introduction}
This paper    analyzes   finite element approximations  for the   nonlinear Maxwell system
\begin{subequations} 
\begin{equation}\label{eq:maxwell intro 1}
    \begin{cases}
      \partial_t\f{D}(x,t) - \curl\, \textbf{\textit{H}}(x,t) + \f{J}_{tot}(x,\f{E}(x,t),\f{H}(x,t),t) = \textbf{\textit{f}}(x,t) & \text{in } \Omega \times (0,T)\\
      \partial_t\f{B}(x,t)+ \curl\, \textbf{\textit{E}}(x,t) = 0 & \text{in } \Omega \times (0,T)\\
      (\f{E},\f{H})(\cdot,0)=(\f{E}_0,\f{H}_0) & \text{in } \Omega
    \end{cases}       
\end{equation}
under the non-autonomous constitutive laws
\begin{equation}\label{eq:maxwell intro 2}
\f{D}(x,t) = \varepsilon(x,t) \f{E}(x,t),\quad \f{B}(x,t)=\mu(x,t)\f{H}(x,t) \quad \text{ in } \Omega \times (0,T)
\end{equation}
and the generalized Silver-M\"uller  boundary condition 
\begin{equation}\label{eq:maxwell intro 3}
\boldsymbol{\nu} \times \textbf{\textit{H}} = \boldsymbol{\nu} \times  \textbf{\textit{b}}(\cdot, \boldsymbol{\nu} \times \textbf{\textit{E}})  \quad \text{on } \partial \Omega \times (0,T)\end{equation}
\end{subequations}
for some given nonlinearity $\f{b}:\partial \Omega \times \mathbb R^3 \to \mathbb R^3$. In the setting of \eqref{eq:maxwell intro 1}-\eqref{eq:maxwell intro 3}, $\Omega$ denotes a bounded Lipschitz polyhedral domain, $\f{D}:  \Omega \times (0,T) \to \mathbb R^3$  the displacement field, $\f{E}: \Omega \times (0,T) \to \mathbb R^3$  the electric field, $\f{B}:\Omega\times (0,T)\to \mathbb{R}^{3}$ the magnetic flux density, $\f{H}: \Omega \times (0,T) \to \mathbb R^3$ the magnetic field, $\f{f}: \Omega \times (0,T) \to \mathbb R^3$ the applied current source, and $\boldsymbol{\nu}: \partial \Omega \to \mathbb R^3$ the  outer unit normal vector to $\partial \Omega$. The underlying material parameters comprise the time-dependent anisotropic electric permittivity $\varepsilon : \Omega \times (0,T) \to \mathbb R^{3\times 3}$ and the time-dependent anisotropic magnetic permeability $\mu: \Omega \times (0,T)\to \mathbb R^{3\times 3}$. Furthermore,  $\f{J}_{tot}: \Omega \times \mathbb R^3 \times  \mathbb R^3 \times  (0,T) \to \mathbb R^3$ represents the total current density featuring nonlinear $\f E$-$\f H$ coupling and explicit time dependence. The precise mathematical assumptions for all given data  involved in \eqref{eq:maxwell intro 1}-\eqref{eq:maxwell intro 3}  are specified in Assumption \ref{ass}. 

The main characteristic   of our problem is the simultaneous presence of time-varying  material parameters  \eqref{eq:maxwell intro 2} and  the  nonlinear $\f E$-$\f H$ coupling    in the total current density \eqref{eq:maxwell intro 1} and  in the boundary condition \eqref{eq:maxwell intro 3}. Time-dependent phenomena in 
$\varepsilon$ and $\mu$ naturally arise in materials whose dielectric and magnetic properties can be strongly influenced by  external or internal  time-dependent factors such as varying temperature and evolving microstructure. Furthermore, while dependence on the electric field in   $\f{J}_{tot}$ is well-known in the context of Ohm's law and its nonlinear  extension,   dependence on the   magnetic field  occurs in more complex electromagnetic settings. A prominent  example is type-II superconductivity, where $\f{J}_{tot}$ is governed by a critical current with a strong dependence on the magnetic field and temperature (see \cite{Kim,MR4780410}). 
Finally, the nonlinear Silver-M\"uller boundary condition \eqref{eq:maxwell intro 3} is considered to ensure energy stabilization as  $T \to \infty$ and to circumvent undesired spurious reflections when considering artificial boundaries. Regarding the exponential decay of the continuous solution to Maxwell's equations under \eqref{eq:maxwell intro 3}, we refer to the seminal works  by Nicaise et al. \cite{nicaise1,nicaise2}  for   time-independent coefficients and  by  Nicaise and Pignotti  \cite{Nicaise3} for time-varying material parameters.

This paper is devoted to the finite element analysis of \eqref{eq:maxwell intro 1}-\eqref{eq:maxwell intro 3}, focusing on the associated $\f{E}$-$\f{H}$ formulation and its fully discrete scheme \eqref{time-space-discrete}  based on implicit Euler time stepping combined with N\'ed\'elec edge elements and piecewise constant elements in the spirit of Monk \cite{monk2,monk3,monk5}. While this mixed method is not new and has been widely considered in the literature (see \cite{slodicka1,MR3013583,MR2229840,MR2846772}), to the best of our knowledge,   finite element approximations of the present non-autonomous Maxwell system with simultaneous  nonlinear distributed  and boundary $\f E$-$\f H$ coupling have not yet been analyzed. The interplay of these features makes the numerical analysis particularly challenging. Our primary contributions lie in the development of new techniques   enabling us to establish a  uniform convergence result on the finite time interval $[0,T]$ and a fully discrete unconditional exponential decay result on the infinite time horizon $[0,\infty)$.   We briefly summarize the key ideas below.

\vspace{-1em}
\paragraph{Uniform convergence.}
Two analytical tools are commonly used in the convergence analysis of nonlinear Maxwell-type problems: (i) the Minty-Browder argument (cf. \cite{slodicka1,MR2846772} for its application) and (ii) the discrete compactness property \cite{Kikuchi} for   N\'ed\'elec edge elements (cf.   \cite{MR4780410} for its application). However, neither tool is applicable to the Maxwell system \eqref{eq:maxwell intro 1}--\eqref{eq:maxwell intro 3}. Indeed, even if the nonlinear total current density  $\f{J}_{tot}$ is assumed to be monotone with respect to its second and third arguments, the resulting Maxwell system \eqref{eq:maxwell intro 1}-\eqref{eq:maxwell intro 3} and its fully discrete scheme \eqref{time-space-discrete} are in general not monotone (cf. Remark \ref{rem:notmonotone}). In particular, this property is lost by   the time-varying material parameters    appearing in the time derivatives and the  $\f{H}$-dependence
in $\f{J}_{tot}$. To overcome the inapplicability of  the tools (i) and (ii), we develop a     Cauchy-type argument within a  time-dependent Hilbert space  framework, leading to a uniform convergence result (see Theorem \ref{theorem:convergence}). To the best of our knowledge, this approach is novel and constitutes the primary ingredient of our convergence analysis. In particular, we believe that these arguments could be adapted to other nonlinear  problems where the aforementioned tools fail to apply.  
\vspace{-1em}
\paragraph{Fully discrete exponential decay.} While exponential stability in the continuous case is   well understood from the aforementioned contributions by Nicaise et al. \cite{nicaise1,nicaise2,Nicaise3} from the early 2000s, the corresponding fully discrete counterpart has remained an open question. So far, we   are only aware of the recent result by Egger et al. \cite{egger3} for the linear setting with time-independent coefficients under the  PEC boundary condition $   \boldsymbol{\nu} \times \f{E}=0$ (see also \cite{MR3778338,Zuazua} for related results on linear wave equations). To the best of our knowledge, the present paper is the first to analyze and prove a fully discrete unconditional exponential decay result in a nonlinear and non-autonomous   Maxwell setting. Our analysis relies on a positive homogeneity assumption on the nonlinearities $\f{b}$ and $\f{J}_{tot}$ and the consideration of  the following nonlinear sup-max   problem:
$$
\sup_{n \in \mathbb N } \max_{(v,w)\in \mathcal{W}_h\setminus \{(0,0)\}}\frac{\mathcal{E}_h^n(S_{h,\tau}^n(\f{v},\f{w}))}{\mathcal{E}_h^{n-1}(\f{v},\f{w})}. 
 $$
  Here, $S^n_{h,\tau}$ denotes the solution operator associated with the finite element scheme at   $t_n$ (see \eqref{eq:S_tau,h} for its definition) and $\mathcal{E}_h^n$ (resp. $\mathcal{E}_h^{n-1}$) denotes the electromagnetic energy functional (see \eqref{eq:discrete-EME} for its definition) at   $t_n$ (resp. $t_{n-1}$).  We propose a sufficient condition \ref{ass:B5} that, together with the analysis of the above sup-max problem, yields a fully discrete unconditional exponential decay result (see Theorem \ref{theorem-exp-decay} and Corollary \ref{corollary:exponential-decay}).  The significance of the proposed sufficient condition \ref{ass:B5}  is illustrated by a concrete counterexample showing that, in its absence,  exponential energy decay   fails to hold.

 The rest of this paper is organized as follows: In the next section, we introduce our notation and the standing assumptions. Section \ref{sec:3} is devoted to  the finite element  analysis of   \eqref{eq:maxwell intro 1}-\eqref{eq:maxwell intro 3}, comprising the definitions and properties of the applied finite element spaces  (Sect. \ref{sec:3.1}), the derivation and well-posedness of the fully discrete scheme (Sect. \ref{section: existence}),  the stability analysis (Sect. \ref{section: stability}), the convergence analysis (Sect. \ref{section: convergence}) and the   exponential decay analysis (Sect. \ref{section:exponentialstability}). The final section   presents numerical  experiments.


\section{Preliminaries}
We start by introducing the notation and the standing assumptions of this paper. For a given normed space $V$, we write $\norm{\cdot}_V$ for its norm, $V^*$ for its dual space, and $\langle\cdot.\cdot\rangle_{V^*,V}$ for the associated duality pairing. If $V$ is a Hilbert space, then $(\cdot,\cdot)_V$  represents the scalar product of $V$. If $V=\mathbb{R}^3$, we simply write $a\cdot b$ instead of $(a,b)_{\mathbb{R}^3}$. Functions with values in $\mathbb{R}^3$ and spaces of such functions are denoted using a bold symbol. We   frequently deal with functions of the form $\f{u}:\Omega \times [0,T]\to \mathbb{R}^3$, which we   also interpret as   a Banach-valued  mapping $t\mapsto \f{u}(\cdot, t)$ without further mention.

Let us now introduce the Hilbert space
\begin{equation*}
        \f{H}(\curl) \coloneqq \{\f{u}\in \f{L}^2(\Omega): \curl\,\f{u} \in \f{L}^2(\Omega)\},
\end{equation*}
where the $\curl$-operator is to be understood in a distributional sense.  Furthermore, let  $\tau:H^1(\Omega)\to L^2(\partial \Omega)$ denote the   (Sobolev) trace operator and  
$$
H^{1/2}(\partial \Omega):= \{\tau y :  y \in H^1(\Omega)\}, \quad H^{-1/2}(\partial \Omega):=H^{1/2}(\partial \Omega)^*.
$$
It is well-known (cf. \citegen[Theorem]{monk}{3.29}) that there exists a linear and bounded operator $\gamma_t:\f{H}(\curl)\to \f{H}^{-\frac{1}{2}}(\partial \Omega)$, the so-called tangential trace operator, 
%
satisfying 
\begin{align}\label{eigen trace}
\langle \gamma_t(\f{u}),\tau(\boldsymbol{\phi})\rangle_{\f{H}^{-1/2}(\partial \Omega), \f{H}^{1/2}(\partial \Omega)}=\int_\Omega \curl\,\f{u}\cdot \boldsymbol{\phi}- \f{u}\cdot \curl\,\boldsymbol{\phi}\,dx \quad \forall \f{u}\in \f{H}(\curl), \boldsymbol{\phi}\in \f{H}^1(\Omega).
\end{align}
If  $\gamma_t(\f{u})\in \f{L}^2(\partial \Omega)$ holds for  $\f{u}\in \f{H}(\curl)$, then we formally write   $\boldsymbol{\nu}\times \f{u} := \gamma_t(\f{u}) \in \f{L}^2(\partial \Omega)$.  The main Hilbert space   involved in our analysis is given by
$$\f{V} \coloneqq\{\boldsymbol{\phi} \in \textbf{\textit{H}}(\textbf{curl}): \gamma_t(\boldsymbol{\phi}) \in \textbf{\textit{L}}^2(\partial \Omega) \textrm{ and } \boldsymbol{\nu} \cdot \gamma_t(\boldsymbol{\phi}) = 0\text{ on } \partial \Omega\},$$
endowed with the scalar product
$$
(\boldsymbol{\phi},\boldsymbol{\psi})_{\f{V}}  \coloneqq
 \int_{\Omega}   \boldsymbol{\phi} \cdot \boldsymbol{\psi}  +  \curl\, \boldsymbol{\phi} \cdot\curl\, \boldsymbol{\psi}  \,dx + \int_{\partial \Omega}   \boldsymbol{\nu}\times \boldsymbol{\phi} \cdot \boldsymbol{\nu}\times \boldsymbol{\psi} \,dS \quad \forall \boldsymbol{\phi},\boldsymbol{\psi} \in \f{V} $$
 and the induced norm
 \begin{equation} \label{def: norm V}
 \|\boldsymbol{\phi}\|_\f{V} := (\boldsymbol{\phi},\boldsymbol{\phi})_{\f{V}}^{1/2} =( \|  \boldsymbol{\phi}\|_{\f{H}(\curl)}^2+ \|\boldsymbol{\nu}\times \boldsymbol{\phi}\|_{\f{L}^2(\partial \Omega)}^2)^{1/2} \quad \forall \boldsymbol{\phi} \in \f{V}. 
 \end{equation}
From the density of $\f{C}^\infty(\overline{\Omega})$ in $\f{V}$ (cf. \citegen[Theorem]{monk}{3.54})   and Gauss's theorem, we obtain via a standard approximation argument the following partial integration formula for any vector fields $\f{u},\f{v}\in \f{V}$: 
\begin{equation}\label{eq:partial}
    \int_{\Omega} \textbf{\textit{u}}\cdot \curl\, \textbf{\textit{v}} - \curl\,\textbf{\textit{u}}\cdot  \textbf{\textit{v}}\,dx = \int_{\partial \Omega} \textbf{\textit{u}}_{\tau}\cdot (\boldsymbol{\nu} \times \textbf{\textit{v}})\,dS,  
\end{equation}
 where $\f{u}_\tau\coloneqq (\boldsymbol{\nu}\times \f{u})\times \boldsymbol{\nu}$ denotes the a.e. defined pointwise cross product of $\boldsymbol{\nu}\times \f{u}$ and $\boldsymbol{\nu}$.   We close this section by presenting the standing assumptions for our numerical analysis:
\begin{assumption}\label{ass}
\begin{enumerate}[label=(A\arabic*)]
{\normalfont
\item[]
\item \phantomsection\label{ass:A1} The coefficients satisfy  $\varepsilon,\mu\in W^{2,\infty}(0,T;L^\infty_\text{sym}(\Omega)^{3\times 3})$, and there exist $\underline{\varepsilon},\overline{\varepsilon},\underline{\mu},\overline{\mu}>0$ such that
\begin{equation}\label{eq:uniformposdef}
    \underline{\varepsilon}\lvert \xi\rvert^2\leq \xi^T \varepsilon(x,t)\xi\leq \overline{\varepsilon}\lvert \xi\rvert^2\quad \text{and}\quad \underline{\mu}\lvert \xi\rvert^2\leq \xi^T \mu(x,t)\xi\leq \overline{\mu}\lvert \xi\rvert^2  
    \end{equation} 
hold for a.e. $x\in \Omega$    and all $\xi\in \mathbb{R}^3,t\in [0,T]$.
\item \phantomsection\label{ass:A2} The external current source fulfills $\f{f}\in H^1(0,T;\f{L}^2(\Omega))$.
\item \phantomsection\label{ass:A3} The initial value $(\f{E}_0,\f{H}_0)\in \f{V}\times \f{V}$ satisfies the compability condition $\boldsymbol{\nu}\times \f{H}_0 = \boldsymbol{\nu}\times \f{b}(\cdot, \boldsymbol{\nu}\times \f{E}_0).$
\item \phantomsection\label{ass:A4} The total current density $\f{J}_{tot} :\Omega\times \mathbb{R}^3\times \mathbb{R}^3\times [0,T]\to \mathbb{R}^3$ is measurable w.r.t. the first variable and satisfies $\f{J}_{tot}(\cdot,0,0,0)\in \f{L}^2(\Omega)$. Furthermore, there exists a constant $L_{tot}>0$ such that
\begin{equation}\label{eq:J_tot-estimate}
\lvert \f{J}_{tot}(x,\xi,\eta,t)-\f{J}_{tot}(x,\xi',\eta',t')\rvert \leq L_{tot}\left(\lvert \xi - \xi'\rvert +\lvert \eta-\eta'\rvert+ (1+\lvert \xi\rvert + \lvert \xi'\rvert+\lvert \eta\rvert+\lvert \eta'\rvert)\lvert t-t'\rvert\right)
\end{equation}
holds for a.e. $x\in \Omega$ and all $\xi,\xi',\eta,\eta'\in \mathbb{R}^3, t,t'\in [0,T].$    
\item \phantomsection\label{ass:A5} The nonlinearity $\textbf{\textit{b}}: \partial \Omega \times \mathbb{R}^3 \to \mathbb{R}^3$ is measurable w.r.t. the first variable  and continuous and montone w.r.t. the second one. Furthermore, $ \textbf{\textit{b}}(\cdot,0)=0$, and there exist  $\underline b, \overline b, K>0$ such that  
\begin{equation}\label{eq:propertyb}\! \!\!\!\! \!\! \!\!\textbf{\textit{b}}(x,\xi) \! \cdot \!  \xi\geq \underline b\lvert \xi\rvert^2 \    \textrm{f.a.a. }   x\in \partial \Omega   \text{ and all }  | \xi | \ge K , \quad    \lvert  \textbf{\textit{b}}(x,\xi) \rvert \leq \overline b (\lvert \xi\rvert \!+ \!1)  \  \textrm{f.a.a. }  x\in \partial \Omega \text{ and all }  \xi \in \mathbb R^3.\end{equation}
 
\item \phantomsection\label{ass:A6}It holds that $\f{E}_0,\curl\,\f{E}_0 \in \f{H}^1(\Omega)$ and  $\f{b}:\f{L}^2(\partial \Omega) \to \f{L}^2(\partial \Omega),\f{u}\mapsto \f{b}(\cdot,\f{u}),$ is   Lipschitz continuous at $\boldsymbol{\nu}\times \f{E}_0$.}
\end{enumerate}
\end{assumption}
\begin{remark}  
Note that \ref{ass:A6} is only required for the stability   but not for the well-posedness. Moreover,  under the monotonicity assumption on $\f{b}$, the second condition in \eqref{eq:propertyb} is equivalent to $\textbf{\textit{b}}(x,\xi) \rvert \leq \overline b \lvert \xi\rvert $ for a.e.  $x\in \partial \Omega$   and all  $|\xi| \ge K$, which corresponds to the growth condition in \cite{nicaise1}.  \end{remark}

%
%

\section{Finite element analysis} \label{sec:3}
This section analyzes   the finite element discretization of \eqref{eq:maxwell intro 1}-\eqref{eq:maxwell intro 3}    in terms of well-posedness, stability,   convergence, and ultimately exponential decay. As preparation, let us begin by introducing and discussing the involved finite element spaces and their   respective properties. 

\subsection{Finite element spaces} \label{sec:3.1}
In all that follows, let Assumption \ref{ass} hold and let $\{\mathcal{T}_h\}_{h>0}$ denote a quasi-uniform family of triangulations for $\Omega$. We now introduce the Nédélec finite element space of the first family \cite{ned80} and the piecewise constant finite element space by
\begin{align*}
\textbf{ND}_h&\coloneqq \{\f{v}_h \in \f{H}(\curl) : \f{v}_h\rvert_T = a_T + b_T \times \cdot\text{ for some } a_T,b_T\in \mathbb{R}^3\quad \forall T\in \mathcal{T}_h\},\\
\textbf{DG}_h&\coloneqq \{\f{w}_h\in \f{L}^2(\Omega):\f{w}_h\rvert_T = a_T\text{ for some } a_T\in \mathbb{R}^3 \quad \forall T\in \mathcal{T}_h\}.
\end{align*}
A well-known property of   $\textbf{ND}_h$ is the  discrete inverse curl estimate: 
\begin{equation}\label{eq:inverse}
  \exists C_{\text{inv}}>0 \quad \forall \f{v}_h\in \textbf{ND}_h \,:  \quad  \norm{\curl\,\f{v}_h}_{\f{L}^2(\Omega)}\leq \frac{C_\text{inv}}{h}\norm{\f{v}_h}_{\f{L}^2(\Omega).}
\end{equation}
Let us next introduce the classical $\f{L}^2$-Hilbert projection operator onto $\textbf{DG}_h$ by 
\begin{equation}\label{def:Q_h}
    \f{Q}_h : \f{L}^2(\Omega) \to \textbf{DG}_h,\quad \f{Q}_h \f{w} \coloneqq \argmin_{\f{w}_h\in \textbf{DG}_h} \|\f{w}_h - \f{w}\|_{\f{L}^2(\Omega)} = \sum_{T\in \mathcal{T}_h} \chi_T \frac{1}{\lvert T\rvert}\int_T \f{w}\,dx,
\end{equation}
which satisfies the well-known properties
\begin{equation}\label{eq:qhest}
 \quad (\f{Q}_h\f{w} - \f{w}, \f{w}_h)_{\f{L}^2(\Omega)} = 0\quad
 \textrm{and} \quad  \lim_{h \to 0} \|\f{Q}_h\f{w} -\f{w} \|_{ \f{L}^2(\Omega)} =0 \quad \forall \f{w}\in \f{L}^2(\Omega),\f{w}_h\in \textbf{DG}_h.
\end{equation}
In particular, we have
\begin{equation}\label{eq:qhest2}
\norm{\f{Q}_h \f{w}}_{\f{L}^2(\Omega)}\leq \norm{\f{w}}_{\f{L}^2(\Omega)} \quad \forall \f{w} \in \f{L}^2(\Omega),\quad \norm{\f{Q}_h\f{w}}_{\f{L}^\infty(\Omega)}\leq \norm{\f{w}}_{\f{L}^\infty(\Omega)}\quad \forall \f{w}\in \f{L}^\infty(\Omega).
    \end{equation}
    Note that the same notation   is used for the componentwise extension of $\f{Q}_h$ to matrix-valued functions. All properties \eqref{eq:qhest}  and \eqref{eq:qhest2} remain valid in this setting.
Analogously, since $\textbf{ND}_h\subset \f{V}$ is a closed subspace, the  Hilbert projection theorem allows us to define the following Hilbert projection operator:
\begin{equation} \label{eq:hilbertpro}
\Pi_h:\f{V}\to \textbf{ND}_h, \quad   \Pi_h \f{v} \coloneqq \argmin_{\f{v}_h\in  \textbf{ND}_h} \|\f{v}_h - \f{v}\|_{\f{V}},
\end{equation}
which similar to above satisfies
\begin{equation}\label{eq:Pi_h proj}
  (\Pi_h\f{v} - \f{v}, \f{v}_h)_{\f{V}} = 0\quad \forall \f{v}\in \f{V},\f{v}_h\in \textbf{ND}_h \quad \Rightarrow \quad   \norm{\Pi_h \f{v}}_{\f{V}}\leq \norm{\f{v}}_{\f{V}}  \quad \forall \f{v}\in \f{V}.
    \end{equation} 
In   the next lemma (cf. Appendix \ref{appen b} for its proof), we provide an error estimate for $\Pi_h$ by using  the space
$$\f{H}^s(\curl)\coloneqq \{\f{u}\in \f{H}^s(\Omega):\curl\,\f{u}\in \f{H}^s(\Omega)\}, \quad\norm{\f{u}}_{\f{H}^s(\curl)}\coloneqq \norm{\f{u}}_{\f{H}^s(\Omega)} + \norm{\curl\,\f{u}}_{\f{H}^s(\Omega)}$$
with $s>0$, where $\f{H}^s(\Omega)$ denotes the fractional Sobolev-Slobodeckij space (cf. \cite[p. 18]{roubicek}). 
\begin{lemma}\label{lemma:V}
    \normalfont Let $1/2<s\leq 1$. Then there exists a constant $\widetilde C>0$, independent of $h$ and $\f{u}$, such that 
    \begin{equation}\label{estimate lemma}
    \norm{\Pi_h \f{u} - \f{u}}_\f{V} \leq  \widetilde C h^{s-\frac{1}{2}}\norm{\f{u}}_{\f{H}^s(\curl)} \quad \forall \f{u} \in \f{H}^s(\curl) \quad  \forall h \in (0,1).
    \end{equation}
    Consequently, it holds for all $\f{u}\in \f{V}$ that $\Pi_h \f{u}\to \f{u}$ in $\f{V}$ as $h\to 0$.
\end{lemma}
\subsection{Fully discrete scheme} \label{section: existence}
Since  $\varepsilon,\mu\in W^{2,\infty}(0,T;L^\infty_\text{sym}(\Omega)^{3\times 3})$ according to  \ref{ass:A1},  we may apply the canonical product rule to obtain the $\f{E}$-$\f{H}$ formulation of  the   Maxwell system \eqref{eq:maxwell intro 1}-\eqref{eq:maxwell intro 3}  as follows: 
\begin{equation}\label{eq:maxwell intro full} \tag{P}
    \begin{cases}
      \varepsilon(x,t)\partial_t\f{E}(x,t)  - \curl\, \textbf{\textit{H}}(x,t) + \f{J}(x,\f{E}(x,t),\f{H}(x,t),t) = \textbf{\textit{f}}(x,t)  & \text{in } \Omega \times (0,T)\\
      \mu(x,t)\partial_t\f{H}(x,t) + \curl\, \textbf{\textit{E}}(x,t) +\partial_t \mu(x,t) \f{H}(x,t)= 0 & \text{in } \Omega \times (0,T)\\
      (\f{E},\f{H})(\cdot,0)=(\f{E}_0,\f{H}_0) & \text{in } \Omega\\
      \boldsymbol{\nu} \times \textbf{\textit{H}}(x,t) = \boldsymbol{\nu} \times  \textbf{\textit{b}}(x, \boldsymbol{\nu} \times \textbf{\textit{E}}(x,t))   &\text{on } \partial \Omega \times (0,T)
    \end{cases}       
\end{equation}
with 
\begin{equation} \label{def J}
\f{J}(x,\xi,\eta,t) :=\f{J}_{tot}(x,\xi,\eta,t) + \partial_t\varepsilon(x,t)\xi.
\end{equation}
Due to $ \partial_t \varepsilon\in W^{1,\infty}(0,T;L^\infty_\text{sym}(\Omega)^{3\times 3})$,  \hyperref[ass:A4]{(A4)} implies  that $\f{J}$   satisfies the same estimate as $\f{J}_{tot}$ with Lipschitz constant  $L \coloneqq L_{tot}+\norm{\partial_t\varepsilon}_{W^{1,\infty}(0,T;L^{\infty}(\Omega)^{3\times 3})}$.  

To derive a fully discrete scheme for \eqref{eq:maxwell intro full}, we   fix $h>0$ and set $\tau\coloneqq \frac{T}{N}$, where $N\in \mathbb{N}$ is arbitrary but fixed, and introduce the equidistant partition of the time interval $[0,T]$ as follows:
$$0=t_0<t_1<\dots < t_N = T\quad \text{with } t_n \coloneqq n\tau\quad \forall n\in \{0,\dots,N\}.$$
Now, multiplying the first equation in \eqref{eq:maxwell intro full} with $\f{v}_h\in \textbf{ND}_h$, integrating over $\Omega$ and formally applying \eqref{eq:partial} as well as the definition of $\f{V}$, we obtain  a fully discrete scheme  of \eqref{eq:maxwell intro full}  after employing the implicit Euler time discretization and  the mixed FEM based on $\textbf{ND}_h$ and $\textbf{DG}_h$  as follows:  
\begin{equation}\label{time-space-discrete}\tag{P$_{h,N}$}
\begin{dcases} \textrm{Find $(\f{E}^n_h ,\f{H}^n_h)_{n=1}^N \subset \textbf{ND}_h \times \textbf{DG}_h$ s.t.   for all  $n\in \{1,\dots, N\}$:}\\
\int_\Omega \varepsilon(t_n) \delta \f{E}^n_h\cdot \f{v}_h - \f{H}^n_h\cdot \curl\,\f{v}_h +  \f{J}(\cdot,\f{E}^n_h,\f{H}^n_h, t_n)\cdot \f{v}_h \,dx\\
+ \int_{\partial \Omega} \f{b}(\cdot,\boldsymbol{\nu}\times \f{E}^n_h)\cdot \boldsymbol{\nu}\times \f{v}_h\,dS= \int_\Omega \f{f}^n_h\cdot \f{v}_h\,dx\quad \forall \f{v}_h\in \textbf{ND}_h \\
\int_\Omega (\mu(t_n) \delta \f{H}^n_h + \curl\,\f{E}^n_h + \partial_t \mu(t_n) \f{H}_h^n)\cdot \f{w}_h\,dx= 0\quad \forall \f{w}_h\in \textbf{DG}_h  \\
\f{E}_h^0 \coloneqq \Pi_h \f{E}_0\in \textbf{ND}_h,\quad \f{H}_h^0\coloneqq \f{Q}_h \f{H}_0 \in \textbf{DG}_h
\end{dcases}
\end{equation}
with  
$$ \f{f}_h^n \coloneqq \f{Q}_h\f{f}(t_n), \quad \delta \f{E}^n_h \coloneqq \frac{\f{E}^n_h - \f{E}^{n-1}_h}{\tau},\quad \delta\f{H}^n_h \coloneqq \frac{\f{H}^n_h - \f{H}^{n-1}_h}{\tau}\quad \forall n\in \{1,\dots,N\}.$$
Due to  $\f{H}_h^n,\delta \f{H}_h^n,\curl\,\f{E}_h^n\in \textbf{DG}_h$ and the fact that every function in $\textbf{DG}_h$ is constant on each mesh element, the second variational equation in \eqref{time-space-discrete} is equivalent to 
\begin{equation}\label{eq:discrete-equation}
    \mu_h^n \delta\f{H}_h^n + \curl\,\f{E}_h^n + \mu_{t,h}^n \f{H}_h^n = 0
\end{equation}
with
\begin{equation} \label{eq:muhnt}
\mu_h^n\coloneqq \f{Q}_h \mu(t_n) =   \sum_{T\in \mathcal{T}_h} \chi_T \frac{1}{\lvert T\rvert}\int_T  \mu(t_n)  \,dx\quad \textrm{and} \quad \mu_{t,h}^n \coloneqq \f{Q}_h \partial_t \mu(t_n)   =   \sum_{T\in \mathcal{T}_h} \chi_T \frac{1}{\lvert T\rvert}\int_T \partial_t \mu(t_n)  \,dx.
\end{equation}
Thanks to \ref{ass:A1},  $\mu_h^n$ and $\mu_{t,h}^n$ are symmetric, and it holds    for a.e. $x \in  \Omega$ and all $n \in \{1,\ldots,N\}$ that 
     \begin{equation} \label{ineq:munt}
\underline \mu \lvert \xi\rvert^2\leq \xi^T \mu_h^n(x)  \xi\leq \overline{\mu}\lvert \xi\rvert^2 \quad \forall \xi \in \mathbb R^3. 
    \end{equation} 
The next two propositions establish  well-posedness results with and without relying on a  CFL condition.
\begin{proposition}\label{theorem:existence} 
     \normalfont Let \ref{ass:A1}-\ref{ass:A5} hold. Then, for every $h>0$ and $N\in\mathbb{N}$ satisfying 
    \begin{equation}\label{eq:M(h)}
    N> M(h)\coloneqq T\min\left(\frac{\underline{\mu}h}{C_{\text{inv}}+h\norm{\partial_t \mu}_{L^\infty(0,T;L^\infty(\Omega)^{3\times 3})}},\frac{\underline{\varepsilon}}{2\left(L_{\mathrm{tot}}+\|\partial_t\varepsilon\|_{L^\infty(0,T;L^\infty(\Omega)^{3\times3})}\right)}\right)^{-1}, 
    \end{equation}
 the discrete scheme \eqref{time-space-discrete} admits a unique solution $\{(\f{E}^n_h,\f{H}^n_h)\}_{n=1}^N\subset \textbf{ND}_h \times \textbf{DG}_h$. 
\end{proposition}
\begin{proof}
 Let $h>0$ and  $N \in \mathbb N$ satisfy \eqref{eq:M(h)}  and set $\tau= T/N$. We argue inductively and assume that, for some fixed $n\in \{1,\dots,N\}$, the solution $(\f{E}^{n-1}_h,\f{H}^{n-1}_h) \in \textbf{ND}_h \times \textbf{DG}_h$ for the previous time step is already given. Based on \eqref{eq:M(h)}, we have
 $$\tau\left(\frac{C_\text{inv}}{h} + \norm{\partial_t \mu}_{L^\infty(0,T;L^\infty(\Omega)^{3\times 3})}\right) \leq \underline{\mu}.$$
 Setting $G_n\coloneqq \mu_h^n + \tau \mu_{t,h}^n$ and using \eqref{eq:muhnt}-\eqref{ineq:munt},  we see that for a.e. $x\in \Omega$,
 \begin{equation} \label{eq:G_n}
    \frac{\tau C_\text{inv}}{h}\lvert \xi\rvert^2 <  (\underline{\mu}-\tau \norm{\partial_t \mu}_{L^\infty(0,T;L^\infty(\Omega)^{3\times 3})})\lvert \xi\rvert^2 \le \xi^T G_n(x) \xi   \le (\overline{\mu}+\tau \norm{\partial_t \mu}_{L^\infty(0,T;L^\infty(\Omega)^{3\times 3})})\lvert \xi\rvert^2  \quad \forall \xi \in \mathbb R^3.
 \end{equation}
In particular, $G_n$ is  uniformly positive definite. In view of \eqref{eq:discrete-equation}, we can therefore substitute  
 \begin{equation}\label{eq:Hformula}
 \f{H}^n_h\coloneqq G_n^{-1}(\mu_h^n\f{H}^{n-1}_h - \tau \curl\,\f{E}^n_h)
 \end{equation}
 into the variational equality in (P$_{N,h}$) and multiply  with $\tau$ to  infer  that a solution $\f{E}^n_h \in \textbf{ND}_h$ of \eqref{time-space-discrete} can be determined by solving the fixed point equation $\Phi(\f{E}^n_h) = \f{E}^n_h$, where  $\Phi : \textbf{ND}_h \to \textbf{ND}_h$ is defined as the solution operator of the  following variational problem: For every $\f{q}_h \in \textbf{ND}_h$, $\f{e}_h:= \Phi(\f{q}_h) \in  \textbf{ND}_h$ denotes the unique solution to
        \begin{equation}\label{eq:fixx}
        \begin{split}
    & \int_\Omega \varepsilon(t_n) \f{e}_h \cdot \f{v}_h +\tau^2 G_n^{-1} \curl\,\f{e}_h \cdot \curl\,\f{v}_h\,dx+ \tau \int_{\partial \Omega} \f{b}(\cdot,\boldsymbol{\nu}\times \f{e}_h )\cdot \boldsymbol{\nu}\times \f{v}_h\,dS=\\
    & \int_\Omega  \Bigl(\varepsilon(t_n)\f{E}^{n-1}_h +\tau \f{f}^n_h - \tau  \f{J}(\cdot,\f{q}_h,G_n^{-1}(\mu_h^n \f{H}_h^{n-1} - \tau \curl\,\f{q}_h),t_n)\Bigr)\cdot \f{v}_h\,dx\\
    & +\int_\Omega  \tau G_n^{-1}\mu_h^n\f{H}^{n-1}_h \cdot \curl\,\f{v}_h \,dx \quad \forall \f{v}_h\in \textbf{ND}_h.
\end{split}
\end{equation}
Due to  \ref{ass:A1}, \ref{ass:A5}, and the positive definitneness of $G_n^{-1}$, the left-hand side in \eqref{eq:fixx} defines a strictly monotone, hemicontinuous and coercive operator $T:\textbf{ND}_h \to \textbf{ND}_h^*$ (regarding the coercivity, we refer to the proof of  \cite[Lemma 3.1]{nicaise1}). Thus, since the right-hand side in \eqref{eq:fixx}  represents an element of $\textbf{ND}_h^*$,  the Minty--Browder theorem implies that \eqref{eq:fixx} admits a unique solution $\f{e}_h$. In conclusion, the solution operator $\Phi : \textbf{ND}_h \to \textbf{ND}_h,\f{q}_h\mapsto \f{e}_h$, is well-defined. Now, let  $\f{q}_h,\hat{\f{q}}_h\in \textbf{ND}_h$ and set $\f{e}_h\coloneqq \Phi(\f{q}_h),\hat{\f{e}}_h\coloneqq \Phi(\hat{\f{q}}_h)$. Testing \eqref{eq:fixx} with $\f{v}_h=\f{e}_h -\hat{\f{e}}_h$ and testing the associated variational formulation of $\hat{\f{e}}_h$ with $\f{v}_h=\hat{\f{e}}_h - \f{e}_h$,  we obtain after adding the resulting equalities and  employing the monotonicity of $\f{b}$ that 
      \begin{align*}
          & \underline{\varepsilon}\norm{\f{e}_h-\hat{\f{e}}_h }_{\f{L}^2(\Omega)}^2 + \tau^2\norm{\curl(\f{e}_h-\hat{\f{e}}_h)}_{\f{L}_{G_n^{-1}}^2(\Omega)}^2\\
          & \leq \tau \norm{\f{J}(\cdot,\f{q}_h,G_n^{-1}(\mu_h^n \f{H}_h^{n-1} - \tau \curl\,\f{q}_h),t_n)-\f{J}(\cdot,\hat{\f{q}}_h,G_n^{-1}(\mu_h^n \f{H}_h^{n-1} - \tau \curl\,\hat{\f{q}}_h),t_n)}_{\f{L}^2(\Omega)}\norm{\f{e}_h-\hat{\f{e}}_h}_{\f{L}^2(\Omega)}.
      \end{align*}
    Consequently, it holds that
          \begin{align*}
          \norm{\f{e}_h-\hat{\f{e}}_h }_{\f{L}^2(\Omega)}   
  &\leq \frac{\tau}{ \underline{\varepsilon}} \norm{\f{J}(\cdot,\f{q}_h,G_n^{-1}(\mu_h^n\f{H}^{n-1}_h - \tau \curl\,\f{q}_h),t_n)-\f{J}(\cdot,\hat{\f{q}}_h,G_n^{-1}(\mu_h^n\f{H}^{n-1}_h - \tau \curl\,\hat{\f{q}}_h),t_n)}_{\f{L}^2(\Omega)}  \\
&\underbrace{\leq}_{\eqref{def J},\hyperref[ass:A4]{(A4)}}   \frac{(L_{tot}+\norm{\partial_t\varepsilon}_{L^\infty(0,T;L^\infty(\Omega)^{3\times 3})})\tau}{\underline{\varepsilon}} (\norm{\f{q}_h-\hat{\f{q}}_h}_{\f{L}^2(\Omega)}+\tau\norm{G_n^{-1}\curl(\f{q}_h-\hat{\f{q}}_h)}_{\f{L}^2(\Omega)})\\
& \underbrace{\leq}_{\eqref{eq:inverse},\eqref{eq:G_n}} \frac{2(L_{tot}+\norm{\partial_t\varepsilon}_{L^\infty(0,T;L^\infty(\Omega)^{3\times 3})})\tau}{\underline{\varepsilon}}\norm{\f{q}_h - \hat{\f{q}}_h}_{\f{L}^2(\Omega)}.
      \end{align*}
Since  $\textbf{ND}_h$ is finite dimensional, $\{\textbf{ND}_h, \|\cdot\|_{\f{L}^2(\Omega)} \}$ is a Banach space. Therefore, 
by \eqref{eq:M(h)}, $\Phi$ is a contraction, and Banach's fixed point theorem   implies that $\Phi$  admits a unique fixed point $\f{E}^n_h \in \textbf{ND}_h$.
\end{proof}
\begin{remark}\label{rem:wellposedness0}
 The inequality \eqref{eq:M(h)} is equivalent to the CFL-type condition
    \begin{equation}\label{eq:CFL}
        \tau < \min\left(\frac{\underline{\mu}h}{C_{\text{inv}} + h\norm{\partial_t \mu}_{L^\infty(0,T;L^\infty(\Omega)^{3\times 3})}},\frac{\underline{\varepsilon}}{2(L_{tot}+\norm{\partial_t\varepsilon}_{L^\infty(0,T;L^\infty(\Omega)^{3\times 3})})}\right)
    \end{equation}
    and arises due to the application of the discrete inverse curl estimate \eqref{eq:inverse} in our proof and the time-dependence of the magnetic permeability. We also note that this estimate is not optimal.
    \end{remark}
\begin{proposition} \label{theorem:existence2}
    \normalfont Let \ref{ass:A1}-\ref{ass:A5} hold and let $\partial_t \mu$ be uniformly positive semidefinite. Assume that there exists a function $\hat{c}:\Omega \times [0,T]\to \mathbb{R}$ with $\hat{c}(\cdot,t)\in L^1(\Omega)$ for all $t\in [0,T]$ such that
\begin{equation}\label{eq:J-weak-monoton}
    \f{J}(x,\xi,\eta,t)\cdot \xi\geq  - c_1\lvert \xi\rvert^p - c_2\lvert \eta\rvert^p -\hat{c}(x,t) \quad \text{for a.e. } x\in \Omega \text{ and all } \xi,\eta\in \mathbb{R}^3,t\in [0,T],
\end{equation}
with $p\in \{1,2\}$ and  positive constants $c_1,c_2 >0$, independent of $x,\xi,\eta,t$, satisfying
\begin{equation}\label{eq:con p}
\begin{cases}  0 < c_1 <   \underline \varepsilon/ T   &\textrm{ if $p=2$}\\ 0 < c_1 <\infty   &\textrm{ if $p=1$}
\end{cases}, \quad \begin{cases}
 0 < c_2 < \underline \mu/ 2T   &\textrm{ if $p=2$} \\
  0< c_2 < \infty  &\textrm{ if $p=1$.}\end{cases}
 \end{equation}
  Then, for every $h>0$ and $N\in \mathbb{N}$, the discrete scheme \eqref{time-space-discrete} admits a solution $\{(\f{E}^n_h,\f{H}^n_h)\}_{n=1}^N\subset \textbf{ND}_h \times \textbf{DG}_h$. Furthermore, if there exist constants $c_E,c_H > 0$ with $Tc_E < \underline{\varepsilon},Tc_H< \underline{\mu}$ such that
\begin{equation}\label{eq:J-unique}
    (\f{J}(x,\xi,\eta,t) - \f{J}(x,\xi',\eta',t))\cdot (\xi - \xi')\geq -c_E\lvert \xi - \xi'\rvert^2 - c_H\lvert \eta - \eta'\rvert^2
\end{equation}
holds for a.e. $x\in \Omega$ and all $\xi,\xi',\eta,\eta'\in \mathbb{R}^3,t\in [0,T]$, then \eqref{time-space-discrete} admits at most one solution.
\end{proposition}
\begin{proof} Let $h>0$ and $N \in \mathbb N$ be arbitrarily fixed and  set $\tau= T/N$. As before,  we argue inductively and assume that a solution $(\f{E}^{n-1}_h,\f{H}^{n-1}_h) \in \textbf{ND}_h \times \textbf{DG}_h$ for the previous time step is already given for some fixed $n\in \{1,\dots,N\}$.
    As  $\partial_t \mu$ is uniformly postive semidefinite, we obtain from \eqref{eq:muhnt}-\eqref{ineq:munt} that $G_n\coloneqq \mu_h^n + \tau \mu_{t,h}^n$ is  uniformly positive definite with 
    \begin{equation}\label{eq:G_n2}
       \underline{\mu}\lvert\xi\rvert^2 \le   \xi^T G_n \xi   \quad \forall \xi\in \mathbb{R}^3.
    \end{equation}
 We   now introduce the operator $\f{F}:\textbf{ND}_h \to \textbf{ND}_h^*$  given by 
        \begin{align}\label{eq:F}
  & \forall \f{e}_h, \f{v}_h \in  \textbf{ND}_h: \quad  \langle  \f{F}(\f{e}_h),  \f{v}_h\rangle_{\textbf{ND}_h^*,\textbf{ND}_h}:=  \\ \notag
   &  \int_\Omega \varepsilon(t_n) \f{e}_h \cdot \f{v}_h +\tau^2 G_n^{-1} \curl\,\f{e}_h \cdot \curl\,\f{v}_h\,dx + \tau \int_{\partial \Omega} \f{b}(\cdot,\boldsymbol{\nu}\times \f{e}_h )\cdot \boldsymbol{\nu}\times \f{v}_h\,dS  \\ \notag
    & -   \int_\Omega \left(\varepsilon(t_n)\f{E}^{n-1}_h 
    +\tau \f{f}^n_h - \tau  \f{J}(\cdot,\f{e}_h,G_n^{-1}(\mu_h^n \f{H}_h^{n-1} - \tau \curl\, \f{e}_h),t_n)\right)\cdot \f{v}_h +  \tau G_n^{-1}\mu_h^n \f{H}^{n-1}_h \cdot \curl\,\f{v}_h \,dx.       \end{align}
By the construction of $\f{F}:\textbf{ND}_h \to \textbf{ND}_h^*$,  we know from the proof of Proposition \ref{theorem:existence} that $\f{E}^n_h \in \textbf{ND}_h$ is a solution to  (P$_{N,h}$) iff
$$
 \Phi(\f{E}^n_h)= \f{E}^n_h \quad  \iff  \quad \f{F}(\f{E}^n_h)=0.
$$
 Thus, our aim   is to find  a zero of  $\f{F}:\textbf{ND}_h \to \textbf{ND}_h^*$.   Due to \eqref{eq:J-weak-monoton},  it holds for every $\f{e}_h \in \textbf{ND}_h$ that 
    \begin{align}
   &\notag \tau \int_\Omega \f{J}(\cdot,\f{e}_h,G_n^{-1}(\mu_h^n \f{H}_h^{n-1} - \tau\curl\,\f{e}_h),t_n)\cdot \f{e}_h\,dx \\   \geq \notag
   &-\tau c_1\norm{\f{e}_h}_{\f{L}^p(\Omega)}^p 
   -\tau c_2\norm{G_n^{-1}\mu_h^n\f{H}_h^{n-1} + \tau^2 G_n^{-1}\curl\,\f{e}_h}_{\f{L}^p(\Omega)}^p -\tau \norm{\hat{c}(\cdot,t_n)}_{\f{L}^1(\Omega)}  =: I_p.
\end{align}
If $p=2$, by using \eqref{eq:uniformposdef} and  \eqref{eq:G_n2}, we obtain
$$
\begin{aligned}
 & I_2      \ \  \geq \  \ \, - \frac{\tau c_1}{\underline \varepsilon} \norm{\f{e}_h}_{\f{L}^2_{\varepsilon(t_n)}(\Omega)}^2 - \frac{2\tau^3c_2}{\underline{\mu}}(G_n^{-1}\curl\,\f{e}_h,\curl\,\f{e}_h)_{\f{L}^2(\Omega)}-\beta_2 \\ &\underbrace{\ge}_{\tau = T/N, \eqref{eq:con p}}        - \frac{\vartheta}{N} \norm{\f{e}_h}_{\f{L}^2_{\varepsilon(t_n)}(\Omega)}^2 - \frac{\tau^2\vartheta}{N}(G_n^{-1}\curl\,\f{e}_h,\curl\,\f{e}_h)_{\f{L}^2(\Omega)}-\beta_2 \quad \forall  \f{e}_h \in \textbf{ND}_h
\end{aligned}
$$
with   $\beta_2 >0$ and $\vartheta \in (0,1)$, independent of $\f{e}_h$. If $p=1$, a direct application of Young's inequality yields  
$$
I_1      \geq - \frac{\vartheta}{N} \norm{\f{e}_h}_{\f{L}^2_{\varepsilon(t_n)}(\Omega)}^2 - \frac{ \tau^2\vartheta}{N}(G_n^{-1}\curl\,\f{e}_h,\curl\,\f{e}_h)_{\f{L}^2(\Omega)}-\beta_1  \quad \forall  \f{e}_h \in \textbf{ND}_h
$$
with a positive constant $\beta_1 >0$  independent of $\f{e}_h$. Inserting $\f v_h = \f e_h$ in  \eqref{eq:F} and invoking  the previous estimates and the monotonicity of $\f{b}$, we arrive at
\begin{equation}\label{eq:coerz}
 \langle \f{F}(\f{e}_h),\f{e}_h\rangle_{\textbf{ND}_h^*,\textbf{ND}_h}
  \geq  (1- \vartheta {N}^{-1})(\norm{\f{e}_h}_{\f{L}^2_{\varepsilon(t_n)} (\Omega)}^2 + \tau^2(G_n^{-1}\curl\,\f{e}_h,\curl\,\f{e}_h)_{\f{L}^2(\Omega)}  ) - \max\{\beta_1,\beta_2\}.
\end{equation} Now, let us consider $\textbf{ND}_h$ endowed with the inner product $(\cdot,\cdot)_{\f{L}^2_{\varepsilon(t_n)}(\Omega)}$ and let $\boldsymbol{\mathcal{R}}:\textbf{ND}_h \to \textbf{ND}_h^*$ denote the corresponding Riesz isomorphism. Setting  ${\boldsymbol{\mathcal F}}\coloneqq \boldsymbol{\mathcal{R}}^{-1}\circ \f{F}$,   \eqref{eq:coerz} implies  the existence of a positive constant $R>0$ such that
$$(\boldsymbol{\mathcal{F}}(\f{e}_h),\f{e}_h)_{\textbf{ND}_h} =  \langle \f{F}(\f{e}_h),\f{e}_h\rangle_{\textbf{ND}_h^*,\textbf{ND}_h}> 0\quad \text{for all } \f{e}_h\in \textbf{ND}_h \text{ with } \norm{\f{e}_h}_{\f{L}^2_{\varepsilon(t_n)}(\Omega)}=R.$$
As a consequence, \cite[Lemma 18.2]{Vainberg} implies the existence of $\f{E}^n_h\in \textbf{ND}_h$ with ${\f{F}}(\f{E}^n_h)=0$. 

Next, assume that \eqref{eq:J-unique} holds. Let $(\f{E}_h^n,\f{H}_h^n),(\hat{\f{E}}_h^n,\hat{\f{H}}_h^n)$ denote two solutions of the $n$-th step in \eqref{time-space-discrete}. By subtracting the respective systems from each other, testing suitably and using the positive semidefiniteness of $\partial_t \mu$, the monotonicity of $\f{b}$ and \eqref{eq:J-unique}, we arrive at
$$0\geq \left(\frac{\underline{\varepsilon}}{\tau}-c_E\right)\norm{\f{E}_h^n - \hat{\f{E}}_h^n}_{\f{L}^2(\Omega)}^2 + \left(\frac{\underline{\mu}}{\tau}-c_H\right) \norm{\f{H}_h^n - \hat{\f{H}}_h^n}_{\f{L}^2(\Omega)}^2.$$
By our assumption, the terms in the brackets are positive, which completes our proof.
\end{proof}
\begin{remark} \label{rem:notmonotone}  Let us consider the cut-off function $\theta: \mathbb R \to [-1,1], \theta(x) := \max\{-1,\min\{x,1\} \}$ and
$$
\f{J}(x,\xi,\eta,t)  = \f{J}(\xi,\eta)=  \alpha_1\xi + \alpha_2(\theta(\eta_1),\theta(\eta_2),\theta(\eta_3))  \quad  \textrm{for a.e. } x\in \Omega \text{ and all } \xi,\eta\in \mathbb{R}^3,t\in [0,T]
$$
with positive constants $\alpha_1,\alpha_2>0$. 
For every fixed $\xi,\eta \in \mathbb R^3$,  the mappings $\f{J}(\cdot,\eta): \mathbb R^3 \to \mathbb R^3$ and  $\f{J}(\xi,\cdot): \mathbb R^3 \to \mathbb R^3$ are monotone, but 
$$
(\f{J}(\xi,\eta) -\f{J}(\xi',\eta'),\xi -\xi')_{\mathbb R^3} \ge 0  \quad \forall \xi,\xi',\eta,\eta' \in \mathbb R^3
$$
is not valid. Thus, although $\f{J}$ is monotone with respect to $\xi$ and $\eta$, the induced  nonlinearity $$
 \boldsymbol{\mathcal J}:  \f L^2(\Omega) \times \f L^2(\Omega) \to \f L^2(\Omega) \times \f L^2(\Omega), \quad  \boldsymbol{\mathcal J}(\f u, \f v) := (\f{J}(\f u, \f v),0),  
$$
that appears in the Maxwell system \eqref{eq:maxwell intro full} and the fully discrete scheme \eqref{time-space-discrete}, is not monotone. This simple case shows that the  Minty--Browder argument in general cannot be applied in our analysis. However, straightforward computations yield  
$
\f{J}(\xi,\eta) \cdot \xi =  \alpha_1 |\xi|^2  - \alpha_2\sqrt{3}|\xi| \ge  - \alpha_2\sqrt{3}|\xi|
$
for all  $\xi,\eta \in \mathbb R^3$.
As a conclusion, Proposition \ref{theorem:existence2} is applicable for this example.
\end{remark}
  
\subsection{Stability}\label{section: stability}
We begin by introducing the quantity 
 \begin{equation}\label{eq:delh0}
  \delta \f{H}_h^0 \coloneqq -(\mu_h^0)^{-1} (\mu_{t,h}^0 \f{H}_h^0 + \curl\,\f{E}_h^0)\in \textbf{DG}_h.
  \end{equation}
Furthermore, we define $\delta \f{E}_h^0\in \textbf{ND}_h$ as the unique solution of the variational problem
    \begin{align}\notag
    & \int_\Omega \varepsilon(0)\delta \f{E}^0_h \cdot \f{v}_h -\f{H}_h^0 \cdot \curl\,\f{v}_h +  \f{J}(\cdot,\f{E}^0_h,\f{H}_h^0,0)\cdot \f{v}_h\,dx\\ \label{eq:var-initial}
    & +\int_{\partial \Omega} \f{b}(\cdot,\boldsymbol{\nu}\times \f{E}^0_h)\cdot \boldsymbol{\nu}\times \f{v}_h\,dS = \int_\Omega \f{f}^0_h\cdot \f{v}_h\,dx\quad \forall \f{v}_h\in \textbf{ND}_h,
\end{align}
for which the existence and uniqueness of a solution follow directly from the Lax-Milgram Lemma as well as the norm equivalence of $\norm{\cdot}_\f{V}$ and $\norm{\cdot}_{\f{L}^2(\Omega)}$ on $\textbf{ND}_h$. By the specific construction of $ (\delta \f{E}_h^0, \delta \f{H}_h^0)$, we observe that \eqref{time-space-discrete} is satisfied at the initial time step $n = 0$. This property along with  the boundedness of  $\{\delta \f{E}_h^0\}_{h>0}$ in $\f{L}^2(\Omega)$ (see below) is  crucial  for our stability result. The proof relies also on the  choice  $(\f{E}^0_h,\f{H}^0_h)=(\Pi_h \f{E}_0,\f{Q}_h \f{H}_0)$,  which  together with  Lemma \ref{lemma:V} and \eqref{eq:qhest} yields that
\begin{equation}\label{eq:initialvalues}
    (\f{E}_h^0,\f{H}_h^0)\quad \to \quad (\f{E}_0,\f{H}_0)\quad \text{in } \f{V}\times \f{L}^2(\Omega)\text{ as } h\to 0.
\end{equation}

 \begin{lemma}\label{lemma:initial-bounded} \normalfont Let Assumption \ref{ass} hold. Then 
      $\{\delta\f{E}_h^0\}_{h>0}\subset \f{L}^2(\Omega)$ is bounded.
\end{lemma}
\begin{proof} 
    Setting $\f{v}_h = \delta \f{E}_h^0$ in \eqref{eq:var-initial} yields
    \begin{align}\notag
        \underline{\varepsilon}\norm{\delta \f{E}_h^0}_{\f{L}^2(\Omega)}^2 & \leq \norm{\f{f}_h^0 - \f{J}(\cdot,\f{E}^0_h,\f{H}_h^0,0)}_{\f{L}^2(\Omega)}\norm{\delta \f{E}_h^0}_{\f{L}^2(\Omega)}\\ \label{keylemma1}
        & + (\f{H}_h^0,\curl\,\delta\f{E}_h^0)_{\f{L}^2(\Omega)} - (\f{b}(\cdot, \boldsymbol{\nu}\times \f{E}_h^0),\boldsymbol{\nu}\times \delta \f{E}_h^0)_{\f{L}^2(\partial \Omega)} \quad \forall h >0. 
    \end{align}
    As $\curl\,\textbf{ND}_h\subset \textbf{DG}_h$, \eqref{eq:qhest}  gives
    \begin{align}
         (\f{H}_h^0,\curl\, \delta \f{E}_h^0)_{\f{L}^2(\Omega)} \! =\!(\f{H}_0,\curl\, \delta \f{E}_h^0)_{\f{L}^2(\Omega)}  \label{keylemma2} \!= \!(\curl\,\f{H}_0,\delta \f{E}_h^0)_{\f{L}^2(\Omega)}\! + \!(\f{b}(\cdot, \boldsymbol{\nu}\times \f{E}_0),\boldsymbol{\nu}\times \delta \f{E}_h^0)_{\f{L}^2(\partial \Omega)},
    \end{align}
    where we have also used \eqref{eq:partial} together with \hyperref[ass:A3]{(A3)}. By \hyperref[ass:A6]{(A6)}, there exist constants $L_0,C_0 > 0$   such that
    $$\norm{  \f{b}(\cdot, \boldsymbol{\nu}\times \f{E}_0) - \f{b}(\cdot,\f{u})}_{\f{L}^2(\partial \Omega)} \leq L_0\norm{ \boldsymbol{\nu}\times \f{E}_0- \f{u} }_{\f{L}^2(\partial \Omega)}$$
  holds  for all $\f{u}\in \f{L}^2(\partial \Omega)$ satisfying $\norm{ \boldsymbol{\nu}\times \f{E}_0 - \f{u} }_{\f{L}^2(\partial \Omega)}<C_0.$ Consequently, in view of  \eqref{eq:initialvalues}, we find a constant  $h_0 > 0$  such that
    \begin{align}\notag
        (\f{b}(\cdot, \boldsymbol{\nu}\times \f{E}_0) - \f{b}(\cdot, \boldsymbol{\nu}\times \f{E}_h^0),\boldsymbol{\nu}\times \delta \f{E}_h^0)_{\f{L}^2(\partial \Omega)} \notag
        \leq L_0\norm{ \boldsymbol{\nu}\times \f{E}_0-  \boldsymbol{\nu}\times \f{E}^0_h }_{\f{L}^2(\partial \Omega)} \norm{\boldsymbol{\nu}\times \delta \f{E}_h^0}_{\f{L}^2(\partial \Omega)} \\
       \leq L_0\norm{ \f{E}_0 - \f{E}_h^0 }_{\f{V}} \norm{  \delta \f{E}_h^0}_{\f{L}^2(\partial \Omega)}
        \label{keylemma3}\leq L_0\widetilde Ch^{\frac{1}{2}}\norm{\f{E}_0}_{\f{H}^1(\curl)}\norm{\delta \f{E}_h^0}_{\f{L}^2(\partial \Omega)} \quad \forall  0<h\leq h_0,
    \end{align}
    where we have used Lemma \ref{lemma:V} in combination with \hyperref[ass:A6]{(A6)}.   Applying \eqref{keylemma2} and \eqref{keylemma3} to \eqref{keylemma1} yields
      \begin{align}\notag
        \underline{\varepsilon}\norm{\delta \f{E}_h^0}_{\f{L}^2(\Omega)}^2  \leq &\norm{\f{f}_h^0 -\f{J}(\cdot,\f{E}^0_h,\f{H}_h^0,0)}_{\f{L}^2(\Omega)}\norm{\delta \f{E}_h^0}_{\f{L}^2(\Omega)} +  \norm{\curl\,\f{H}_0}_{\f{L}^2(\Omega)}\norm{\delta \f{E}_h^0}_{\f{L}^2(\Omega)}\\ \label{keylemma4}
        &  +L_0 \widetilde Ch^{\frac{1}{2}}\norm{\f{E}_0}_{\f{H}^1(\curl)}\norm{\delta \f{E}_h^0}_{\f{L}^2(\partial \Omega)} \quad \forall 0<h\leq h_0.
    \end{align} 
    In view of \eqref{eq:qhest2}, \eqref{eq:Pi_h proj},  and \hyperref[ass:A4]{(A4)}, we have
\begin{equation}\label{keylemma5}
\begin{split}
 \widehat C & := \sup_{h>0}\,   \norm{\f{f}_h^0 -\f{J}(\cdot,\f{E}^0_h,\f{H}_h^0,0)}_{\f{L}^2(\Omega)} +  \norm{\curl\,\f{H}_0}_{\f{L}^2(\Omega)}\\
 & = \sup_{h>0}\,   \norm{ \f{Q}_h\f{f}(0) - \f{J}(\cdot,\Pi_h\f{E}_0, \f{Q}_h\f{H}_0,0)}_{\f{L}^2(\Omega)} + \norm{\curl\,\f{H}_0}_{\f{L}^2(\Omega)} < \infty.
 \end{split}
 \end{equation}
    Furthermore, by the quasi-uniformity of $\{\mathcal{T}_h\}_{h>0}$,  the   trace inverse estimate (cf. \cite[Lemma 12.1]{Guermond})
    \begin{align} \label{eq:inverse-trace}
        \norm{\f{v}_h}_{\f{L}^2(\partial \Omega)}  \le \frac{C^\Gamma_\text{inv}}{\sqrt{h}} \norm{\f{v}_h}_{\f{L}^2(\Omega)} \quad  \forall \f{v}_h\in \textbf{ND}_h
    \end{align}
    holds with a constant $C^\Gamma_\text{inv}$, independent of $h$. Note that, similar to \eqref{eq:inverse}, the constant  $C^\Gamma_\text{inv}$  may depend on the quasi-uniformity constants of the triangulation. Now, applying   \eqref{keylemma5} and  \eqref{eq:inverse-trace} to \eqref{keylemma4} results in  
\begin{align*}
    \norm{\delta \f{E}_h^0}_{\f{L}^2(\Omega)} & \leq \underline{\varepsilon}^{-1}(\widehat C +  L_0 \widetilde CC^\Gamma_\text{inv} \norm{\f{E}_0}_{\f{H}^1(\curl)})\quad \forall 0<h\leq h_0.
\end{align*}
Finally, to show the boundedness for $h>h_0$, we apply  \eqref{eq:propertyb} and \eqref{eq:Pi_h proj} to obtain
$$\norm{\f{b}(\cdot,\boldsymbol{\nu}\times \f{E}_h^0)}_{\f{L}^2(\partial \Omega)}\leq \overline{b}(\norm{\boldsymbol{\nu}\times \f{E}_h^0}_{\f{L}^2(\partial \Omega)}  +  \lvert \partial \Omega\rvert^{\frac{1}{2}}) \leq \overline{b}(\norm{\f{E}_0}_\f{V}+  \lvert \partial \Omega\rvert^{\frac{1}{2}})  =:C_b \quad \forall h >0.$$
Hence, using \eqref{keylemma1} and \eqref{keylemma5}, we arrive at
\begin{align*}
    \underline{\varepsilon}\norm{\delta \f{E}_h^0}_{\f{L}^2(\Omega)}^2 & \leq \widehat C \norm{\delta \f{E}_h^0}_{\f{L}^2(\Omega)}
 + \norm{\f{H}_h^0}_{\f{L}^2(\Omega)}\norm{\curl\,\delta \f{E}_h^0}_{\f{L}^2(\Omega)} + \norm{\f{b}(\cdot,\boldsymbol{\nu}\times \f{E}_h^0)}_{\f{L}^2(\partial \Omega)}\norm{\boldsymbol{\nu}\times \delta \f{E}_h^0}_{\f{L}^2(\partial \Omega)}\\
 & \leq \Bigl(\widehat{C}+\frac{C_\text{inv}}{h_0}\norm{\f{H}_0}_{\f{L}^2(\Omega)} + \frac{C_\text{inv}^\Gamma}{\sqrt{h_0}}C_b\Bigr)\norm{\delta \f{E}_h^0}_{\f{L}^2(\Omega)} \quad \forall h > h_0,
    \end{align*}
    where we have used \eqref{eq:inverse},\eqref{eq:qhest2} and \eqref{eq:inverse-trace} for the last inequality.
This completes the proof. 
\end{proof}
\begin{proposition}\label{theorem:stability} \normalfont Let Assumption \ref{ass} hold. Then,
there exist positive constants $C,\beta>0$, not depending on  $h$ and $N$, such that for all $h>0$ and $N> \max\{\beta,M(h)\}$ with $M(h)$ as in \eqref{eq:M(h)}, the solution {\normalfont $\{(\f{E}^n_h,\f{H}^n_h)\}_{n=1}^N$} to {\normalfont (P$_{N,h}$)} satisfies
\begin{align*}
{\normalfont \max_{1\leq m\leq N}\Bigl[ }&{\normalfont
\norm{\f{E}^m_h}_{\f{V}}^2
+\norm{\f{H}^m_h}_{\f{L}^2(\Omega)}^2
+\norm{\delta\f{E}^m_h}_{\f{L}^2(\Omega)}^2
+\norm{\delta\f{H}^m_h}_{\f{L}^2(\Omega)}^2
\Bigr]\leq C.}
\end{align*}
\end{proposition}
The proof of Proposition \ref{theorem:stability} is based on suitably testing the difference of the first two variational equations in \eqref{time-space-discrete} at consecutive time steps along with the application of Lemma \ref{lemma:initial-bounded} and \eqref{eq:delh0}-\eqref{eq:var-initial}. Since no substantially new ideas are required, we omit the proof here and refer to    Appendix \ref{appen c}.

\subsection{Convergence} \label{section: convergence}
In the following, let $N=N(h)$ denote a natural number depending on $h>0$ such that $N(h) \to \infty$ as $h \to 0$ and  $\max\{\beta,M(h)\} < N(h)$ for all $h>0$ with $M(h)$ and $\beta$ as in Proposition \ref{theorem:existence} and Proposition \ref{theorem:stability}, respectively.  Since $N$ is uniquely determined by $h$, we omit $N$ as an index in all relevant quantities. We now introduce piecewise linear and piecewise constant (in time) interpolants of our physical parameters. For example, for the electric field, the magnetic permeability, and the distributed nonlinearity, these read as 
\begin{equation}\label{eq:interpol} 
    \begin{aligned}
        & \f{E}_{h}:[0,T]\to \f{L}^2(\Omega),\quad && t\mapsto 
    \begin{cases}
     \f{E}^0_h, & t = 0\\
     \f{E}^{n-1}_h + (t-t_{n-1})\delta \f{E}^n_h, & t\in (t_{n-1},t_n]
     \end{cases}  \\  
     & \overline{\f{E}}_{h}:[0,T]\to \f{L}^2(\Omega),\quad &&t\mapsto 
    \begin{cases}
     \f{E}^0_h, & t = 0\\
     \f{E}^n_h, & t\in (t_{n-1},t_n]
     \end{cases}\\
     & \mu_{h}:[0,T]\to \f{L}^2(\Omega),\quad && t\mapsto 
    \begin{cases}
     \mu(0), & t = 0\\
     \mu(t_{n-1}) + (t-t_{n-1})\frac{\mu(t_n)-\mu(t_{n-1})}{\tau}, & t\in (t_{n-1},t_n]
     \end{cases}\\
     & \overline{\mu}_{h}:[0,T]\to \f{L}^2(\Omega),\quad &&t\mapsto 
    \begin{cases}
     \mu(0), & t = 0\\
     \mu(t_n), & t\in (t_{n-1},t_n]
     \end{cases}\\
       & \overline{\f{J}}_{h}:[0,T]\to \f{L}^2(\Omega),\quad &&t\mapsto 
    \begin{cases}
      \f{J}\left(\cdot,\f{E}^0_h, \f{H}_h^0, 0\right), & t = 0\\
      \f{J}\left(\cdot,\f{E}^n_h, \f{H}_h^n, t_n\right), & t\in (t_{n-1},t_n],
     \end{cases}
     \end{aligned}
\end{equation}
where quantities without (resp. with) an overline correspond to  piecewise linear (resp. piecewise constant) interpolations. By this construction and  Proposition \ref{theorem:stability}, it holds that 
$$(\f{E}_{h},\f{H}_{h})\in W^{1,\infty}(0,T;\f V \times \f{L}^2(\Omega))\quad \text{with}\quad  \partial_t(\f{E}_{h},\f{H}_{h})(t) = (\delta\f{E}^n_h,\delta\f{H}^n_h)\quad \forall t\in (t_{n-1},t_n]$$ 
and
\begin{equation}\label{eq:TC/N}
    \begin{split}
    \norm{\f{E}_{h}(t)-\overline{\f{E}}_{h}(t)}_{\f{L}^2(\Omega)}& \leq \tau \max_{1\leq n\leq N} \norm{\delta \f{E}^n_h}_{\f{L}^2(\Omega)}\leq \frac{TC}{N}\quad \forall t\in [0,T]\\
    \norm{\f{H}_{h}(t)-\overline{\f{H}}_{h}(t)}_{\f{L}^2(\Omega)}&\leq \tau \max_{1\leq n\leq N} \norm{\delta \f{H}^n_h}_{\f{L}^2(\Omega)}\leq \frac{TC}{N}\quad \forall t\in [0,T],
    \end{split}
\end{equation}
from which it follows that
    \begin{equation}\label{eq:E_N overline E_N} 
        \lim_{h\to 0} \left( \sup_{t\in [0,T]}\norm{\f{E}_{h}(t) - \overline{\f{E}}_{h}(t)}_{\f{L}^2(\Omega)} \right) = \lim_{h\to 0} \left(\sup_{t\in [0,T]}\norm{\f{H}_h(t) - \overline{\f{H}}_h(t)}_{\f{L}^2(\Omega)}\right) = 0.
    \end{equation}
Analogously,  in view of  \ref{ass:A1}, it holds that  $\varepsilon_h,\mu_h\in W^{1,\infty}(0,T;L^\infty(\Omega)^{3\times 3})$,  and for all $h>0$, we have
\begin{equation}\label{eq:coefficientest}
\begin{split}
\max\left\{
\norm{\overline{\varepsilon}_h - \varepsilon}_{L^\infty(0,T;L^\infty(\Omega)^{3\times 3})}
\norm{\overline{\varepsilon}_h - \varepsilon_h}_{L^\infty(0,T;L^\infty(\Omega)^{3\times 3})}
\right\}
&\leq \frac{T\norm{\partial_t\varepsilon}_{L^\infty(0,T;L^\infty(\Omega)^{3\times 3})}}{N(h)}\\
\max\left\{
\norm{\overline{\mu}_h - \mu}_{L^\infty(0,T;L^\infty(\Omega)^{3\times 3})},
\norm{\overline{\mu}_h - \mu_h}_{L^\infty(0,T;L^\infty(\Omega)^{3\times 3})}
\right\}
&\leq \frac{T\norm{\partial_t\mu}_{L^\infty(0,T;L^\infty(\Omega)^{3\times 3})}}{N(h)}\\
\norm{\overline{\mu}_{t,h} - \partial_t\mu}_{L^\infty(0,T;L^\infty(\Omega)^{3\times 3})}
&\leq \frac{T\norm{\partial_{tt}\mu}_{L^\infty(0,T;L^\infty(\Omega)^{3\times 3})}}{N(h)}.
\end{split}
\end{equation}
   Also, by \hyperref[ass:A2]{(A2)}, \eqref{eq:qhest} and \eqref{eq:qhest2},  
        \begin{equation}\label{eq:fconv}
        \overline{\f{f}}_h \quad \to\quad  \f{f}\quad \text{in } L^{2}(0,T;\f{L}^2(\Omega))\text{ as } h\to 0.
    \end{equation}
Now, utilizing \eqref{eq:E_N overline E_N}, Proposition \ref{theorem:stability}, and  Assumption \ref{ass}, standard arguments yield the existence of subsequences, which we denote  w.l.o.g. by the same symbol, such that as $h \to 0$,
\begin{equation}\label{eq:convergences}
    \begin{aligned}
          (\f{E}_{h},\f{H}_{h}) \quad & \rightharpoonup\quad &&(\f{E},\f{H}) &&\quad \text{weakly-}^*\text{ in }  L^\infty(0,T;\f{V} \times   \f L^2(\Omega))\\
          \partial_t(\f{E}_{h},\f{H}_{h}) \quad & \rightharpoonup\quad &&\partial_t (\f{E},\f{H}) &&\quad \text{weakly-}^*\text{ in } L^\infty(0,T;\f{L}^2(\Omega) \times  \f{L}^2(\Omega))\\
              (\overline{\f{E}}_{h},\overline{\f{H}}_{h})\quad & \rightharpoonup \quad && ( {\f{E}}, {\f{H}}) && \quad \text{weakly-}^*\text{ in }  L^\infty(0,T;\f{V} \times \f L^2(\Omega))\\
                     (\f{E}_h,\f{H}_h)(t) \quad &\rightharpoonup  \quad && (\f{E},\f{H})(t)  &&\quad \text{weakly}  \, \, \, \, \text{ in }  \f{L}^2(\Omega) \times \f L^2(\Omega)   \text{ for all } t\in [0,T] \\    
       \overline{\f{J}}_h \quad & \rightharpoonup\quad && \overline{\f{J}}&&\quad  \text{weakly-}^*\text{ in } L^\infty(0,T;\f{L}^2(\Omega))\\ 
          \f{b}(\cdot, \boldsymbol{\nu}\times \overline{\f{E}}_h)\quad & \rightharpoonup\quad && \overline{\f{b}}&&\quad \text{weakly-}^*\text{ in } L^\infty(0,T;\f{L}^2(\partial\Omega))\\
    \end{aligned}
    \end{equation}
    for  some  $(\f{E},\f{H})  \in      W^{1,\infty}(0,T;  \f{L}^2(\Omega)  \! \times \!   \f{L}^2(\Omega)) \cap  L^\infty(0,T; \f V  \times  \f{L}^2(\Omega))$ satisfying $(\f E, \f H)(0)= (\f E_0, \f H_0)$ and some $\overline{\f{J}} \in  L^\infty(0,T;\f{L}^2(\Omega)), \overline{\f{b}} \in  L^\infty(0,T;\f{L}^2(\partial\Omega))$.

 Our goal  is to prove the uniform convergence of the subsequence $(\f{E}_{h},\f{H}_{h})$ towards a solution of  \eqref{eq:maxwell intro full}.    Thanks to the uniqueness of solutions to  \eqref{eq:maxwell intro full}  (see Appendix \ref{appen a}), standard arguments  show that the whole sequence converges uniformly. According to    
 \eqref{time-space-discrete} and based on our construction  of  the interpolants (cf. \eqref{eq:interpol}),  it holds for all $h>0$ and all $t \in (0,T]$:
{\small \begin{equation}\tag{$\textnormal{P}_{h}$}\label{eq:P-hat} \left\{
\begin{aligned}
&\int_\Omega \!\!\overline{\varepsilon}_{h}(t)  \partial_t \f{E}_{h}(t) \! \cdot \!  \f{v}_h\,dx = \!\! \int_\Omega \!\!\overline{\f{f}}_{h}(t) \!  \cdot \!  \f{v}_h \! -\! \overline{\f{J}}_{h}(t) \! \cdot \!  \f{v}_h + \overline{\f{H}}_{h}(t) \!  \cdot \!  \curl\,\f{v}_h  \, dx \!
\\ &  \phantom{\int_\Omega \overline{\varepsilon}_{h}(t)  \partial_t \f{E}_{h}(t) \! \cdot \!  \f{v}_h\,dx} - \!\!\int_{\partial \Omega} \!\!  \!\!   \f{b}(\cdot , \!\boldsymbol{\nu} \!\times \!\overline{\f{E}}_{h}(t)) \! \cdot  \! \boldsymbol{\nu} \! \times \! \f{v}_h\,dS \quad \forall \f{v}_h\in \textbf{ND}_h,\\
&\int_\Omega (\overline{\mu}_h \partial_t \f{H}_{h}(t) + \curl\,\overline{\f{E}}_{h}(t) + \overline{\mu}_{t,h}(t)\overline{\f{H}}_h)\cdot \f{w}_h\,dx  = 0\quad \forall \f{w}\in \textbf{DG}_h. 
\end{aligned}
\right.
\end{equation}}
For later use, we introduce, for every $h>0$ and $t\in(0,T]$, the operator $\Phi_h(t)\in \f{V}^*$ by
\begin{equation}\label{eq:Phi_h-def}
    \Phi_h(t)\f{v}\coloneqq \int_\Omega \overline{\f{f}}_{h}(t)   \cdot   \f{v}  - \overline{\f{J}}_{h}(t)  \cdot   \f{v} + \overline{\f{H}}_{h}(t)   \cdot   \curl\,\f{v}  \, dx 
- \int_{\partial \Omega}    \f{b}(\cdot , \boldsymbol{\nu} \times \overline{\f{E}}_{h}(t))  \cdot   \boldsymbol{\nu}  \times  \f{v}\,dS.
\end{equation}
 By definition, we have
\begin{equation}\label{eq:Phi_h-eq}
    \Phi_h(t)\f{v}_h = \int_\Omega \overline{\varepsilon}_{h}(t)  \partial_t \f{E}_{h}(t)  \cdot   \f{v}_h\,dx\quad \forall \f{v}_h \in \textbf{ND}_h.
\end{equation}
Taking into account Proposition \ref{theorem:stability}, \eqref{eq:fconv}, \ref{ass:A2} and \hyperref[ass:A4]{(A4)}-\hyperref[ass:A6]{(A6)},   we also obtain \begin{equation}\label{eq:infty}
    \sup_{h>0,\, t \in (0,T]}\,\norm{\Phi_h (t)}_{\f{V}^*}< \infty.
\end{equation}  
In the following, we shortly address the difficulties in proving the uniform convergence and how we overcome them. The standard approach in our setting would be to first pass to the limit in \eqref{eq:P-hat} to show that $(\f{E},\f{H})$ (see \eqref{eq:convergences}) is a solution to \eqref{eq:maxwell intro full}. Afterwards, one subtracts \eqref{eq:P-hat} and the weak formulation of \eqref{eq:maxwell intro full} from each other and tests suitably, which yields a uniform bound between $(\f{E}_h,\f{H}_h)$ and $(\f{E},\f{H})$, thus establishing uniform convergence. In view of \eqref{eq:convergences}, the first step can only be carried out  provided that
\begin{equation} \label{explanation}
 \overline{\f{J}} = \f J(\cdot, \f E, \f H, \cdot). 
\end{equation}
 However, even if the mapping $\f{J} : \Omega \times \mathbb{R}^3 \times \mathbb{R}^3 \times [0,T] \to \mathbb{R}^3$ is assumed to be monotone with respect to its second and third arguments (cf. Remark \ref{rem:notmonotone}),   Minty's trick  cannot be applied to prove \eqref{explanation}. Furthermore, even  if $\f E^k_h \in  \textbf{ND}_h$ is assumed to be   discrete divergence-free for all $k=0,\ldots, N$,  the discrete compactness property   \cite{Kikuchi} yields only   the uniform convergence in $\f E_h$.  This alone  is not enough to conclude the limit  in   \eqref{explanation}. In conclusion, even under strengthened and unrealistic assumptions, monotonicity and discrete compactness arguments still fail. 

In light of the previous difficulties, one could instead try to consider the difference between $\eqref{eq:P-hat}$ and (P$_r$) for fixed $h,r>0$ in order to effectively employ the properties of $\f{J}$ and $\f{b}$. Since $\eqref{eq:P-hat}$ and (P$_r$) can only be tested with functions in $\textbf{ND}_h$ and $\textbf{ND}_r$, respectively, this strategy would result in error terms of the form
\begin{equation}\label{eq:errorcauchy}
    \overline{\f{E}}_r - \Pi_h \overline{\f{E}}_r,\quad \Pi_r \overline{\f{E}}_h - \overline{\f{E}}_h,
\end{equation}
(see also \eqref{eq:error}) which in general do not admit Cauchy-like behaviour,   as $\Pi_h \overline{\f{E}}_r$ and $\Pi_r \overline{\f{E}}_h$ are not expected to weakly converge when $h,r$ \emph{simultaneously} go to zero.

To cope with the aforementioned difficulties, we propose a new strategy based on a   modified Cauchy argument taking into account a differentiable   time-dependent Hilbert space as follows: For every arbitrarily fixed $h>0$, define the space 
$$\mathcal{H}_h(t) \coloneqq (\f{L}^2(\Omega)\times \f{L}^2(\Omega),(\cdot,\cdot)_{\mathcal{H}_h(t)}) \quad \forall t\in [0,T],$$
where the time-evolving inner product $(\cdot,\cdot)_{\mathcal{H}_h(t)}$ is given by
\begin{equation} \label{timedependentscalar}
((\f  e , \f h),(\f  u, \f v))_{\mathcal{H}_h(t) } :=  \int_{\Omega}  {\varepsilon}_h(t) \f e  \cdot \f u  \, dx +   \int_{\Omega}  {\mu}_h(t) \f h \cdot \f v  \, dx. 
\end{equation}
Notice that  the time-dependent inner product is defined using the piecewise linear interpolants $ \varepsilon_h,\mu_h \in W^{1,\infty}(0,T;L^\infty(\Omega)^{3\times 3})$, rather than the piecewise constant interpolants $\overline{\varepsilon}_h,\overline{\mu}_h \in L^{\infty}(0,T;L^\infty(\Omega)^{3\times 3})$. This  guarantees that the corresponding differentiability   is preserved in $\mathcal{H}_h(t)$, which is essential for our subsequent analysis.  Now, they key idea to deal with the error terms presented in \eqref{eq:errorcauchy} is  to first take the limit inferior as $r\to 0$ (for fixed $h$) and afterwards let $h\to 0$. To accomplish this, we employ the weakly lower   semicontinuity of the squared norm and the pointwise weak convergence in \eqref{eq:convergences}, which yields 
\begin{equation}\label{eq:liminf}
    \begin{aligned}
       \min\{\underline \varepsilon, \underline \mu\} \norm{(\f{E}_h - \f{E},\f{H}_h-\f{H})(t)}_{\f{L}^2(\Omega)^2}^2 \leq  \norm{(\f{E}_h - \f{E},\f{H}_h-\f{H})(t)}_{\mathcal{H}_h(t)}^2\\
        \leq  \liminf_{r\to 0} \norm{(\f{E}_h - \f{E}_r,\f{H}_h-\f{H}_r)(t)}_{\mathcal{H}_h(t)}^2 \quad \forall h >0, t \in [0,T].
    \end{aligned}
    \end{equation}
The main task is to prove that the limit inferior is arbitrarily small, uniformly on $[0,T]$, for all sufficiently small $h$, which is going to involve Cauchy-type arguments. Based on this strategy, the uniform convergence is established in the following theorem:

 \begin{theorem}\label{theorem:convergence}    \normalfont  Let Assumption \ref{ass} be satisfied. Then, it holds that
   \begin{equation}\label{eq:inequality}
  {\normalfont(\f{E}_h,\f{H}_h)\quad \to\quad (\f{E},\f{H})\quad \text{in } C([0,T],\f{L}^2(\Omega) \times \f{L}^2(\Omega)) \quad \text{as } h\to 0}\\
   \end{equation}
    with   $(\f{E},\f{H})  \in      W^{1,\infty}(0,T;  \f{L}^2(\Omega)  \! \times \!   \f{L}^2(\Omega)) \cap  L^\infty(0,T; \f V  \times  \f{L}^2(\Omega))$  being the  unique solution to {\normalfont \eqref{eq:maxwell intro full}}.
    \end{theorem}
    \begin{remark}To simplify our notation in the upcoming proof, we   let $c>0$ denote a generic constant, independent of $h,r,t$, which may vary from one expression to another.       \end{remark}

\begin{proof} By the previously described strategy, we investigate the limit inferior in   \eqref{eq:liminf}. To this aim, let $h,r>0$ and $t \in (0,T]$ be arbitrary but fixed. We introduce the error terms
\begin{equation}\label{eq:error}
\begin{split}
    \mathcal{R}_{h,r}^{E,H}(t)\coloneqq &\,\norm{(\f{E}_h,\f{E}_r)(t) - (\overline{\f{E}}_h,\overline{\f{E}}_r)(t)}_{\f{L}^2(\Omega)^2} +\norm{(\f{H}_h,\f{H}_r)(t) - (\overline{\f{H}}_h,\overline{\f{H}}_r)(t)}_{\f{L}^2(\Omega)^2}\\
     \mathcal{R}_{h,r}^{\varepsilon,\mu}(t)\coloneqq &\,\norm{\overline{\varepsilon}_h(t) - \overline{\varepsilon}_r(t)}_{L^\infty(\Omega)^{3\times 3}} + \norm{\varepsilon_h(t) - \overline{\varepsilon}_h(t)}_{L^\infty(\Omega)^{3\times 3}} + \norm{\overline{\mu}_h(t) - \overline{\mu}_r(t)}_{L^\infty(\Omega)^{3\times 3}}\\
    &  + \norm{\mu_h(t) - \overline{\mu}_h(t)}_{L^\infty(\Omega)^{3\times 3}} + \norm{\overline{\mu}_{t,h}(t) - \overline{\mu}_{t,r}(t)}_{L^\infty(\Omega)^{3\times 3}},\\
    \mathcal{R}_{h,r}^{\Pi,1}(t)\coloneqq &\,(\overline{\varepsilon}_h(t)\partial_t \f{E}_h(t), \Pi_h \overline{\f{E}}_r(t) -\overline{\f{E}}_r(t))_{\f{L}^2(\Omega)} - (\overline{\varepsilon}_r(t)\partial_t \f{E}_r(t), \overline{\f{E}}_h(t) - \Pi_r \overline{\f{E}}_h(t))_{\f{L}^2(\Omega)},\\
     \mathcal{R}_{h,r}^{\Pi,2}(t)\coloneqq &\, \Phi_h(\overline{\f{E}}_r(t) - \Pi_h \overline{\f{E}}_r(t)) - \Phi_r(\Pi_r \overline{\f{E}}_h(t) - \overline{\f{E}}_h(t)),\\
     \mathcal{R}_{h,r}^{Q}(t) \coloneqq &\,\left(\overline{\mu}_h(t)\partial_t \f{H}_h(t)+\overline{\mu}_{t,h}(t)\overline{\f{H}}_h(t),\f{Q}_h\overline{\f{H}}_r(t)-\overline{\f{H}}_r(t)\right)_{\f{L}^2(\Omega)}\\
    & +\left(\overline{\mu}_r(t)\partial_t\f{H}_r(t)+\overline{\mu}_{t,r}(t)\overline{\f{H}}_r(t),\f{Q}_r\overline{\f{H}}_h(t)-\overline{\f{H}}_h(t)\right)_{\f{L}^2(\Omega)}.
   \end{split}
\end{equation}
Note  that these error quantities are introduced to improve the  readability of the proof. Although they may yield suboptimal estimates, utilizing them is  sufficient to prove the  uniform convergence result. In view of  \eqref{timedependentscalar} along with $ \varepsilon_h,\mu_h \in W^{1,\infty}(0,T;L^\infty(\Omega)^{3\times 3})$, it holds  that 
    \begin{align*}
 \frac{1}{2}\partial_t\norm{(\f{E}_h - \f{E}_r,\f{H}_h-\f{H}_r)(t)}_{\mathcal{H}_h(t)}^2 & =   \frac{1}{2}(\partial_t {\varepsilon}_h(t)(\f{E}_h - \f{E}_r)(t), (\f{E}_h - \f{E}_r)(t))_{\f{L}^2(\Omega)}\\
 & + \frac{1}{2}(\partial_t {\mu}_h(t)(\f{H}_h - \f{H}_r)(t), (\f{H}_h - \f{H}_r)(t))_{\f{L}^2(\Omega)}\\
  &   + ( {\varepsilon}_h(t) \partial_t(\f{E}_h - \f{E}_r)(t), (\f{E}_h - \f{E}_r)(t))_{\f{L}^2(\Omega)}\\
   &  + (\mu_h\partial_t(\f{H}_h - \f{H}_r)(t),(\f{H}_h-\f{H}_r)(t))_{\f{L}^2(\Omega)}
   =: I_\varepsilon + I_\mu + I_E + I_H,
    \end{align*}
    where all time derivatives  above are understood in the weak sense.
Due to \hyperref[ass:A1]{(A1)} and Proposition \ref{theorem:stability}, we have $$I_\varepsilon \leq c \norm{\f{E}_h(t) - \f{E}_r(t)}_{\f{L}^2(\Omega)}^2,\quad I_\mu \leq c \norm{\f{H}_h(t) - \f{H}_r(t)}_{\f{L}^2(\Omega)}^2$$ 
and
\begin{equation}\label{eq:I_E}
   I_E  \leq (\overline{\varepsilon}_h(t)\partial_t\f{E}_h(t) - \overline{\varepsilon}_r(t)\partial_t\f{E}_r(t), \overline{\f{E}}_h(t) - \overline{\f{E}}_r(t))_{\f{L}^2(\Omega)} + c\mathcal{R}_{h,r}^{E,H}(t)+c\mathcal{R}_{h,r}^{\varepsilon,\mu}(t). 
\end{equation}
The main difficulty now lies in the estimation of the scalar product in \eqref{eq:I_E}. Based on \eqref{eq:Phi_h-eq}, we can write
\begin{equation}\label{eq:est1}
\begin{aligned}
   (\overline{\varepsilon}_h(t)\partial_t\f{E}_h(t) - \overline{\varepsilon}_r(t)\partial_t\f{E}_r(t),&  \overline{\f{E}}_h(t) - \overline{\f{E}}_r(t))_{\f{L}^2(\Omega)}  = (\overline{\varepsilon}_h(t)\partial_t \f{E}_h(t), \overline{\f{E}}_h(t) - \Pi_h \overline{\f{E}}_r(t))_{\f{L}^2(\Omega)} \\
   & - (\overline{\varepsilon}_r(t)\partial_t \f{E}_r(t), \Pi_r \overline{\f{E}}_h(t) - \overline{\f{E}}_r(t))_{\f{L}^2(\Omega)}+\mathcal{R}^{\Pi,1}_{h,r}(t)\\
   & = \Phi_h(\overline{\f{E}}_h(t) - \Pi_h \overline{\f{E}}_r(t)) - \Phi_r(\Pi_r \overline{\f{E}}_h(t) - \overline{\f{E}}_r(t)) +\mathcal{R}^{\Pi,1}_{h,r}(t)\\
   & = \Phi_h(\overline{\f{E}}_h(t) - \overline{\f{E}}_r(t)) - \Phi_r(\overline{\f{E}}_h(t) - \overline{\f{E}}_r(t)) + \mathcal{R}^{\Pi,1}_{h,r}(t) + \mathcal{R}_{h,r}^{\Pi,2}(t).
\end{aligned}
\end{equation}
From the definition   \eqref{eq:Phi_h-def} of $\Phi_h$, we see that the first two terms in the last line of \eqref{eq:est1} are equal to
\begin{align*}
    &  (\overline{\f{f}}_h(t) - \overline{\f{f}}_r(t), \overline{\f{E}}_h(t) - \overline{\f{E}}_r(t))_{\f{L}^2(\Omega)} - (\overline{\f{J}}_h(t) - \overline{\f{J}}_r(t), \overline{\f{E}}_h(t) - \overline{\f{E}}_r(t))_{\f{L}^2(\Omega)}\\
    & + (\overline{\f{H}}_h(t) - \overline{\f{H}}_r(t), \curl(\overline{\f{E}}_h(t) - \overline{\f{E}}_r(t)))_{\f{L}^2(\Omega)}-(\f{b}(\cdot,\boldsymbol{\nu}\times \overline{\f{E}}_h(t)) - \f{b}(\cdot,\boldsymbol{\nu}\times \overline{\f{E}}_r(t)), \boldsymbol{\nu}\times (\overline{\f{E}}_h(t) \\
    & \coloneqq T_f + T_J + T_{\text{curl}} + T_b.
\end{align*}
By taking into account the previous estimates for $I_\varepsilon,I_\mu$ and $I_E$, we arrive at
\begin{equation}\begin{aligned}\label{eq:estmid}
\frac{1}{2}\partial_t\norm{(\f{E}_h - \f{E}_r,\f{H}_h-\f{H}_r)(t)}_{\mathcal{H}_h(t)}^2  \leq c\Bigl(  \norm{\f{E}_h(t) - \f{E}_r(t)&}_{\f{L}^2(\Omega)}^2 + \mathcal{R}_{h,r}^{E,H}(t) + \mathcal{R}_{h,r}^{\varepsilon,\mu}(t)+\lvert \mathcal{R}^{\Pi,1}_{h,r}(t)\rvert \\
& + \lvert \mathcal{R}^{\Pi,2}_{h,r}(t)\rvert + I_H + T_f + T_J + T_\text{curl} + T_b \Bigr).
\end{aligned}
\end{equation}
It remains to estimate the last five terms appearing in the bracket on the right-hand side of \eqref{eq:estmid}. In the following, we will be using Proposition \ref{theorem:stability} without further mention. First of all, it is clear that
$$T_f \leq c \norm{\overline{\f{f}}_h(t) - \overline{\f{f}}_r(t)}_{\f{L}^2(\Omega)} \quad \textrm{and} \quad T_b\leq 0$$
due to the monotonicity of $\f{b}$. Next, let $(n,m)\in \{1,\dots,N(h)\}\times\{1,\dots, N(r)\}$ be the unique element satisfying
$$t\in \left((n-1)\frac{T}{N(h)}, n\frac{T}{N(h)}\right]\cap \left((m-1)\frac{T}{N(r)}, m\frac{T}{N(r)}\right],$$
which, in particular, implies that
$$\left\lvert n\frac{T}{N(h)} - m\frac{T}{N(r)}\right\rvert \leq \frac{T}{N(h)}+\frac{T}{N(r)}.$$
Consequently, by \ref{ass:A4} and \eqref{eq:interpol}, we have
\begin{align*}
    \norm{\overline{\f{J}}_h(t) - \overline{\f{J}}_r(t)}_{\f{L}^2(\Omega)}& = \norm{ \f{J}\left(\cdot,\f{E}^n_h, \f{H}_h^n, n\frac{T}{N(h)}\right) -  \f{J}\left(\cdot,\f{E}^m_r, \f{H}_r^m, m\frac{T}{N(r)}\right)}_{\f{L}^2(\Omega)}\\
    & \leq c\Bigl(\norm{(\overline{\f{E}}_h(t)- \overline{\f{E}}_r(t),\overline{\f{H}}_h(t)- \overline{\f{H}}_r(t))}_{\f{L}^2(\Omega)^2} + \frac{1}{N(h)}+\frac{1}{N(r)}\Bigr)\\
    & \leq c\Bigl(\norm{(\f{E}_h(t)- \f{E}_r(t),\f{H}_h(t)- \f{H}_r(t))}_{\f{L}^2(\Omega)^2} +\mathcal{R}_{h,r}^{E,H}(t)+\frac{1}{N(h)}+\frac{1}{N(r)}\Bigr).
\end{align*}
Now, using the above estimate in combination with
\begin{align*}
    & \norm{(\f{E}_h(t)- \f{E}_r(t),\f{H}_h(t)- \f{H}_r(t))}_{\f{L}^2(\Omega)^2}\norm{\overline{\f{E}}_h(t) - \overline{\f{E}}_r(t)}_{\f{L}^2(\Omega)}\\
    & \leq \norm{(\f{E}_h(t)- \f{E}_r(t),\f{H}_h(t)- \f{H}_r(t))}_{\f{L}^2(\Omega)^2}\left(\norm{\f{E}_h(t)- \f{E}_r(t)}_{\f{L}^2(\Omega)} + \mathcal{R}_{h,r}^{E,H}(t)\right)\\
    & \leq \norm{(\f{E}_h(t)- \f{E}_r(t),\f{H}_h(t)- \f{H}_r(t))}_{\f{L}^2(\Omega)^2}^2 + c\mathcal{R}_{h,r}^{E,H}
(t),\end{align*}
it follows that
\begin{align*}
    &T_J  \leq \norm{\overline{\f{J}}_h(t) - \overline{\f{J}}_r(t)}_{\f{L}^2(\Omega)} \norm{\overline{\f{E}}_h(t) - \overline{\f{E}}_r(t)}_{\f{L}^2(\Omega)}\\
    &\leq c\Bigl(\norm{(\f{E}_h(t) - \f{E}_r(t),\f{H}_h(t) - \f{H}_r(t))}_{\f{L}^2(\Omega)^2}^2 +\mathcal{R}_{h,r}^{E,H}(t)+ \frac{1}{N(h)}+\frac{1}{N(r)}\Bigr).
\end{align*} 
 Finally, using the second variational equalities in \eqref{eq:P-hat} and (P$_r$) with the test functions
 $$\overline{\f{H}}_h(t) - \f{Q}_h \overline{\f{H}}_r(t)\quad \text{and} \quad \f{Q}_r \overline{\f{H}}_h(t) - \overline{\f{H}}_r(t),$$
respectively, we obtain
\begin{equation}
\begin{split}
T_{\mathrm{curl}}
&= \left(\overline{\mu}_r(t)\partial_t\f{H}_r(t)-\overline{\mu}_h(t)\partial_t\f{H}_h(t),
\overline{\f{H}}_h(t)-\overline{\f{H}}_r(t)\right)_{\f{L}^2(\Omega)}\\
&\quad + \left(\overline{\mu}_{t,r}(t)\overline{\f{H}}_r(t)-\overline{\mu}_{t,h}(t)\overline{\f{H}}_h(t),
\overline{\f{H}}_h(t)-\overline{\f{H}}_r(t)\right)_{\f{L}^2(\Omega)}
+\mathcal{R}_{h,r}^{Q}(t).
\end{split}
\end{equation}
Thus, summing up with $I_H= (\mu_h\partial_t(\f{H}_h - \f{H}_r)(t),(\f{H}_h-\f{H}_r)(t))_{\f{L}^2(\Omega)}$, we see that
\begin{equation}
I_H+T_{\mathrm{curl}}
\leq c\left(
\norm{\f{H}_h(t)-\f{H}_r(t)}_{\f{L}^2(\Omega)}^2
+\mathcal{R}_{h,r}^{E,H}(t)
+\mathcal{R}_{h,r}^{\varepsilon,\mu}(t)
+\left|\mathcal{R}_{h,r}^{Q}(t)\right|
\right).
\end{equation}
By applying the estimates for $I_H, T_f, T_J,T_\text{curl},T_b$ to \eqref{eq:estmid}, we conclude that
$$\frac{1}{2}\frac{d}{dt}\norm{(\f{E}_h(t) - \f{E}_r(t), \f{H}_h(t) - \f{H}_r(t))}_{\mathcal{H}_h(t)}^2\leq c\bigg(\norm{(\f{E}_h(t) - \f{E}_r(t), \f{H}_h(t) - \f{H}_r(t))}_{\mathcal{H}_h(t)}^2 +  \Psi_{h,r}(t)\bigg), $$
with
$$ 
    \Psi_{h,r}(t)  \coloneqq  \mathcal{R}_{h,r}^{E,H}(t) + \mathcal{R}_{h,r}^{\varepsilon,\mu}(t)+\lvert \mathcal{R}_{h,r}^{\Pi,1}(t)\rvert + \lvert\mathcal{R}_{h,r}^{\Pi,2}(t)\rvert+ \lvert\mathcal{R}_{h,r}^Q(t)\rvert+\norm{\overline{\f{f}}_h(t) - \overline{\f{f}}_r(t)}_{\f{L}^2(\Omega)}+  \frac{1}{N(h)}+\frac{1}{N(r)}.
$$
Now, Grönwall's inequality in differential form implies
\begin{align*}
    \norm{(\f{E}_h - \f{E}_r,\f{H}_h-\f{H}_r)(t)}_{\mathcal{H}_h(t)}^2\leq c\left(\norm{(\f{E}_h^0 - \f{E}_r^0,\f{H}_h^0-\f{H}_r^0)}_{\mathcal{H}_0}^2 + \int_0^t \Psi_{h,r}(s)\,ds\right).
\end{align*}
By employing the properties of $\Pi_h,\f{Q}_h$ and $\Phi_h$ (in particular \eqref{eq:qhest} and \eqref{eq:infty}), Lemma \ref{lemma:V}, \eqref{eq:E_N overline E_N}-\eqref{eq:fconv} as well as the fact that $h$ is fixed, it follows that
$$\lim_{r\to 0} \int_0^t \Psi_{h,r}(s)\,ds \leq c\int_0^t \Gamma_h(s)\,ds,$$
    where
\begin{align*}
    \Gamma_h(s) & \coloneqq \norm{(\f{E}_h(s)-\overline{\f{E}}_h(s),\f{H}_h(s)-\overline{\f{H}}_h(s))}_{\f{L}^2(\Omega)^2} + \norm{\overline{\varepsilon}_h(s) - \varepsilon(s)}_{L^\infty(\Omega)^{3\times 3}}+ \norm{\varepsilon_h(s) - \overline{\varepsilon}_h(s)}_{L^\infty(\Omega)^{3\times 3}}\\
    & +\norm{\overline{\mu}_h(s) - \mu(s)}_{L^\infty(\Omega)^{3\times 3}}+\norm{\mu_h(s) - \overline{\mu}_h(s)}_{L^\infty(\Omega)^{3\times 3}} +\norm{\overline{\mu}_{t,h}(s) - \partial_t \mu (s)}_{L^\infty(\Omega)^{3\times 3}}\\
    & +   \Bigl|(\overline{\varepsilon}_h(s) \partial_t\f{E}_h(s), \Pi_h \f{E}(s) - \f{E}(s))_{\f{L}^2(\Omega)} \Bigr|+ |\Phi_h(\f{E}(s) - \Pi_h \f{E}(s))|\\
    & +\left|\left(\overline{\mu}_h(s)\partial_t\f{H}_h(s)+\overline{\mu}_{t,h}(s)\overline{\f{H}}_h(s),\f{Q}_h\f{H}(s)-\f{H}(s)\right)_{\f{L}^2(\Omega)}\right|+\norm{\overline{\f{f}}_h(s) - \f{f}(s)}_{\f{L}^2(\Omega)}+\frac{1}{N(h)}.
\end{align*}
Taking into account \eqref{eq:liminf} and \eqref{eq:initialvalues}, we therefore arrive at
\begin{equation}\label{eq:psihrliminf}
\begin{aligned}
\norm{(\f{E}_h - \f{E},\f{H}_h-\f{H})(t)}_{\f{L}^2(\Omega)^2}^2 & \leq c\left(\norm{(\f{E}_h^0 - \f{E}_0,\f{H}_h^0-\f{H}_0)}_{\mathcal{H}_0}^2 + \int_0^T \Gamma_h(s)\,ds\right).
\end{aligned}
\end{equation}
Now, we see that the right-hand side of \eqref{eq:psihrliminf} is independent of $t$ and tends to 0 as $h\to 0$ by Lemma \ref{lemma:V}, \eqref{eq:qhest}, \eqref{eq:initialvalues}, \eqref{eq:E_N overline E_N}, \eqref{eq:coefficientest}, \eqref{eq:fconv} and \eqref{eq:infty}. Thus,  since  $t \in [0,T]$ was chosen arbitrarily, \eqref{eq:psihrliminf} implies that
\begin{equation}\label{eq:strongconvergence}
    (\f{E}_h, \f{H}_h)\quad \to\quad  (\f{E},\f{H})\quad \text{in } C([0,T], \f{L}^2(\Omega) \times\f{L}^2(\Omega))\text{ as } h\to 0.
\end{equation}
Consequently, along with \hyperref[ass:A4]{(A4)}, the above convergence yields
\begin{equation}\label{eq:Jconv}
    \overline{\f{J}}_h \quad \to \quad  \f{J}(\cdot,\f{E},\f{H},\cdot)\quad \text{in } L^\infty(0,T;\f{L}^2(\Omega))\text{ as } h\to 0.
\end{equation}

We are left to show that $(\f{E},\f{H})$ is a solution to \eqref{eq:maxwell intro full}. To see that $(\f{E},\f{H})$ satisfies the first two equations in \eqref{eq:maxwell intro full}, we choose an arbitrary test function $(\f{v},\f{w})\in \f{H}_0(\curl)\times \f{L}^2(\Omega)$ and a sequence $\{(\f{v}_h,\f{w}_h)\}_h \subset \textbf{ND}_h\times \textbf{DG}_h$ such that $(\f{v}_h,\f{w}_h) \to (\f{v},\f{w})$ in $\f{V}\times \f{L}^2(\Omega)$ as $h\to 0$, which is possible by Lemma \ref{lemma:V} and \eqref{eq:qhest}. After passing to the limit in \eqref{eq:P-hat}, by invoking  \eqref{eq:convergences}-\eqref{eq:fconv} and \eqref{eq:strongconvergence}-\eqref{eq:Jconv}, we arrive at $\f{H}\in L^2(0,T;\f{H}(\curl))$ and
\begin{equation}\label{eq:system}
    \begin{cases}
      \epsilon\partial_t\f{E} - \curl\, \textbf{\textit{H}} + \f{J}(\cdot,\f{E},\f{H},\cdot) = \textbf{\textit{f}} & \text{in } \Omega \times (0,T)\\
      \mu\partial_t\f{H} + \curl\, \textbf{\textit{E}} + \partial_t \mu \f{H} = 0 & \text{in } \Omega \times (0,T)\\
      (\f{E},\f{H})(0)=(\f{E}_0,\f{H}_0).
    \end{cases}       
\end{equation}
Again, by choosing an arbitrary $\f{v}\in \f{V}$ and arguing similarly as before, we see via the partial integration formula \eqref{eq:partial} and \eqref{eq:convergences} that $\boldsymbol{\nu} \times \textbf{\textit{H}} = \boldsymbol{\nu} \times  \overline{\f{b}}$. It now remains to prove that $\overline{\f{b}} = \f{b}(\cdot,\boldsymbol{\nu}\times \f{E}).$ Since the mapping
$$\Phi:\f{L}^2(0,T;\f{L}^2(\partial\Omega))\to \f{L}^2(0,T;\f{L}^2(\partial\Omega)),\quad \f{u}\mapsto \f{b}(\cdot,\f{u}),$$
is continuous and monotone by \hyperref[ass:A5]{(A5)}, it is also maximal monotone. By employing \eqref{eq:P-hat}, \eqref{eq:convergences}-\eqref{eq:fconv}, and \eqref{eq:strongconvergence}-\eqref{eq:system} in combination with \eqref{eq:partial} and the definition of $\f{V}$, we calculate:
{\small
\begin{equation*}
    \begin{aligned}
    & \int_0^T (\f{b}(\cdot,\boldsymbol{\nu}\times \overline{\f{E}}_h(t)),\boldsymbol{\nu}\times \overline{\f{E}}_h(t))_{\f{L}^2(\partial \Omega)}\,dt = \int_0^T (\overline{\f{f}}_h(t), \overline{\f{E}}_h(t))_{\f{L}^2(\Omega)}-(\overline{\f{J}}_h(t),\overline{\f{E}}_h(t))_{\f{L}^2(\Omega)} + (\overline{\f{H}}_h(t),\curl\,\overline{\f{E}}_h(t))_{\f{L}^2(\Omega)}\,dt\\
    & -\int_0^T (\overline{\varepsilon}_h(t)\partial_t \f{E}_h(t),\overline{\f{E}}_h(t))_{\f{L}^2(\Omega)}\,dt \overset{h\to 0}{\rightarrow} \int_0^T (\f{f}(t), \f{E}(t))_{\f{L}^2(\Omega)}-(\f{J}(\cdot,\f{E}(t),\f{H}(t),t),\f{E}(t))_{\f{L}^2(\Omega)} + (\f{H}(t),\curl\,\f{E}(t))_{\f{L}^2(\Omega)}\,dt\\
    & -\int_0^T (\varepsilon(t)\partial_t \f{E}(t),\f{E}(t))_{\f{L}^2(\Omega)}\,dt = \int_0^T (\f{H}(t),\curl\,\f{E}(t))_{\f{L}^2(\Omega)} - (\curl\,\f{H}(t),\f{E}(t))_{\f{L}^2(\Omega)}\,dt = \int_0^T (\overline{\f{b}},\boldsymbol{\nu}\times \f{E}(t))_{\f{L}^2(\partial \Omega)}\,dt.
    \end{aligned}
\end{equation*}}
The conclusion follows from \cite[Chapter 4, Proposition 1.6]{showalter}.\end{proof}

 \begin{subsection}{Exponential stability}\label{section:exponentialstability}
This section analyzes the exponential stability of the proposed numerical scheme as $n \to \infty$, where all given data may now vary over the infinite time horizon $[0, \infty)$. Throughout this section, let $h, \tau  >0$ and set $t_n \coloneqq n\tau$ for every $n\in\mathbb{N}_0$.  

 \begin{assumption}\label{assumption2}
  
\begin{enumerate}[label=(B\arabic*)]

{\normalfont
\item[]
\item \phantomsection \label{ass:B1} The electric permittivity satisfies $\varepsilon\in W^{2,\infty}_{\text{loc}}(0,\infty;L^\infty_{\text{sym}}(\Omega)^{3\times 3}) \cap W^{1,\infty}(0,\infty;L^\infty_{\text{sym}}(\Omega)^{3\times 3})$, and  the magnetic permeability  satisfies  $\mu\in W^{2,\infty}(0,\infty;L^\infty_\text{sym}(\Omega)^{3\times 3})$  and $\mu(t)=\mu_\infty$ for all $t\geq T_\mu$ for some $T_\mu\geq 0$ and uniformly positive definite $\mu_\infty\in L_\text{sym}^\infty(\Omega)^{3\times 3}$. Furthermore, \eqref{eq:uniformposdef} is valid  for a.e. $x\in \Omega$ and all $t\in [0,\infty)$.

\item \phantomsection \label{ass:B2} The total current density $\f{J}_{tot}:\Omega\times \mathbb{R}^3\times \mathbb{R}^3\times [0,\infty)\to \mathbb{R}^3$ is measurable w.r.t the first variable with $\f{J}_{tot}(\cdot,0,0,0)\in \f{L}^2(\Omega)$.   There exists  $L_{tot}>0$ such that  $\f{J}_{tot}$ satisfies \eqref{eq:J_tot-estimate} for a.e. $x\in \Omega$ and all $\xi,\xi',\eta,\eta'\in \mathbb{R}^3, t,t'\in [0,\infty)$. Furthermore, there exists $T_J\geq 0$ such that  for all $t \ge T_J$ and $\lambda \ge 0$, 
$$ \f{J}_{tot}(x,\lambda \xi,\eta,t) = \lambda \f{J}_{tot}(x,\xi,\eta,t) \quad \textrm{and} \quad \f{J}_{tot}(x,\xi,\eta,t)=\f{J}_{tot}(x,\xi,\eta',t) $$ 
hold for a.e. $x \in \Omega$ and  all $\xi,\eta,\eta'\in \mathbb{R}^3$.
\item \phantomsection \label{ass:B3} The nonlinearity $\f{b}$ satisfies \ref{ass:A5} and $\f{b}(x,\lambda \xi) = \lambda \f{b}(x,\xi)$   for a.e. $ x\in \partial \Omega$ and all $\xi \in \mathbb{R}^3,\lambda\geq 0.$
\item \phantomsection \label{ass:B4} It holds that $\f{f}\in H^1(0,\infty;\f{L}^2(\Omega))$ and there exists $T_f \ge0$ such that   $\f{f}(t) = 0$ for all $t \ge T_f$.
\item \phantomsection\label{ass:B5} There exists $\gamma > 0$ such that  
\begin{align*}
    & \frac{1}{2}(\varepsilon(t_n)(\f{v}_h-\hat{\f{v}}_h),\f{v}_h-\hat{\f{v}}_h)_{\f{L}^2(\Omega)} + \tau (\f{J}(\cdot,\f{v}_h,t_n),\f{v}_h)_{\f{L}^2(\Omega)}\\
    & -\frac{1}{2}((\varepsilon(t_n)-\varepsilon(t_{n-1}))\hat{\f{v}}_h,\hat{\f{v}}_h)_{\f{L}^2(\Omega)}\geq \gamma \norm{\f{v}_h}_{\f{L}^2(\Omega)}^2   \quad \forall (\f{v}_h,\hat{\f{v}}_h)\in \textbf{ND}_h\times \textbf{ND}_h,  t_n \geq T_J.
\end{align*}
}
\end{enumerate}
\end{assumption}
\begin{remark}  

The condition \ref{ass:B2} states that for all $t \ge T_J$, the   total current density $\f{J}_{tot}$ becomes  positively homogeneous with respect to the second variable  (electric field). Let us emphasize that the positive homogeneity together with the Lipschitz condition \eqref{eq:J_tot-estimate}   readily implies that $\f{J}_{tot}$ must be independent of the third variable, which we have explicitly included in \ref{ass:B2}. Note that these   properties remain valid for the induced nonlinearity $\f{J}(x,\xi,\eta,t)=\f{J}_{tot}(x,\xi,\eta,t) + \partial_t\varepsilon(x,t)\xi$ (see \eqref{def J}). Thus, we have omitted the dependence of $\f{J}$ on  $\eta$ in \ref{ass:B5} and in the following analysis as soon as  $t \ge T_J$.

\end{remark}
 \begin{remark} The positive homogeneity of $\f b$ in the second component along with \eqref{eq:propertyb} implies for a.e. $x \in \partial \Omega$ and all $\lambda >0$, $\xi \in \mathbb R$ that 
 $ \lambda |\f b (x, \xi)| = |\f b (x, \lambda \xi)| \le \overline b (\lambda |\xi|+1) $  and so $|\f b (x, \xi)| \le \overline b (|\xi| + 1/\lambda)$. Thus letting $\lambda \to \infty$, we conclude that 
 \begin{equation} \label{con:b}
   \lvert  \textbf{\textit{b}}(x,\xi) \rvert \leq \overline b  \lvert \xi\rvert   \  \textrm{for a.e. }  x\in \partial \Omega \text{ and all }  \xi \in \mathbb R^3
 \end{equation}   
 holds under \ref{ass:B3}.
 \end{remark}
\begin{remark} \label{rem:(B5)}
Assumption \ref{ass:B5} is satisfied if there exist constants $c_J,L_\varepsilon>0$ such that
        $$(\f{J}(\cdot,\f{v}_h,t_n),\f{v}_h)_{\f{L}^2(\Omega)}\geq  c_J \norm{\f{v}_h}_{\f{L}^2(\Omega)}^2,\quad \norm{\varepsilon(t_n)-\varepsilon(t_{n-1})}_{L^\infty(\Omega)^{3\times 3}}\leq \tau L_\varepsilon,\quad \tau L_\varepsilon\leq \frac{1}{2}\underline{\varepsilon},\quad L_\varepsilon<c_J,$$
       for all $\f{v}_h \in \textbf{ND}_h$ and all $t_n \geq T_J,n\in \mathbb{N}$. In this case, we have 
       $$-\frac{1}{2}((\varepsilon(t_n)-\varepsilon(t_{n-1}))\hat{\f{v}}_h,\hat{\f{v}}_h)_{\f{L}^2(\Omega)}\geq -\frac{1}{2}\tau L_\varepsilon\norm{\hat{\f{v}}_h}_{\f{L}^2(\Omega)}^2\geq -\tau L_\varepsilon(\norm{\f{v}_h-\hat{\f{v}}_h}_{\f{L}^2(\Omega)}^2 +\norm{\f{v}_h}_{\f{L}^2(\Omega)}^2),$$
       from which it follows that \ref{ass:B5} holds with $\gamma\coloneqq \tau(c_J - L_\varepsilon)$.
\end{remark}
   Based on \eqref{time-space-discrete} and \eqref{eq:discrete-equation}, we now consider, for every $n\in \mathbb{N}$,  the following variational problem: Given $(\f{v},\f{w})\in \textbf{ND}_h\times \textbf{DG}_h$, find $(\f{E}^{\f{v},\f{w}},\f{H}^{\f{v},\f{w}})\in \textbf{ND}_h\times \textbf{DG}_h$ such that 
     \begin{equation}\label{eq:S_tau,h}
\begin{dcases}
\int_\Omega \varepsilon(t_n) \frac{\f{E}^{\f{v},\f{w}} - \f{v}}{\tau}\cdot \f{v}_h - \f{H}^{\f{v},\f{w}}\cdot \curl\,\f{v}_h +  \f{J}(\cdot,\f{E}^{\f{v},\f{w}},\f{H}^{\f{v},\f{w}}, t_n)\cdot \f{v}_h \,dx\\
+ \int_{\partial \Omega} \f{b}(\cdot,\boldsymbol{\nu}\times \f{E}^{\f{v},\f{w}})\cdot \boldsymbol{\nu}\times \f{v}_h\,dS = \int_\Omega \f{f}^h_n\cdot \f{v}_h\,dx \quad \forall \f{v}_h\in \textbf{ND}_h, \\
\mu_h^n \frac{\f{H}^{\f{v},\f{w}} - \f{w}}{\tau} + \curl\,\f{E}^{\f{v},\f{w}} + \mu_{t,h}^n \f{H}^{\f{v},\f{w}} = 0.
\end{dcases}
\end{equation}
Here, we recall that $\mu_h^n= \f{Q}_h \mu(t_n)$ and $\mu_{t,h}^n = \f{Q}_h \partial_t \mu(t_n) $. Note that   Assumption \ref{assumption2} particularly implies that   \ref{ass:A1}-\ref{ass:A5} hold on every finite time interval $[0,T]$.  Thus, Propositions \ref{theorem:existence} and \ref{theorem:existence2} provide (sufficient) conditions ensuring the existence and uniqueness of a solution to \eqref{eq:S_tau,h} for all $n\in \mathbb{N}$. For the remainder of this section, the well-posedness will be explicitly assumed:

 \begin{assumption}\label{assumption2b} \normalfont For every $n \in \mathbb N$, the solution operator associated with \eqref{eq:S_tau,h} $$S_{h,\tau}^n:\textbf{ND}_h\times \textbf{DG}_h\to \textbf{ND}_h\times \textbf{DG}_h,\quad (\f{v},\f{w})\mapsto (\f{E}^{\f{v},\f{w}},\f{H}^{\f{v},\f{w}})$$
 is well-defined. 
 \end{assumption}

 By the explicit use of the solution operator $S^n_{h,\tau}$, the fully discrete scheme \eqref{time-space-discrete} on the finite time interval $[0,T]$ can be concisely written as 
$$S_{h,\tau}^n(\f{E}_h^{n-1},\f{H}_h^{n-1}) = (\f{E}_h^n,\f{H}_h^n)\quad \forall n\in \{1,\dots,N\},\quad (\f{E}_h^0,\f{H}_h^0)\coloneqq (\Pi_h \f{E}_0, \f{Q}_h\f{H}_0)\in \textbf{ND}_h\times \textbf{DG}_h.$$
Our goal now is to investigate the asymptotic energy behavior of $(\f{E}_h^n,\f{H}_h^n)$   for $n \to \infty$. To this aim,  we define the discrete electromagnetic energy of $(\f{v},\f{w})\in \textbf{ND}_h \times \textbf{DG}_h$ at $t=t_n$ by 
\begin{equation}\label{eq:discrete-EME}
    \mathcal{E}_h^n(\f{v},\f{w}) \coloneqq  \frac{1}{2}\left((\varepsilon(t_n)\f{v},\f{v})_{\f{L}^2(\Omega)} + (\mu_h^n\f{w},\f{w})_{\f{L}^2(\Omega)}\right)  \quad \forall n \in \mathbb{N}_0
\end{equation}
and introduce the spaces
\begin{equation} \label{def:Wh}
\begin{aligned}
         \mathcal W_h & \coloneqq \textbf{ND}_h \times \mu_{h,\infty}^{-1}\curl\,\textbf{ND}_h\subset \f{L}^2(\Omega)\times \f{L}_{\mu_{h,\infty}}^2(\Omega), \quad \mu_{h,\infty}\coloneqq \f{Q}_h \mu_\infty,\\
     \mathcal{N}_h&\coloneqq \mathcal W_h^{\perp} = \{0\}\times \{\f{w}_h\in \textbf{DG}_h : (\f{w}_h,\curl\,\f{v}_h)_{\f{L}^2(\Omega)} = 0\quad \forall \f{v}_h\in \textbf{ND}_h\}.
     \end{aligned}
     \end{equation}
    In particular, it holds that
     $\textbf{ND}_h \times \textbf{DG}_h = \mathcal W_h\oplus \mathcal{N}_h.$
Let us now set 
$$m\coloneqq \min \left\{n\in \mathbb{N}_0: t_n \geq \max\{T_J,T_f,T_\mu\}\right\} = \left \lceil \frac{\max\{T_J,T_f,T_\mu\}}{\tau}\right\rceil.$$
Then, in view of \ref{ass:B1}-\ref{ass:B4} and \eqref{eq:S_tau,h}, it holds for all $n \ge m$ that 
\begin{equation}\label{eq:S_tau,hfinal}
 S_{h,\tau}^n(\f{v},\f{w}) = (\f{E}^{\f{v},\f{w}},\f{H}^{\f{v},\f{w}}) \  \Leftrightarrow \ 
\begin{dcases}
\int_\Omega \varepsilon(t_n) \frac{\f{E}^{\f{v},\f{w}} - \f{v}}{\tau}\cdot \f{v}_h - \f{H}^{\f{v},\f{w}}\cdot \curl\,\f{v}_h +  \f{J}(\cdot,\f{E}^{\f{v},\f{w}},  t_n)\cdot \f{v}_h \,dx\\
+ \int_{\partial \Omega} \f{b}(\cdot,\boldsymbol{\nu}\times \f{E}^{\f{v},\f{w}})\cdot \boldsymbol{\nu}\times \f{v}_h\,dS =0 \quad \forall \f{v}_h\in \textbf{ND}_h, \\
\mu_{h,\infty} \frac{\f{H}^{\f{v},\f{w}} - \f{w}}{\tau} + \curl\,\f{E}^{\f{v},\f{w}}  = 0.
\end{dcases}
\end{equation}
Also, for later use, we note that 
\begin{equation}\label{eq:alles0}
  \quad \f{b}(\cdot,0)=0 \quad \textrm{and} \quad  \f{J}(\cdot,0,t_n)=0 \quad \forall n \ge m.
\end{equation}
     The goal of the remainder of this section is to prove the following fully discrete  unconditional exponential stability result:

\begin{theorem}\label{theorem-exp-decay}
         \normalfont Let $h, \tau >0$  and suppose that Assumptions \ref{assumption2} and \ref{assumption2b} are satisfied. Let $(\f{E}_h^0,\f{H}_h^0)\in \textbf{ND}_h\times \textbf{DG}_h$ be given and set 
         \begin{equation} \label{eq:deftheo}(\f{E}_h^n,\f{H}_h^n)\coloneqq S_{h,\tau}^n(\f{E}_h^{n-1},\f{H}_h^{n-1})\quad \forall n\in \mathbb{N}.\end{equation}
         Furthermore, write 
         \begin{equation} \label{eq:decomptheo}(\f{E}_h^{m},\f{H}_h^{m}) = (\f{E}_{h}^*, \f{H}_{h}^*)+(0,\f{H}_{h,\infty})\in \mathcal W_h\oplus \mathcal{N}_h.
         \end{equation}
        Then, there exists a positive constant $c_{h,\tau}>0$ such that 
        \begin{equation}\label{eq:expoest} 
            \mathcal{E}_h^n({\f{E}}_h^n,{\f{H}}_h^n)\leq \mathcal{E}_h^{m}(0,\f{H}_{h,\infty}) + e^{-c_{h,\tau}(t_n-t_{m})}\mathcal{E}_h^{m}(\f{E}_h^*,\f{H}_h^*)   \quad \forall n\geq {m+1}.
        \end{equation}
     \end{theorem}\begin{corollary}[Exponential decay]\label{corollary:exponential-decay}
         \normalfont Let the   assumptions of Theorem \ref{theorem-exp-decay} be satisfied. If $(\f{E}_h^{m},\f{H}_h^{m})\in \mathcal W_h$, then $\{\mathcal{E}_h^n(\f{E}_h^n,\f{H}_h^n)\}_{n=1}^\infty$ decays exponentially. In particular, this holds  if $T_\mu=0$ and $(\f{E}_h^0,\f{H}_h^0)\in \mathcal W_h$.
     \end{corollary}
     \begin{proof}[Proof of Corollary \ref{corollary:exponential-decay}]
         The first part is clear from \eqref{eq:decomptheo} and \eqref{eq:expoest}. Next, let us assume that $(\f{E}_h^0,\f{H}_h^0)\in \mathcal W_h$ and $T_\mu=0$, i.e. $\mu(t)=\mu_\infty$ for all $t\in [0,\infty)$. Then by the second equation in \eqref{eq:S_tau,h} along with $\mu^n_{h,t} \equiv 0$ for all $n \in \mathbb N$,  it follows that $(\f{E}_h^n,\f{H}_h^n)\in \mathcal W_h$ for all $n\in \mathbb N$ and thus $(\f{E}_h^{m},\f{H}_h^{m})\in \mathcal W_h$, which completes our proof.
     \end{proof}
Before proceeding to the proof of Theorem \ref{theorem-exp-decay}, we shall need some auxiliary results.
     \begin{lemma}\label{lemma:energy-est}
         \normalfont Let $(\f{v},\f{w})\in \textbf{ND}_h \times \textbf{DG}_h$ and $n\geq m+1$. Then 
         \begin{equation}\label{eq:energy-est}
             \mathcal{E}^n_h(S_{h,\tau}^n(\f{v},\f{w})) - \mathcal{E}^{n-1}_h(\f{v},\f{w}) + \gamma \norm{\f{E}^{\f{v},\f{w}}}_{\f{L}^2(\Omega)}^2 + \frac{1}{2} \norm{\f{H}^{\f{v},\f{w}}-\f{w}}_{\f{L}^2_{\mu_{h,\infty}}(\Omega)}^2 \leq 0.
             \end{equation}
         In particular,   $\mathcal{E}^n_h(S_{h,\tau}^n(\f{v},\f{w}))\leq \mathcal{E}^{n-1}_h(\f{v},\f{w})$ holds.
     \end{lemma}
     \begin{proof}
         We insert $\f{v}_h = \f{E}^{\f{v},\f{w}}$ into the first equation of \eqref{eq:S_tau,hfinal} and use the second equation  to obtain
         \begin{align*}
        \left(\varepsilon(t_n)\frac{\f{E}^{\f{v},\f{w}} - \f{v}}{\tau},\f{E}^{\f{v},\f{w}}\right)_{\f{L}^2(\Omega)} & + \left(\mu_{h,\infty}\frac{\f{H}^{\f{v},\f{w}} - \f{w}}{\tau},\f{H}^{\f{v},\f{w}}\right)_{\f{L}^2(\Omega)}\\
             & +(\f{J}(\cdot,\f{E}^{\f{v},\f{w}},t_n),\f{E}^{\f{v},\f{w}})_{\f{L}^2(\Omega)}+(\f{b}(\cdot,\boldsymbol{\nu}\times \f{E}^{\f{v},\f{w}}),\boldsymbol{\nu}\times \f{E}^{\f{v},\f{w}})_{\f{L}^2(\partial \Omega)} = 0.
         \end{align*}
         Employing the monotonicity of $\f{b}$ and the identity
         $$A(a-b)\cdot a = \frac{1}{2}\left[ Aa\cdot a - Ab\cdot b + A(a-b)\cdot (a-b)\right],\quad a,b\in \mathbb{R}^3,A\in \mathbb{R}^{3\times 3}_{\text{sym}},$$
         with $A = \varepsilon(t_n),\mu_{h,\infty}$ yields
         \begin{align*}
             & \mathcal{E}_h^n(\f{E}^{\f{v},\f{w}},\f{H}^{\f{v},\f{w}})-\mathcal{E}_h^{n-1}(\f{v},\f{w}) +\frac{1}{2} (\varepsilon(t_n)(\f{E}^{\f{v},\f{w}} - \f{v}),\f{E}^{\f{v},\f{w}}-\f{v})_{\f{L}^2(\Omega)}+\frac{1}{2}(\mu_{h,\infty}(\f{H}^{\f{v},\f{w}} - \f{w}),\f{H}^{\f{v},\f{w}}-\f{w})_{\f{L}^2(\Omega)}\\
             & +\tau (\f{J}(\cdot,\f{E}^{\f{v},\f{w}},t_n),\f{E}^{\f{v},\f{w}})_{\f{L}^2(\Omega)} -  \frac{1}{2}((\varepsilon(t_n)-\varepsilon(t_{n-1}))\f{v},\f{v})_{\f{L}^2(\Omega)} \leq 0.
         \end{align*}
         The conclusion follows from Assumption \ref{ass:B5}.
     \end{proof}
     \begin{lemma}\label{lemma:S_tau,h-equiv}
         \normalfont Let $(\f{v},\f{w})\in \textbf{ND}_h \times \textbf{DG}_h$ and $n\geq m+1$. Then
         $$(\f{v},\f{w})\in \mathcal{N}_h\quad \Leftrightarrow\quad \mathcal{E}_h^n(S_{h,\tau}^n (\f{v},\f{w}))=\mathcal{E}_h^{n-1}(\f{v},\f{w})$$
         and either of these statements implies $S_{h,\tau}^n(\f{v},\f{w})=(\f{v},\f{w})$.
     \end{lemma}
     \begin{proof}
         First assume that $\mathcal{E}_h^n(S_{h,\tau}^n (\f{v},\f{w}))=\mathcal{E}_h^{n-1}(\f{v},\f{w})$, i.e.
         \begin{equation}\label{eq:energyeq}
             (\varepsilon(t_n)\f{E}^{\f{v},\f{w}},\f{E}^{\f{v},\f{w}})_{\f{L}^2(\Omega)} + (\mu_{h,\infty} \f{H}^{\f{v},\f{w}},\f{H}^{\f{v},\f{w}})_{\f{L}^2(\Omega)} = (\varepsilon(t_{n-1})\f{v},\f{v})_{\f{L}^2(\Omega)} + (\mu_{h,\infty} \f{w},\f{w})_{\f{L}^2(\Omega)}.
         \end{equation}
    It follows from \eqref{eq:energy-est} that $\f{E}^{\f{v},\f{w}} = 0$ and $\f{H}^{\f{v},\f{w}} = \f{w}$. Therefore, \eqref{eq:energyeq} yields $\f{v} = 0$, and hence $S_{h,\tau}^n(\f{v},\f{w})=(\f{E}^{\f{v},\f{w}},\f{H}^{\f{v},\f{w}}) = (\f{v},\f{w})$. Then, by \eqref{eq:S_tau,hfinal}-\eqref{eq:alles0} and $\f{v} = \f{E}^{\f{v},\f{w}}=0$, we obtain
    $$(\f{H}^{\f{v},\f{w}},\curl\,\f{v}_h)_{\f{L}^2(\Omega)} = 0\quad \forall \f{v}_h\in \textbf{ND}_h$$
and thus $(\f{v},\f{w})\in \mathcal{N}_h$. On the other hand, if $(0,\f{w})\in \mathcal{N}_h$, then it follows again from \eqref{eq:S_tau,hfinal}-\eqref{eq:alles0}  that $S_{h,\tau}^n(0,\f{w})=(0,\f{w})$ and  $\mathcal{E}_h^n(0,\f{w})=\mathcal{E}_h^{n-1}(0,\f{w})$.
     \end{proof}

     \begin{lemma}\label{lemma:strictdecay}
         \normalfont Let $(\f{v},\f{w})\in  \mathcal W_h $ and $n\geq m+1$.  Then  $S_{h,\tau}^n(\f{v},\f{w})\in \mathcal W_h$, and $\mathcal{E}_h^n(S_{h,\tau}^n (\f{v},\f{w}))< \mathcal{E}_h^{n-1}(\f{v},\f{w})$ holds as long as $(\f{v},\f{w})\neq (0,0).$
     \end{lemma}
     \begin{proof}
     Since $\f{w}\in \mu_{h,\infty}^{-1}\curl\,\textbf{ND}_h$,   the second equation of \eqref{eq:S_tau,hfinal} implies  $\f{H}^{\f{v},\f{w}}\in \mu_{h,\infty}^{-1}\curl\,\textbf{ND}_h$, and   thus $S_{h,\tau}^n(\f{v},\f{w})\in \mathcal W_h$. The second claim is a   consequence of the previous two lemmas and   $\mathcal{N}_h = \mathcal W_h^{\perp}$.
     \end{proof}
     \begin{lemma}
    \normalfont Let $n\geq m + 1$. Then $S_{h,\tau}^n:  \textbf{ND}_h\times \textbf{DG}_h\to \textbf{ND}_h\times \textbf{DG}_h$ is continuous and 
    \begin{eqnarray}\label{eq:homogen}
    S_{h,\tau}^n(\lambda(\f{v},\f{w})) &=& \lambda S_{h,\tau}^n(\f{v},\f{w})\quad \forall (\f{v},\f{w})\in \textbf{ND}_h \times \textbf{DG}_h,\lambda \geq 0 \\ \label{eq:additive}
       S_{h,\tau}^n(\f{v},\f{w}+\f{w}') &=& S_{h,\tau}^n(\f{v},\f{w}) + (0,\f{w}')\quad \forall (\f{v},\f{w})\in \mathcal W_h,(0,\f{w}')\in \mathcal{N}_h.
    \end{eqnarray}
\end{lemma}
\begin{proof}
    Let $(\f{v},\f{w})\in \textbf{ND}_h\times \textbf{DG}_h$ be fixed and suppose that 
    $$\{(\f{v}_k,\f{w}_k)\}_{k=1}^\infty\subset \textbf{ND}_h\times \textbf{DG}_h \quad \text{with}\quad (\f{v}_k,\f{w}_k)\to (\f{v},\f{w})\text{ in } \textbf{ND}_h\times \textbf{DG}_h$$
    is given. For all $k \in \mathbb N$, we  set $(\f{E}_k,\f{H}_k)\coloneqq S_{h,\tau}^n(\f{v}_k,\f{w}_k)$. Due to \eqref{eq:energy-est}, $\{(\f{E}_k,\f{H}_k)\}_{k=1}^\infty$ is bounded in $\f{L}^2(\Omega)^2$. Since $\textbf{ND}_h\times \textbf{DG}_h$ is finite-dimensional, we can extract a subsequence $\{(\f{E}_{k_l},\f{H}_{k_l})\}_{l=1}^\infty\subset \{(\f{E}_k,\f{H}_k)\}_{k=1}^\infty$ such that 
    $$(\f{E}_{k_l},\f{H}_{k_l}) \to (\f{E},\f{H})\text{ in } \f{V}\times \f{L}^2(\Omega)\quad \text{for some } (\f{E},\f{H})\in \textbf{ND}_h\times \textbf{DG}_h,$$
    which, due to \ref{ass:B2} and \ref{ass:B3}, implies $\f{J}(\cdot,\f{E}_{k_l},  t_n) \to \f{J}(\cdot,\f{E},  t_n)$ in $\f{L}^2(\Omega)$ and $ \f{b}(\cdot,\boldsymbol{\nu}\times\f{E}_{k_l})  \to \f{b}(\cdot,\boldsymbol{\nu}\times\f{E})$ in $\f{L}^2(\partial \Omega)$. In conclusion,   we    can pass to the limit in  \eqref{eq:S_tau,hfinal} (with $(\f{v},\f{w})$ replaced by $(\f{v}_{k_l},\f{w}_{k_l})$) to obtain   $(\f{E},\f{H})=S_{h,\tau}^n(\f{v},\f{w})$. This implies the continuity of $S^n_{h,\tau}$.

    The property \eqref{eq:homogen}  is immediately obtained by multiplying the first and second equations in \eqref{eq:S_tau,hfinal} with $\lambda \ge 0$ and invoking the positive homogeneity of $\f b$ and $\f J$ with respect to the second variable. To show that \eqref{eq:additive} is valid, let $(\f{v},\f{w})\in \mathcal W_h$,  $(0,\f{w}')\in \mathcal{N}_h$,    and $(\f{E}^{\f{v},\f{w}},\f{H}^{\f{v},\f{w}}):=S_{h,\tau}^n(\f{v},\f{w})$.   According to \eqref{def:Wh},  $(\f{w}',\curl \,\f{v}_h)_{\f{L}^2(\Omega)} =0$ holds for all $\f{v}_h \in \textbf{ND}_h$. Thus,  \eqref{eq:S_tau,hfinal} implies   
     \begin{equation*} 
\begin{dcases}
\int_\Omega \varepsilon(t_n) \frac{\f{E}^{\f{v},\f{w}} - \f{v}}{\tau}\cdot \f{v}_h - (\f{H}^{\f{v},\f{w}}+ \f{w}')\cdot \curl\,\f{v}_h +  \f{J}(\cdot,\f{E}^{\f{v},\f{w}}, t_n)\cdot \f{v}_h \,dx\\
+ \int_{\partial \Omega} \f{b}(\cdot,\boldsymbol{\nu}\times \f{E}^{\f{v},\f{w}})\cdot \boldsymbol{\nu}\times \f{v}_h\,dx = 0 \quad \forall \f{v}_h\in \textbf{ND}_h, \\
\mu_{h,\infty} \frac{\f{H}^{\f{v},\f{w}} +\f{w}' - (\f{w} + \f{w}')}{\tau} + \curl\,\f{E}^{\f{v},\f{w}}  = 0,
\end{dcases}
\end{equation*}
and so, again due to \eqref{eq:S_tau,hfinal}, it follows that  $ S_{h,\tau}^n(\f{v},\f{w}+\f{w}')=(\f{E}^{\f{v},\f{w}},\f{H}^{\f{v},\f{w}}+ \f{w}') = S_{h,\tau}^n(\f{v},\f{w}) + (0,\f{w}')$.   
\end{proof}
     \begin{proof}[Proof of Theorem \ref{theorem-exp-decay}]
         We divide the proof into three steps:

         Step 1. For every $n\geq m + 1$, we claim that 
        \begin{equation}\label{eq:q_n}
            q_n\coloneqq \max_{(\f{v},\f{w})\in \mathcal W_h\setminus \{(0,0)\}}\frac{\mathcal{E}_h^n(S_{h,\tau}^n(\f{v},\f{w}))}{\mathcal{E}_h^{n-1}(\f{v},\f{w})}\in (0,1).
        \end{equation}
        To prove this, note first that $B^{n-1}_h\coloneqq \{(\f{v},\f{w})\in \mathcal W_h : \mathcal{E}_h^{n-1}(\f{v},\f{w}) = 1\}$ is compact and $\mathcal{E}_h^n\circ S_{h,\tau}^n$ is continuous. Therefore, by the Weierstrass theorem,   
         $$
         \exists (\f{v}^*,\f{w}^*)\in B_h^{n-1}:  \quad \mathcal{E}_h^n(S_{h,\tau}^n(\f{v}^*,\f{w}^*)) = \max_{({\f{v}},{\f{w}})\in B_h^{n-1}}\,\mathcal{E}_h^n(S_{h,\tau}^n({\f{v}},{\f{w}})).$$
          Since $(\f{v}^*,\f{w}^*)\neq (0,0)$,  \eqref{eq:S_tau,hfinal} and \eqref{eq:alles0} yield that    $S_{h,\tau}^n(\f{v}^*,\f{w}^*)\neq (0,0)$ and hence $\mathcal{E}_h^n(S_{h,\tau}^n(\f{v}^*,\f{w}^*))>0$. Thus, along with Lemma \ref{lemma:strictdecay}, it follows that    $$
     0< \mathcal{E}_h^n(S_{h,\tau}^n(\f{v}^*,\f{w}^*))  < \mathcal{E}_h^{n-1}(\f{v}^*,\f{w}^*) = 1.  $$ Finally, 
         since 
         ${\mathcal{E}_h^{n-1}(\f{v},\f{w})} >0$ holds for all $(\f{v},\f{w}) \in \mathcal W_h \setminus \{(0,0)\},$
         the positive homogeneity of   $S_{h,\tau}^n$ (see  \eqref{eq:homogen}) and the quadratic homogeneity of $\mathcal{E}_h^n$  and $\mathcal{E}_h^{n-1}$ imply that $$ 
          \max_{(\f{v},\f{w})\in \mathcal W_h\setminus \{(0,0)\}}\frac{\mathcal{E}_h^n(S_{h,\tau}^n(\f{v},\f{w}))}{\mathcal{E}_h^{n-1}(\f{v},\f{w})} =\max_{({\f{v}},{\f{w}})\in B_h^{n-1}}\,\mathcal{E}_h^n(S_{h,\tau}^n({\f{v}},{\f{w}}))=\mathcal{E}_h^n(S_{h,\tau}^n(\f{v}^*,\f{w}^*))\in (0,1).
         $$
Step 2. We show that
         \begin{equation}\label{eq:sup q_n}
             q_{h,\tau}\coloneqq \sup_{n\geq m+1} q_n \in (0,1).
         \end{equation}
 Suppose, to the contrary, that $q_{h,\tau}=1$. Then there exist sequences $n_k \geq m +1, n_k\to \infty$, and $\{(\f{v}_k,\f{w}_k)\}_{k=1}^\infty\subset \mathcal W_h$ such that
\begin{equation}\label{eq:energyconv}
         \mathcal{E}_h^{n_k - 1}(\f{v}_k,\f{w}_k) = 1,\quad \mathcal{E}_h^{n_k}(S_{h,\tau}^{n_k}(\f{v}_k,\f{w}_k)) = q_{n_k}\quad\forall k\in \mathbb{N},\quad \lim_{k\to \infty} q_{n_k} = 1.    
         \end{equation}
Acorrding to Lemma \ref{lemma:strictdecay}, 
         \begin{equation}\label{eq:E_kH_k}
             (\f{E}_k,\f{H}_k)\coloneqq S_{h,\tau}^{n_k}(\f{v}_k,\f{w}_k) \in \mathcal W_h \quad \forall k \in \mathbb N.
         \end{equation}
         Since $\varepsilon, \mu_{h,\infty}$ are uniformly positive definite, \eqref{eq:energyconv} and \eqref{eq:E_kH_k} imply that $\{(\f{v}_k,\f{w}_k)\}_{k=1}^\infty,\{(\f{E}_k,\f{H}_k))\}_{k=1}^\infty\subset \mathcal W_h$ are bounded. As $\mathcal W_h$ is finite-dimensional, after selecting subsequences if necessary, it follows that
         \begin{equation}\label{eq:v_kw_k}
             (\f{v}_k,\f{w}_k)\to (\f{E}^-,\f{H}^-) \quad \text{in } \mathcal W_h\quad \textrm{and} \quad  (\f{E}_k,\f{H}_k) \to (\f{E}, \f{H})\quad \text{in } \mathcal W_h
         \end{equation}
         for some $(\f{E}^-,\f{H}^-),(\f{E},\f{H})\in \mathcal W_h$. Taking into account \eqref{eq:energy-est} and \eqref{eq:energyconv}, we have 
         \begin{equation}\label{eq:E+H+}
            \f{E} = 0,\quad \f{H}^-=\f{H}.
         \end{equation}
         Therefore, \ref{ass:B1} implies
        $\lim_{k\to\infty} (\varepsilon(t_{n_k})\f{E}_k,\f{E}_k)_{\f{L}^2(\Omega)}= 0,$ 
         which together with \eqref{eq:energyconv} leads to
         \begin{equation}\label{eq:H=1}
            \frac{1}{2}(\mu_{h,\infty}\f{H},\f{H})=1. 
         \end{equation}
         Furthermore, we have
         \begin{align*}
             \frac{\underline{\varepsilon}}{2}\norm{\f{v}_k}_{\f{L}^2(\Omega)}^2 \leq \frac{1}{2}(\varepsilon(t_{n_k - 1})\f{v}_k,\f{v}_k)& =\mathcal{E}_h^{n_k - 1}(\f{v}_k,\f{w}_k) - \mathcal{E}_h^{n_k}(\f{E}_k,\f{H}_k) + \frac{1}{2}\norm{\f{E}_k}_{\f{L}^2_{\varepsilon(t_{n_k})}(\Omega)}^2\\
             & - \frac{1}{2}\norm{\f{w}_k}_{\f{L}^2_{\mu_{h,\infty}}(\Omega)}^2+ \frac{1}{2}\norm{\f{H}_k}_{\f{L}^2_{\mu_{h,\infty}}(\Omega)}^2.
             \end{align*}
        Since the right-hand side of this inequality converges to 0 by \eqref{eq:energyconv}-\eqref{eq:H=1}, it follows that $\f{E}^- = 0$, and hence $(0,\f{H})= (\f{E}^-,\f{H}^-) \in \mathcal W_h$. On the other hand, according to \eqref{eq:E_kH_k} and \eqref{eq:S_tau,hfinal}, it holds for all $k\in \mathbb{N}$ that
        \begin{equation}\label{eq:defsolop}
        (\f{H}_k,\curl\,\f{v}_h)_{\f{L}^2(\Omega)} = \left(\varepsilon(t_{n_k})\frac{\f{E}_k - \f{v}_k}{\tau}+\f{J}(\cdot,\f{E}_k,t_{n_k}),\f{v}_h\right)_{\f{L}^2(\Omega)}+(\f{b}(\cdot,\boldsymbol{\nu}\times \f{E}_k),\boldsymbol{\nu}\times \f{v}_h)_{\f{L}^2(\partial \Omega)}.
        \end{equation}
        Due to \eqref{eq:v_kw_k} and the fact that $\norm{\cdot}_{\f{V}}$ is an equivalent norm on the finite-dimensional space $\textbf{ND}_h$, we have 
        $$\lim_{k\to \infty} \norm{\f{E}_k}_{\f{V}}=\lim_{k\to \infty} \norm{\f{v}_k}_{\f{L}^2(\Omega)}=0.$$
         Moreover, by \ref{ass:B2} and \eqref{con:b}, it holds for all $k \in \mathbb N$ that
        $$\norm{\f{J}(\cdot,\f{E}_k,t_{n_k})}_{\f{L}^2(\Omega)}\leq (L_{tot} + \norm{\partial_t \varepsilon}_{L^\infty(0,\infty;L^\infty(\Omega)^{3\times 3}})\norm{\f{E}_k}_{\f{L}^2(\Omega)},\quad \norm{\f{b}(\cdot,\boldsymbol{\nu}\times \f{E}_k)}_{\f{L}^2(\partial \Omega)} \leq \overline b \norm{\boldsymbol{\nu}\times \f{E}_k}_{\f{L}^2(\partial \Omega)}.$$
        Therefore, passing to the limit $k \to \infty$ in \eqref{eq:defsolop}   yields
         $$(\f{H},\curl\,\f{v}_h)_{\f{L}^2(\Omega)} = 0\quad \forall \f{v}_h\in \textbf{ND}_h,$$
         and consequently $(0,\f{H})\in \mathcal{N}_h$. As $(0,\f{H})\in \mathcal W_h$, it must be that $\f{H}=0$, which contradicts \eqref{eq:H=1}. 

Step 3. Setting
    $$  T^n_{h,\tau} \coloneqq S_{h,\tau}^n\circ S_{h,\tau}^{n-1}\circ \cdots\circ S_{h,\tau}^{m + 1} \quad \forall n\geq m + 1,$$ 
 it holds in view of \eqref{eq:deftheo} that 
\begin{equation} \label{eq:deftheoinproof}
(\f{E}_h^n, \f{H}_h^n) = T^n_{h,\tau}(\f{E}_h^{m}, \f{H}_h^{m}) \quad  \quad \forall n\geq m + 1.
\end{equation}
Now, we recall from the decomposition \eqref{eq:decomptheo} that 
\begin{equation} \label{eq:decomptheoinproof}
(\f{E}_h^{m},\f{H}_h^{m}) = (\f{E}_{h}^*, \f{H}_{h}^*)+(0,\f{H}_{h,\infty})\in \mathcal W_h\oplus \mathcal{N}_h.
\end{equation}
Since $(\f{E}_{h}^*, \f{H}_{h}^*)\in \mathcal W_h$, Lemma \ref{lemma:strictdecay} implies that $T^n_{h,\tau}(\f{E}_{h}^*, \f{H}_{h}^*)\in \mathcal W_h$ holds for all $n\geq m + 1$. For this reason,   we can iteratively apply \eqref{eq:q_n} and \eqref{eq:sup q_n} to obtain
    \begin{equation}\label{eq:exp-est}
    \begin{split} \mathcal{E}^n_h(T^n_{h,\tau}(\f{E}_{h}^*, \f{H}_{h}^*)) & \leq q_{h,\tau} \mathcal{E}_h^{n-1}(T^{n-1}_{h,\tau}(\f{E}_{h}^*, \f{H}_{h}^*))\leq \cdots \leq q_{h,\tau}^{n-m}\mathcal{E}_h^{m}(\f{E}_{h}^*, \f{H}_{h}^*)\\
    & = e^{-c_{h,\tau} (t_n-t_{m})} \mathcal{E}_h^{m}(\f{E}_{h}^*, \f{H}_{h}^*)\quad \forall n\geq m + 1
    \end{split} 
    \end{equation}
with $c_{h,\tau} := -\frac{1}{\tau}\ln(q_{h,\tau}) \in (0,\infty)$ where we have used $q_{h,\tau} \in (0,1)$. Finally,   \eqref{eq:deftheoinproof}, \eqref{eq:decomptheoinproof}, \eqref{eq:additive}, and Lemma \ref{lemma:S_tau,h-equiv} imply that    
$$
(\f{E}_h^n, \f{H}_h^n) = T^n_{h,\tau}(\f{E}_h^*,\f{H}_h^*) + (0,\f{H}_{h,\infty}) \quad  \forall n\geq m + 1.
$$
Thus, by  the orthogonality of $T^n_{h,\tau}(\f{E}_{h}^*, \f{H}_{h}^*)\in \mathcal W_h$ and $(0,\f{H}_{h,\infty})\in \mathcal{N}_h$, it follows  for all $n\geq m + 1$ that 
    \begin{equation}\label{eq:exp-est2}
    \mathcal{E}_h^n(\f{E}_h^n,\f{H}_h^n)  = \mathcal{E}_h^n(T^n_{h,\tau}(\f{E}_{h}^*, \f{H}_{h}^*)) + \mathcal{E}_h^{n}((0,\f{H}_{h,\infty}))
     = \mathcal{E}_h^n(T^n_{h,\tau}(\f{E}_{h}^*, \f{H}_{h}^*)) + \mathcal{E}_h^{m}(0,\f{H}_{h,\infty}),
    \end{equation}
     where  Lemma \ref{lemma:S_tau,h-equiv} was used for the last equality.  The claim \eqref{eq:expoest} now follows from \eqref{eq:exp-est} and \eqref{eq:exp-est2}.
\end{proof}
 \end{subsection}

To close this section, we present an example illustrating that the exponential decay result in general does not hold if \ref{ass:B5} is not satisfied. To this end, we first fix a function $\phi_h \in P_1^h\cap H_0^1(\Omega)$ with $\|\nabla \phi_h\|_{\f{L}^2(\Omega)} \neq 0$, where $P_1^h$ denotes the space of continuous piecewise linear functions associated with $\mathcal{T}_h$. Then,   we set 
$$(\f{E}_h^0,\f{H}_h^0)= (\nabla \phi_h,0)\in \mathcal W_h,\quad \f{f} =  0,  \quad \varepsilon(t)=(2-e^{-t})I, \quad  \mu = I,
$$
 where $I \in \mathbb{R}^{3\times3}$ denotes the identity matrix. Furthermore, $\f{b}$ can be arbitrarily chosen in accordance with Assumption \ref{assumption2}.

Now, if $\f{J}_{tot} = 0$, then by testing \ref{ass:B5} with $\f{v}_h = \hat{\f{v}}_h$, it is easy to see that there exists no $\gamma > 0$ such that \ref{ass:B5} holds. On the other hand, if $\f{J}_{tot} = 2\f{E}$, then, according to Remark \ref{rem:(B5)} (with  $c_J = 2$, $L_\epsilon=1$, $\underline \varepsilon =1$),  \ref{ass:B5}  is satisfied as long as $\gamma = \tau\leq \frac{1}{2}$. Furthermore, it follows from a simple induction argument in \eqref{eq:S_tau,hfinal} along with $\curl\,\nabla \phi_h = 0$ and $\boldsymbol \nu \times \nabla \phi_h=0$ that $\{(\f{E}_h^n,\f{H}_h^n)\}_n$ is given explicitly by $(\f{E}_h^n,\f{H}_h^n)=(\alpha_n \nabla {\phi}_h,0),n\in \mathbb{N}_0,$
where $\alpha_0 = 1$ and 
    \begin{equation}\label{eq:alpha_n formula}
        \alpha_n = \left(\prod_{j=1}^n r_j\right) \alpha_0,\quad r_j = \begin{cases}
    \frac{1}{1+\tau\frac{e^{-t_j}}{2-e^{-t_j}}}\quad \text{if } \f{J}_{tot} = 0,\\
    \frac{1}{1+\tau\frac{2+e^{-t_j}}{2-e^{-t_j}}}\quad \text{if } \f{J}_{tot} = 2\f{E}.
\end{cases}
    \end{equation}
In the case $\f{J}_{tot}=0$, it follows from
$$\sum_{j=1}^\infty \tau \frac{e^{-t_j}}{2-e^{-t_j}}<\infty$$ 
that $\{\alpha_n\}_{n=1}^\infty$ converges to some real number $\alpha_\infty >0$. Consequently, 
$$\mathcal{E}_h^n({\f{E}}_h^n,{\f{H}}_h^n) = \frac{1}{2}(2-e^{-t_n})\alpha_n^2 \norm{\nabla \phi_h}_{\f{L}^2(\Omega)^2}\to \alpha_\infty^2\norm{\nabla \phi_h}_{\f{L}^2(\Omega)}^2>0,$$
which shows that $\mathcal{E}_h^n$ is not decaying exponentially. 
In the second case $\f{J}_{tot} = 2\f{E}$, we have $r_n \leq (1+\tau)^{-1}$ and thus $\alpha_n \leq (1+\tau)^{-n}$, which leads to
$$\mathcal{E}_h^n({\f{E}}_h^n,{\f{H}}_h^n) = \frac{1}{2}(2-e^{-t_n})\alpha_n^2 \norm{\nabla \phi_h}_{\f{L}^2(\Omega)}^2 \leq C(1+\tau)^{-2n} \quad \forall n\in \mathbb{N},$$
for some   $C>0$ independent of $n$. Consequently, the electromagnetic energy is decaying exponentially.
 
 \section{Numerical results}\label{section:numericalresults}
     We present two numerical tests illustrating the convergence and exponential stability results from Sections \ref{section: convergence} and \ref{section:exponentialstability}. Here, the numerical computation of the finite element scheme \eqref{time-space-discrete} is realized based on the fixed-point iterative algorithm presented in the first part of the proof of Proposition \ref{theorem:existence}. This was carried out in Legacy FEniCS (2019.1.0) via an iterative nonlinear solver based on Newton's method. 
     For our numerical test below, we set  $\Omega = (-1,1)^3$, $\varepsilon(t)=(1-\frac{1}{4}\sin(t))I$, and
     $$\f{b}(\xi) =  
{\footnotesize \begin{pmatrix}
2.0 & 0.3 & 0.2\\
0.3 & 1.4 & 0.25\\
0.2 & 0.25 & 0.9
\end{pmatrix}}
\xi + \alpha_1 q(\xi), \quad \f{J}_{tot,\chi}(x,\xi,\eta,t) = \alpha_2 \xi + \alpha_3 q(\xi) + \chi(t)\frac{\eta}{1+\lvert\eta\rvert^2},$$
with a  Lipschitz continuous, monotone and homogeneous function $q: \mathbb R^3 \to \mathbb R^3, q(\xi) = \frac{1}{2}\nabla\norm{\xi}_{l^4}^2$.  Moreover, the constants $\alpha_1,\alpha_2,\alpha_3 \in \mathbb R$  and  the Lipschitz continuous function $\chi: \mathbb R \to \mathbb R$  are specified below.

\subsection{Convergence test} The goal of our first test is to quantify the convergence behaviour of $(\f{E}_{N,h},\f{H}_{N,h})$ using an exact solution $(\f{E},\f{H})$ of $\eqref{eq:maxwell intro full}$. To this end, we set $T=1.75, (\alpha_1,\alpha_2,\alpha_3)=(0,1,0), \chi\equiv 1, \mu(t)=(1+t)I,$ and
$\Psi(x_1,x_2,x_3)\coloneqq g(x_1)(1-x_2^2)^2(1-x_3^2)^2 v$
with
$$g(x_1)=(1+x_1)^2[1+(\lambda - 1)(x_1 - 1)],\quad \lambda = \frac{23-5\sqrt{2}}{20},\quad v=(0,1,\sqrt{2}-1).$$
One readily checks that 
$ \f{E}(x,t)\coloneqq \frac{\Psi(x)}{(1+t)^2}$ and $ \f{H}(x,t)\coloneqq \frac{\curl\,\Psi(x)}{(1+t)^2}$
form a solution to \eqref{eq:maxwell intro full} provided that
$$(\f{E}_0,\f{H}_0)\coloneqq (\Psi,\curl\,\Psi),\quad \f{f}(x,t)\coloneqq \partial_t (\varepsilon\f{E})(x,t) - \curl\,\f{H}(x,t)+\f{J}_{tot,\chi}(x,\f{E}(x,t),\f{H}(x,t),t).$$
Straightforward computations confirm that Assumption \ref{ass} is fulfilled for our first test.  Table 1 comprises our numerical results and displays the error quantities
\begin{equation*}
    e_{N,h}^n \coloneqq
    \left(
        \|\f{E}_h^n-\f{E}(t_n)\|_{L^2(\Omega)}^2
        +
        \|\f{H}_h^n-\f{H}(t_n)\|_{L^2(\Omega)}^2
    \right)^{1/2},\quad  e^{\max}_{N,h} \coloneqq \max_{1\leq n\leq N} e_{N,h}^n, \quad e^{\mathrm{time}}_{N,h} \coloneqq
    \left(\tau\sum_{n=1}^{N} (e^n_{N,h})^2\right)^{1/2}
\end{equation*}
 for varying and proportionally chosen values of $N$ and $h$ with $\mathrm{DoFs}=\mathrm{DoFs}(\mathbf{ND}_h)
        +\mathrm{DoFs}(\mathbf{DG}_h)$ denoting the total
        degrees of freedom. Furthermore, for consecutive numerical tests with corresponding parameters $N,h$ and $N',h'$ satisfying $N<N'$ and $h>h',$ we also have calculated the associated experimental order of convergence 
$\text{EOC}\coloneqq \frac{\log(e/e')}{\log(h/h')}.$
Here, either $e=e_{N,h}^\text{max}$ and $e'=e_{N',h'}^\text{max}$ or $e=e_{N,h}^\text{time}$ and $e'=e_{N',h'}^\text{time}$.    Table 1  confirms a convergence behavior of the numerical solution that is in agreement with Theorem \ref{theorem:convergence}. Moreover, as can be seen by Table 1, the experimental order of convergence approaches 1 as $h$ (resp. $N$) gets smaller (resp. larger). Together with Figure 1, this strongly indicates a first-order convergence rate.

\begin{table}[H]
    \centering
    \captionsetup{
        position=bottom,
        justification=centering,
        singlelinecheck=false
    }

        \begin{minipage}[t]{0.70\textwidth}
        \vspace{0pt}
        \caption{First test.}
        \label{tab:example1-convergence}
    \end{minipage}
    \hfill
    \begin{minipage}[t]{0.28\textwidth}
        \vspace{0pt}
        \caption{Second test.}
        \label{tab:example2-energy}
    \end{minipage}
    \begin{minipage}[t]{0.7\textwidth}
        \vspace{0pt}
        \centering
        \small
        \setlength{\tabcolsep}{4pt}

        \resizebox{0.9\linewidth}{!}{%
            \begin{tabular}{@{}rrrrrrr@{}}
                \toprule
                $N$ & $h$ & DoFs
                & $e_{\max}$ & $\mathrm{EOC}_{\max}$
                & $e_{\mathrm{time}}$ & $\mathrm{EOC}_{\mathrm{time}}$ \\
                \midrule
                280 & 0.34641 & 25\,930
                & 0.77215 & \noeoc
                & 0.45981 & \noeoc \\
                420 & 0.23094 & 86\,445
                & 0.52103 & 0.970
                & 0.30917 & 0.979 \\
                560 & 0.17321 & 203\,660
                & 0.39262 & 0.984
                & 0.23262 & 0.989 \\
                700 & 0.13856 & 396\,325
                & 0.31483 & 0.989
                & 0.18638 & 0.993 \\
                840 & 0.11547 & 683\,190
                & 0.26272 & 0.993
                & 0.15546 & 0.995 \\
                980 & 0.09897 & 1\,083\,005
                & 0.22538 & 0.994
                & 0.13332 & 0.996 \\
                \bottomrule
            \end{tabular}%
        }
    \end{minipage}
    \hfill
    \begin{minipage}[t]{0.28\textwidth}
        \vspace{0pt}
        \centering
        \small
        \setlength{\tabcolsep}{12pt}

        \resizebox{\linewidth}{!}{%
            \begin{tabular}{@{}rrr@{}}
                \toprule
                $n$ & $t_n$ & $\mathcal{E}_h^n({\f{E}}_h^n,{\f{H}}_h^n)$ \\
                \midrule
                 100 &  0.2 & $2.01714$               \\
                 500 &  1.0 & $5.61065\cdot 10^{-1}$  \\
                1000 &  2.0 & $3.22100\cdot 10^{-2}$  \\
                2000 &  4.0 & $4.09324\cdot 10^{-4}$  \\
                3000 &  6.0 & $1.42239\cdot 10^{-5}$  \\
                4000 &  8.0 & $5.73890\cdot 10^{-7}$  \\
                5000 & 10.0 & $5.81251\cdot 10^{-8}$  \\
                6000 & 12.0 & $7.77206\cdot 10^{-9}$  \\
                \bottomrule
            \end{tabular}%
        }
    \end{minipage}
\end{table}
\subsection{Exponential decay test} In the second example, we fix $\tau=0.002, h\approx 0.099,$ and set $\mu\equiv I, (\alpha_1,\alpha_2,\alpha_3)=(1,0,1),\f{E}_0=\f{H}_0=0$ and $\chi(t)=\max\{0, 1-5t\}$. We also apply the circular current 
    \begin{equation*}\f{f}:\Omega \times [0,\infty) \to \mathbb{R}^3, \quad 
    \f{f}((x,y,z),t)\coloneqq 100\max\{0,1-10t\}\cdot
   \left(0, \frac{-z}{(y^2 +z^2)^{1/2}}, \frac{y}{(y^2 +z^2)^{1/2}}\right)\chi_{\Omega_p}
     \end{equation*}
     to the cylindrical pipe coil
    $\Omega_p\coloneqq \{(x,y,z)\in \mathbb{R}^3:0\leq x\leq 1/2\text{ and } 3/10 \leq (y^2 +z^2)^{1/2}\leq 1/2\}.$
   In this setting,  Assumption \ref{assumption2} satisfied with $t_{m} = 0.2$.  To check   \ref{ass:B5}, we simply verify Remark \ref{rem:(B5)}   with $c_J = \tfrac{1}{\sqrt{3}}-\frac{1}{4},L_\varepsilon = \tfrac{1}{4}$. The computed energy values are summarized in Table 2 and Figure 2. In conclusion, our numerical results confirm an exponential decay behavior that   is consistent with Theorem \ref{theorem-exp-decay}.

\begin{figure}[H]
    \centering
    \begin{minipage}[t]{0.49\textwidth}
        \vspace{0pt}
        \centering
        \includegraphics[width=0.9\linewidth]{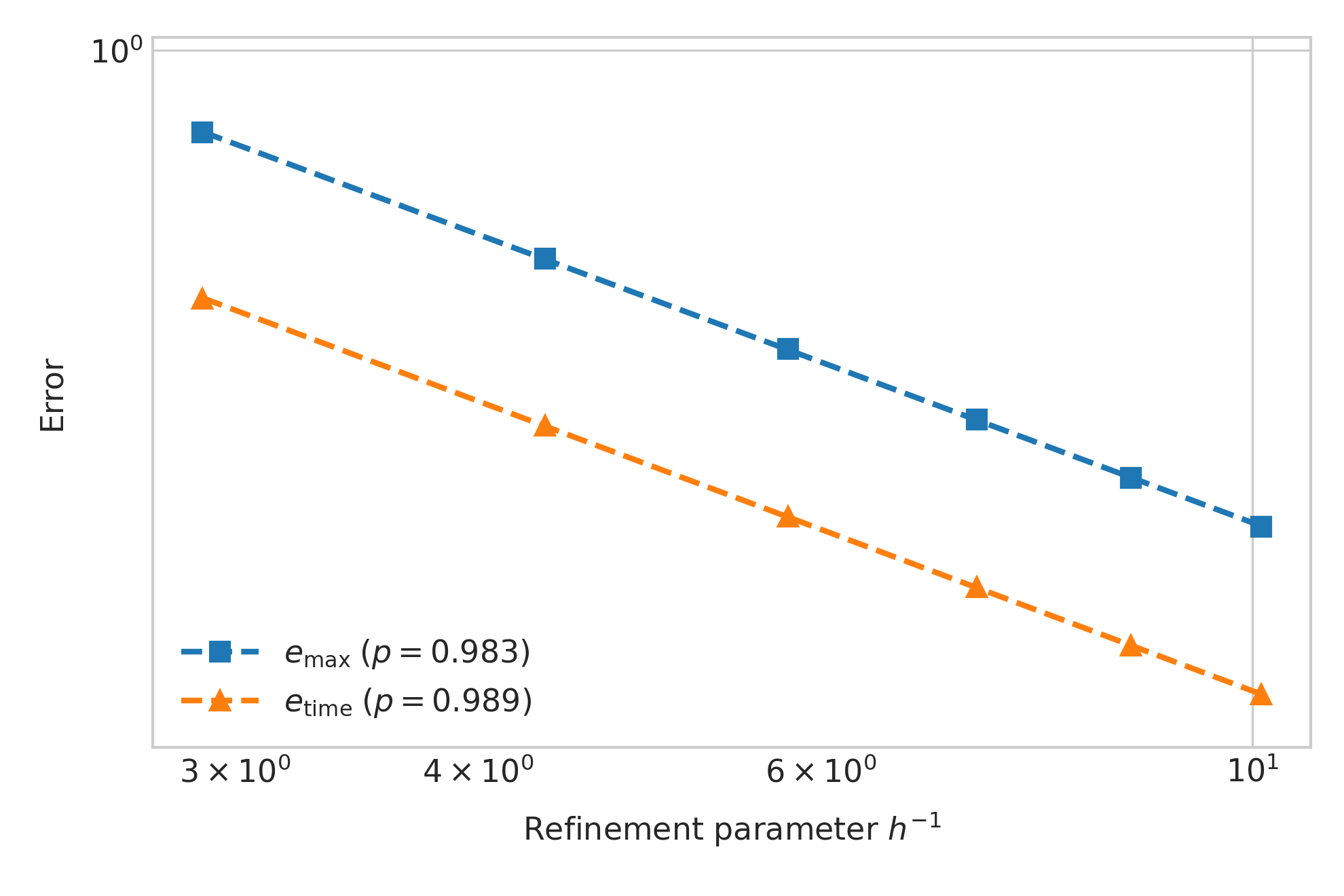}
        \caption{The plotted errors $e_\text{max},e_\text{time}$ and their corresponding $\log$-$\log$ least-squares fits with $p$ denoting the slope of the respective best-fitting line.}
        \label{fig:convergence}
    \end{minipage}
    \hfill
    \begin{minipage}[t]{0.49\textwidth}
        \vspace{0pt}
        \centering
        \includegraphics[width=0.9\linewidth]{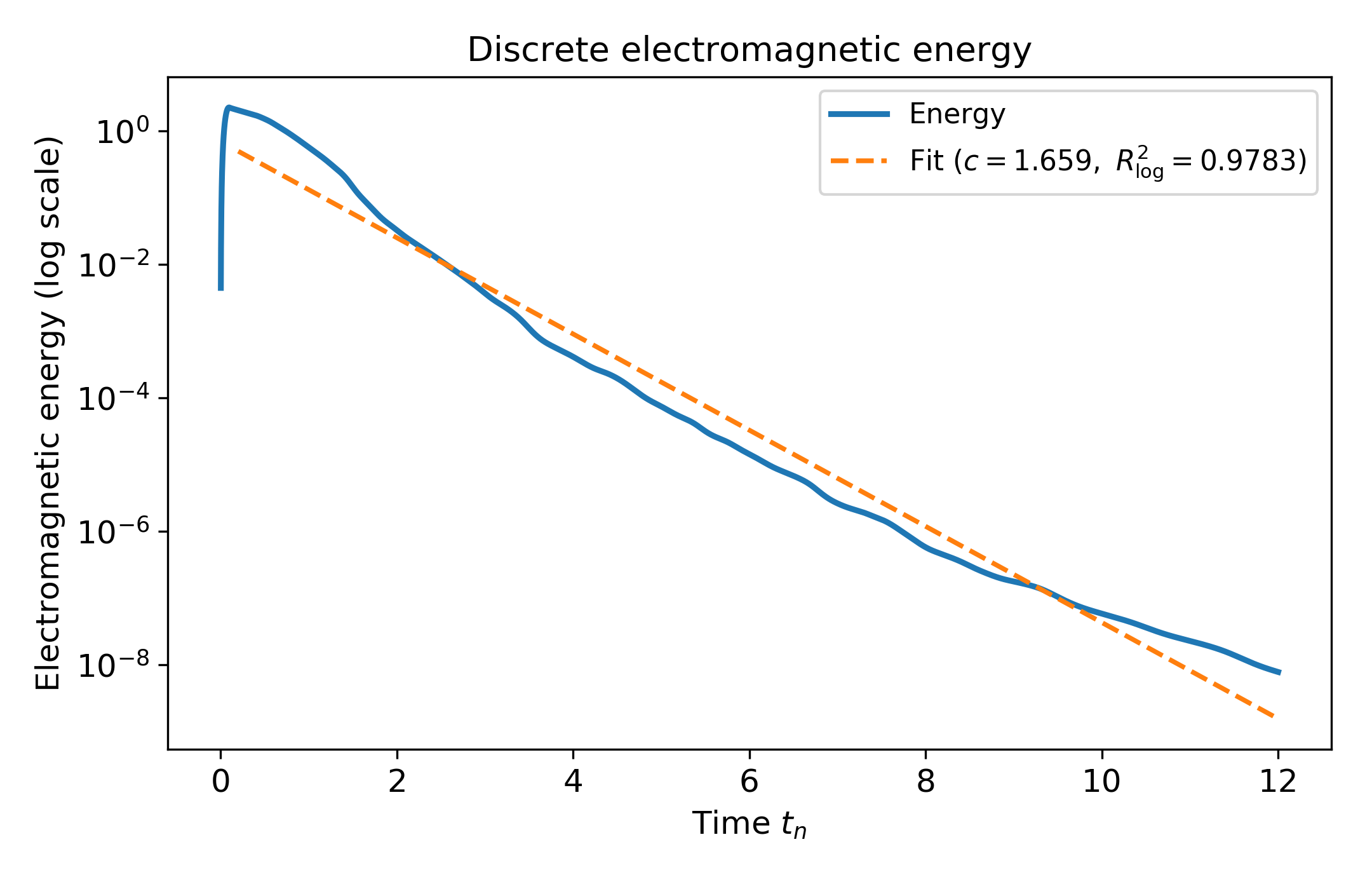}
        \caption{The discrete electromagnetic energy for $t_n\geq 0.002$ and its exponential least-squares fit on the time interval [0.2,12] on a logarithmic scale.}
        \label{fig:energy-decay}
    \end{minipage}
\end{figure}

{\footnotesize 
\printbibliography}

@article {MR3778338,
    AUTHOR = {Egger, H. and Kugler, T.},
     TITLE = {Damped wave systems on networks: exponential stability and
              uniform approximations},
   JOURNAL = {Numer. Math.},
  FJOURNAL = {Numerische Mathematik},
    VOLUME = {138},
      YEAR = {2018},
    NUMBER = {4},
     PAGES = {839--867},
}

@article{ned80,
 author = { N\'ed\'elec, J.~C.},
 title = {Mixed finite elements in $\mathbb{R}^3$},
 journal = {Numer. Math.},
 volume = {35},
 year = {1980},
 pages = {315--341},
 }

@article {MR2846772,
    AUTHOR = {Durand, S. and Slodi\v cka, M.},
     TITLE = {Fully discrete finite element method for {M}axwell's equations
              with nonlinear conductivity},
   JOURNAL = {IMA J. Numer. Anal.},
  FJOURNAL = {IMA Journal of Numerical Analysis},
    VOLUME = {31},
      YEAR = {2011},
    NUMBER = {4},
     PAGES = {1713--1733},
      ISSN = {0272-4979,1464-3642},
   MRCLASS = {65M60 (65M12 65M15 78A25 78M10)},
  MRNUMBER = {2846772},
       DOI = {10.1093/imanum/drr007},
       URL = {https://doi.org/10.1093/imanum/drr007},
}

@book {MR3013583,
    AUTHOR = {Li, J. and Huang, Y.},
     TITLE = {Time-domain finite element methods for {M}axwell's equations
              in metamaterials},
    SERIES = {Springer Series in Computational Mathematics},
    VOLUME = {43},
 PUBLISHER = {Springer, Heidelberg},
      YEAR = {2013},
     PAGES = {xii+302},
      ISBN = {978-3-642-33788-8; 978-3-642-33789-5},
   MRCLASS = {65M60 (35Q61)},
  MRNUMBER = {3013583},
MRREVIEWER = {Barbara\ Zubik-Kowal},
       DOI = {10.1007/978-3-642-33789-5},
       URL = {https://doi.org/10.1007/978-3-642-33789-5},
}

@article {MR2229840,
    AUTHOR = {Li, J. and Chen, Y.},
     TITLE = {Analysis of a time-domain finite element method for 3-{D}
              {M}axwell's equations in dispersive media},
   JOURNAL = {Comput. Methods Appl. Mech. Engrg.},
  FJOURNAL = {Computer Methods in Applied Mechanics and Engineering},
    VOLUME = {195},
      YEAR = {2006},
    NUMBER = {33-36},
     PAGES = {4220--4229},
      ISSN = {0045-7825,1879-2138},
   MRCLASS = {78M10 (65N30 78A48)},
  MRNUMBER = {2229840},
       DOI = {10.1016/j.cma.2005.08.002},
       URL = {https://doi.org/10.1016/j.cma.2005.08.002},
}

@article {MR4780410,
    AUTHOR = {Hensel, M. and Winckler, M. and Yousept, I.},
     TITLE = {Numerical solutions to hyperbolic {M}axwell quasi-variational
              inequalities in {B}ean-{K}im model for type-{II}
              superconductivity},
   JOURNAL = {ESAIM Math. Model. Numer. Anal.},
  FJOURNAL = {ESAIM. Mathematical Modelling and Numerical Analysis},
    VOLUME = {58},
      YEAR = {2024},
    NUMBER = {4},
     PAGES = {1385--1411},
      ISSN = {2822-7840,2804-7214},
   MRCLASS = {65M60 (65K15)},
  MRNUMBER = {4780410},
       DOI = {10.1051/m2an/2024034},
       URL = {https://doi.org/10.1051/m2an/2024034},
}

@book{monk,
    author = {Monk, Peter},
    title = {Finite Element Methods for Maxwell's Equations},
    publisher = {Oxford University Press},
    year = {2003},
    month = {04},
    isbn = {9780198508885},
    doi = {10.1093/acprof:oso/9780198508885.001.0001},
    url = {https://doi.org/10.1093/acprof:oso/9780198508885.001.0001},
}

@article {monk2,
    AUTHOR = {Monk, Peter},
     TITLE = {A mixed method for approximating {M}axwell's equations},
   JOURNAL = {SIAM J. Numer. Anal.},
  FJOURNAL = {SIAM Journal on Numerical Analysis},
    VOLUME = {28},
      YEAR = {1991},
    NUMBER = {6},
     PAGES = {1610--1634},
      ISSN = {0036-1429},
   MRCLASS = {65N30 (35Q60)},
  MRNUMBER = {1135758},
MRREVIEWER = {P.\ Rochus},
       DOI = {10.1137/0728081},
       URL = {https://doi.org/10.1137/0728081},
}

@article {monk3,
    AUTHOR = {Monk, Peter},
     TITLE = {A comparison of three mixed methods for the time-dependent
              {M}axwell's equations},
   JOURNAL = {SIAM J. Sci. Statist. Comput.},
  FJOURNAL = {Society for Industrial and Applied Mathematics. Journal on
              Scientific and Statistical Computing},
    VOLUME = {13},
      YEAR = {1992},
    NUMBER = {5},
     PAGES = {1097--1122},
      ISSN = {0196-5204},
   MRCLASS = {65N30 (65M60 78-08)},
  MRNUMBER = {1177800},
MRREVIEWER = {Teodor\ Potra},
       DOI = {10.1137/0913064},
       URL = {https://doi.org/10.1137/0913064},
}

@article {monk5,
    AUTHOR = {Monk, Peter},
     TITLE = {An analysis of {N}\'ed\'elec's method for the spatial
              discretization of {M}axwell's equations},
   JOURNAL = {J. Comput. Appl. Math.},
  FJOURNAL = {Journal of Computational and Applied Mathematics},
    VOLUME = {47},
      YEAR = {1993},
    NUMBER = {1},
     PAGES = {101--121},
      ISSN = {0377-0427,1879-1778},
   MRCLASS = {65M60 (78-08)},
  MRNUMBER = {1226366},
MRREVIEWER = {P.\ Rochus},
       DOI = {10.1016/0377-0427(93)90093-Q},
       URL = {https://doi.org/10.1016/0377-0427(93)90093-Q},
}

@book{showalter,
  author    = {Showalter, R. E.},
  title     = {Monotone Operators in Banach Space and Nonlinear Partial Differential Equations},
  series    = {Mathematical Surveys and Monographs},
  volume    = {49},
  publisher = {American Mathematical Society, Providence, RI},
  year      = {1997},
  pages     = {xiv+278},
  doi       = {10.1090/surv/049}
}

@incollection {nicaise1,
    AUTHOR = {Eller, M. and Lagnese, J. E. and Nicaise, S.},
     TITLE = {Decay rates for solutions of a {M}axwell system with nonlinear
              boundary damping},
      NOTE = {Special issue in memory of Jacques-Louis Lions},
   JOURNAL = {Comput. Appl. Math.},
  FJOURNAL = {Computational \& Applied Mathematics},
    VOLUME = {21},
      YEAR = {2002},
    NUMBER = {1},
     PAGES = {135--165},
      ISSN = {1807-0302},
   MRCLASS = {78A25 (35L50 35Q60 93D15)},
  MRNUMBER = {2009950},
}

@article {nicaise2,
    AUTHOR = {Eller, Matthias and Lagnese, John E. and Nicaise, Serge},
     TITLE = {Stabilization of heterogeneous {M}axwell's equations by linear
              or nonlinear boundary feedback},
   JOURNAL = {Electron. J. Differential Equations},
  FJOURNAL = {Electronic Journal of Differential Equations},
      YEAR = {2002},
     PAGES = {No. 21, 26},
      ISSN = {1072-6691},
   MRCLASS = {93D15 (35B35 35L50 35Q60 78A25 93B05 93C20)},
  MRNUMBER = {1884990},
MRREVIEWER = {Enrique\ Zuazua},
}

@article {Nicaise3,
    AUTHOR = {Nicaise, Serge and Pignotti, Cristina},
     TITLE = {Boundary stabilization of {M}axwell's equations with
              space-time variable coefficients},
   JOURNAL = {ESAIM Control Optim. Calc. Var.},
  FJOURNAL = {ESAIM. Control, Optimisation and Calculus of Variations},
    VOLUME = {9},
      YEAR = {2003},
     PAGES = {563--578},
      ISSN = {1292-8119,1262-3377},
   MRCLASS = {93D15 (35Q60 78A02 93C20)},
  MRNUMBER = {1998715},
MRREVIEWER = {Bing-Yu\ Zhang},
       DOI = {10.1051/cocv:2003027},
       URL = {https://doi.org/10.1051/cocv:2003027},
}

@book{roubicek,
  author    = {Roubíček, Tomáš},
  title     = {Nonlinear Partial Differential Equations with Applications},
  series    = {International Series of Numerical Mathematics},
  volume    = {153},
  publisher = {Birkhäuser, Basel},
  year      = {2005},
  pages     = {xviii+405},
  isbn      = {3-7643-7293-1},
  doi       = {10.1007/3-7643-7397-0}
}

@article{slodicka1,
title = {Fully discrete finite element scheme for {M}axwell's equations with non-linear boundary condition},
journal = {Journal of Mathematical Analysis and Applications},
volume = {375},
number = {1},
pages = {230-244},
year = {2011},
issn = {0022-247X},
doi = {https://doi.org/10.1016/j.jmaa.2010.09.016},
url = {https://www.sciencedirect.com/science/article/pii/S0022247X10007493},
author = {Marián Slodička and Stephane Durand}
}

@book {Guermond,
    AUTHOR = {Ern, Alexandre and Guermond, Jean-Luc},
     TITLE = {Finite elements {I}---{A}pproximation and interpolation},
    SERIES = {Texts in Applied Mathematics},
    VOLUME = {72},
 PUBLISHER = {Springer, Cham},
      YEAR = {[2021] \copyright 2021},
     PAGES = {xii+325},
      ISBN = {978-3-030-56340-0; 978-3-030-56341-7},
   MRCLASS = {65-01},
  MRNUMBER = {4242224},
       DOI = {10.1007/978-3-030-56341-7},
       URL = {https://doi.org/10.1007/978-3-030-56341-7},
}

@book {Vainberg,
    AUTHOR = {Va\u inberg, M. M.},
     TITLE = {Variational method and method of monotone operators in the
              theory of nonlinear equations},
    EDITOR = {Louvish, D.},
      NOTE = {Translated from the Russian by A. Libin},
 PUBLISHER = {Halsted Press [John Wiley \& Sons], New York-Toronto; Israel
              Program for Scientific Translations, Jerusalem-London},
      YEAR = {1973},
     PAGES = {xi+356},
   MRCLASS = {47H15},
  MRNUMBER = {467428},
MRREVIEWER = {Mieczyslaw\ Altman},
}

@article{Kim,
  title = {Magnetization and Critical Supercurrents},
  author = {Kim, Y. B. and Hempstead, C. F. and Strnad, A. R.},
  journal = {Phys. Rev.},
  volume = {129},
  issue = {2},
  pages = {528--535},
  numpages = {0},
  year = {1963},
  month = {Jan},
  publisher = {American Physical Society},
  doi = {10.1103/PhysRev.129.528},
  url = {https://link.aps.org/doi/10.1103/PhysRev.129.528}
}

@article {Kikuchi,
    AUTHOR = {Kikuchi, Fumio},
     TITLE = {On a discrete compactness property for the {N}\'ed\'elec
              finite elements},
   JOURNAL = {J. Fac. Sci. Univ. Tokyo Sect. IA Math.},
  FJOURNAL = {Journal of the Faculty of Science. University of Tokyo.
              Section IA. Mathematics},
    VOLUME = {36},
      YEAR = {1989},
    NUMBER = {3},
     PAGES = {479--490},
      ISSN = {0040-8980},
   MRCLASS = {65N30},
  MRNUMBER = {1039483},
MRREVIEWER = {Reinhard\ Scholz},
}

@article {egger3,
    AUTHOR = {Egger, H. and Kurz, S. and L\"oscher, R.},
     TITLE = {On the exponential stability of uniformly damped wave
              equations and their structure-preserving discretization},
   JOURNAL = {Results Appl. Math.},
  FJOURNAL = {Results in Applied Mathematics},
    VOLUME = {24},
      YEAR = {2024},
     PAGES = {Paper No. 100502, 12},
      ISSN = {2590-0374},
   MRCLASS = {65M12 (35L05 35Q61 65M60)},
  MRNUMBER = {4806503},
       DOI = {10.1016/j.rinam.2024.100502},
       URL = {https://doi.org/10.1016/j.rinam.2024.100502},
}

@article {Zuazua,
    AUTHOR = {Ervedoza, Sylvain and Zuazua, Enrique},
     TITLE = {Uniformly exponentially stable approximations for a class of
              damped systems},
   JOURNAL = {J. Math. Pures Appl. (9)},
  FJOURNAL = {Journal de Math\'ematiques Pures et Appliqu\'ees. Neuvi\`eme
              S\'erie},
    VOLUME = {91},
      YEAR = {2009},
    NUMBER = {1},
     PAGES = {20--48},
      ISSN = {0021-7824},
   MRCLASS = {65L05 (93C55 93D05)},
  MRNUMBER = {2487899},
MRREVIEWER = {Richard\ A.\ Al\`o},
       DOI = {10.1016/j.matpur.2008.09.002},
       URL = {https://doi.org/10.1016/j.matpur.2008.09.002},
}

\appendix
\renewcommand{\thesection}{\Alph{section}} 
\setcounter{section}{0}
\section{Appendix A}   \label{appen a}
\textit{Uniqueness of a solution to \eqref{eq:maxwell intro full}.} Let $(\f{E}_1,\f{H}_1), (\f{E}_2,\f{H}_2)$ denote two solutions of \eqref{eq:maxwell intro full} and set $(\f{E},\f{H})\coloneqq (\f{E}_1 - \f{E}_2, \f{H}_1 - \f{H}_2)$. For $t\in [0,T]$, we calculate
\begin{align*}
\frac{1}{2}
\norm{(\f{E},\f{H})(t)}_{\f{L}^2_{\varepsilon(t)}(\Omega)\times \f{L}^2_{\mu(t)}(\Omega)}^2
&=
\frac{1}{2}\int_0^t
\frac{d}{ds}
\norm{(\f{E},\f{H})(s)}_{\f{L}^2_{\varepsilon(s)}(\Omega)\times \f{L}^2_{\mu(s)}(\Omega)}^2
\,ds
\\
&\underbrace{=}_{\eqref{eq:maxwell intro full}, \eqref{eq:partial}}
-\frac{1}{2}\int_0^t
(\partial_t\varepsilon(s)\f{E}(s),\f{E}(s))_{\f{L}^2(\Omega)}\,ds
-\frac{1}{2}\int_0^t
(\partial_t\mu(s)\f{H}(s),\f{H}(s))_{\f{L}^2(\Omega)}\,ds
\\
&\quad
-\int_0^t
\big(
\f{J}_{tot}(\cdot,\f{E}_1(s),\f{H}_1(s),s)
-\f{J}_{ tot}(\cdot,\f{E}_2(s),\f{H}_2(s),s),
\f{E}(s)
\big)_{\f{L}^2(\Omega)}
\,ds
\\
&\quad
-\int_0^t
\big(
\f{b}(\cdot, 
\boldsymbol\nu\times\f{E}_1(s))
-\f{b}(\cdot,\boldsymbol\nu\times\f{E}_2(s)),
\boldsymbol\nu\times\f{E}(s)
\big)_{\f{L}^2(\partial\Omega)}
\,ds
\\
&\underbrace{\leq}_{\ref{ass:A1},\ref{ass:A4},\ref{ass:A5}}
c\int_0^t
\norm{(\f{E},\f{H})(s)}_{\f{L}^2(\Omega)\times \f{L}^2(\Omega)}^2
\,ds
\end{align*}
with a constant $c>0$, independent of $\f{E}$ and  $\f{H}$. By the uniform positive definiteness of $\varepsilon$ and $\mu$, Grönwall's
inequality therefore yields $\f{E}=\f{H}=0$.

\section{Appendix B}   \label{appen b}
\begin{proof}[Proof of Lemma \ref{lemma:V}]
Let $r_h$ denote the interpolation operator associated with $\textbf{ND}_h$, which maps every sufficiently smooth $\f{u}\in \f{H}(\curl)$ to the unique element of $\textbf{ND}_h$ with the same global degrees of freedom as $\f{u}$. It is well known that   $r_h$ is well-defined on $\f{H}^s(\curl)$   for any $1/2 < s \leq 1$  (cf. \citegen[Theorem]{monk}{3.9}) and that
\begin{align}\label{eq:r_h-est}
   \norm{r_h\f{u} - \f{u}}_{\f{H}(\curl)}\leq C_1h^s \norm{\f{u}}_{\f{H}^s(\curl)}\quad \forall \f{u}\in \f{H}^s(\curl) \quad \forall h>0 
\end{align}
for some   constant $C_1>0$ not depending on $h$ and $\f{u}$ (cf. \citegen[Theorem]{monk}{5.41}). 

Let  $ \f{u} \in \f{H}^s(\curl)$. As $s > 1/2$, the trace theorem \citegen[Theorem]{monk}{3.9} implies  that the trace of $\f{u}$ is well-defined in   $\f{L}^2(\partial \Omega)$. Furthermore, since $|\boldsymbol{\nu}(s)|=1$ holds for a.e. $s \in \partial \Omega$, it follows that   
    $$\norm{\boldsymbol{\nu} \times (r_h\f{u} - \f{u})}_{\f{L}^2(\partial \Omega)}\leq \norm{r_h\f{u} - \f{u}}_{\f{L}^2(\partial \Omega)}.$$
    As shown in the proof of  \citegen[Lemma]{monk}{5.52}, there exists a constant $C_2>0$, independent of $\f{u}$ and $h$, such that 
    \begin{equation} \label{eq:r_h-est2}
     \norm{r_h\f{u} - \f{u}}_{\f{L}^2(\partial \Omega)}\le C_2 h^{s-\frac{1}{2}}\norm{\f{u}}_{\f{H}^s(\curl)}  \quad  \forall h >0. 
  \end{equation}
    Now, by the definition of the Hilbert projection operator $\Pi_h:\f{V}\to \textbf{ND}_h$, we have
   \begin{align*}
\|\Pi_h \f{u} - \f{u}\|_{\f{V}}
  &= \min_{\f{u}_h \in \textbf{ND}_h} \| \f{u}_h - \f{u} \|_{\f{V}} \le \| r_h\f{u} - \f{u} \|_{\f{V}} \\
  &= \bigl(\| r_h\f{u} - \f{u} \|_{\f{H}(\curl)}^2
      + \| \boldsymbol{\nu} \times(r_h\f{u} - \f{u})\|_{\f{L}^2(\partial\Omega)}^2\bigr)^{1/2} \\
  &\overset{\mathclap{\eqref{eq:r_h-est},\eqref{eq:r_h-est2}}}{\le}
    (C_1 h^s + C_2 h^{s-\tfrac12}) \|\f{u}\|_{\f{H}^s(\curl)},
\end{align*}
from which the desired estimate \eqref{estimate lemma} follows. The second claim of the lemma is a direct consequence of \eqref{estimate lemma} and the density of $\f{C}^\infty(\overline{\Omega})$ in $\f{V}$.
\end{proof}

\section{Appendix C} \label{appen c}
In order to prove Proposition \ref{theorem:stability}, we shall first need the following technical result:
\begin{lemma}\label{lemma:a}  \normalfont 
    Let \ref{ass:A1} hold and let $N \in \mathbb N$, $1\leq m\leq N$, and $a_0,a_1,\dots,a_m\in \f{L}^2(\Omega)$. Then
        \begin{align*}
\norm{a_m}^2_{\f{L}^2_{\varepsilon(t_m)}(\Omega)} & \leq \left(\frac{LT}{\underline{\varepsilon}N} + 1\right)\norm{a_0}^2_{\f{L}^2_{\varepsilon(0)}(\Omega)} + \sum_{k=1}^{m-1}\frac{LT}{\underline{\varepsilon}N} \norm{a_{k}}^2_{\f{L}^2_{\varepsilon(t_{k})}(\Omega)} + 2\sum_{k=1}^m (a_k - a_{k-1},a_k)_{\f{L}^2_{\varepsilon(t_k)}(\Omega)}.
        \end{align*}
  The same estimate holds with $\underline{\varepsilon},\varepsilon(t_k)$ and $L$ replaced by $\underline \mu,\mu_h^k$ and $\norm{\partial_t \mu}_{L^\infty(0,T;L^\infty(\Omega)^{3\times 3}}$, respectively.
\end{lemma}
\begin{proof}[Proof of Lemma \ref{lemma:a}]We calculate
    \begin{equation*}
    \begin{split}
        \sum_{k=1}^m (a_k - a_{k-1},a_k)_{\f{L}^2_{\varepsilon(t_k)}} & = \sum_{k=1}^m \norm{a_k - a_{k-1}}^2_{\f{L}_{\varepsilon(t_k)}^2(\Omega)} +\sum_{k=1}^m (a_k - a_{k-1},a_{k-1})_{\f{L}_{\varepsilon(t_k)}^2(\Omega)}\\
        & = \sum_{k=1}^m \norm{a_k - a_{k-1}}^2_{\f{L}_{\varepsilon(t_k)}^2(\Omega)} -\sum_{k=1}^m (a_k - a_{k-1},a_k)_{\f{L}_{\varepsilon(t_k)}^2(\Omega)}\\
        & + \sum_{k=1}^m \left(\norm{a_k}_{\f{L}_{\varepsilon(t_k)}^2(\Omega)}^2 - \norm{a_{k-1}}_{\f{L}^2_{\varepsilon(t_k)}(\Omega)}^2\right).
    \end{split}
\end{equation*}
The last sum on the right-hand side equals
$$\norm{a_m}^2_{\f{L}^2_{\varepsilon(t_m)}(\Omega)} - \norm{a_0}^2_{\f{L}^2_{\varepsilon(0)}(\Omega)} + \sum_{k=1}^m \left(\norm{a_{k-1}}^2_{\f{L}^2_{\varepsilon(t_{k-1})}(\Omega)}-\norm{a_{k-1}}^2_{\f{L}^2_{\varepsilon(t_{k})}(\Omega)}\right).$$
Using \hyperref[ass:A1]{(A1)}, we estimate 
\begin{equation*}
    \begin{split}
        \norm{a_{k-1}}^2_{\f{L}^2_{\varepsilon(t_{k})}(\Omega)} - \norm{a_{k-1}}^2_{\f{L}^2_{\varepsilon(t_{k-1})}(\Omega)}& = \Bigl(\bigl(\varepsilon(t_k) - \varepsilon(t_{k-1})\bigr)a_{k-1}, a_{k-1}\Bigr)_{\f{L}^2(\Omega)}\\
        & \leq \norm{\varepsilon(t_k) - \varepsilon(t_{k-1})}_{\f{L}^\infty(\Omega)^{3\times 3}}\norm{a_{k-1}}_{\f{L}^2(\Omega)}^2\\
        & \leq \frac{LT}{\underline{\varepsilon}N} \norm{a_{k-1}}_{\f{L}^2_{\varepsilon(t_{k-1})}(\Omega)}^2.
    \end{split}
\end{equation*}
After an index shift, the claim follows.
\end{proof}
\begin{proof}[Proof of Proposition \ref{theorem:stability}]
Let $h>0$ and $N>M(h).$ Moreover,
let $m\in \{1,\dots,N\}$ and $k\in \{1,\dots,m\}$ be fixed. By subtracting the first equality in (P$_{N,h}$) for $k$ and $k-1$ from one another and testing with $\f{w} = \delta \f{E}^k_h\in \textbf{ND}_h$, it follows that
\begin{align*}
&(\varepsilon(t_k)\delta \f{E}^k_h-\varepsilon(t_{k-1})\delta \f{E}^{k-1}_h,\delta\f{E}^k_h)_{\f{L}^2(\Omega)}
+(\f{J}(\cdot,\f{E}^k_h,\f{H}_h^k,t_k)-\f{J}(\cdot,\f{E}^{k-1}_h,\f{H}^{k-1}_h,t_{k-1}),\delta\f{E}^k_h)_{\f{L}^2(\Omega)}\\
&\quad-(\delta \f{H}^k_h,\curl(\f{E}^k_h-\f{E}^{k-1}_h))_{\f{L}^2(\Omega)}
+(\f{b}(\cdot,\boldsymbol{\nu}\times \f{E}^k_h)-\f{b}(\cdot,\boldsymbol{\nu}\times \f{E}^{k-1}_h),
\boldsymbol{\nu}\times \delta \f{E}^k_h)_{\f{L}^2(\partial \Omega)}\\
&=(\f{f}^k_h-\f{f}^{k-1}_h,\delta\f{E}^k_h)_{\f{L}^2(\Omega)}.
\end{align*}
Similarly, subtracting the second equality in (P$_{N,h}$) for $k$ and $k-1$ and taking the $\f{L}^2(\Omega)$-scalar product with $\delta\f{H}_h^k$ gives
\begin{align*}
&(\mu_h^k\delta\f{H}_h^k-\mu_h^{k-1}\delta\f{H}_h^{k-1},\delta\f{H}_h^k)_{\f{L}^2(\Omega)}
+(\mu_{t,h}^k\f{H}_h^k-\mu_{t,h}^{k-1}\f{H}_h^{k-1},\delta\f{H}_h^k)_{\f{L}^2(\Omega)}\\
&\quad+(\curl(\f{E}_h^k-\f{E}_h^{k-1}),\delta\f{H}_h^k)_{\f{L}^2(\Omega)}=0.
\end{align*}
We rewrite the first terms on the left-hand sides  as
$$
(\delta\f{E}^k_h-\delta\f{E}^{k-1}_h,\delta\f{E}^k_h)_{\f{L}^2_{\varepsilon(t_k)}(\Omega)}
+\left((\varepsilon(t_k)-\varepsilon(t_{k-1}))\delta\f{E}^{k-1}_h,\delta\f{E}^k_h\right)_{\f{L}^2(\Omega)}
$$
and
$$
(\delta\f{H}^k_h-\delta\f{H}^{k-1}_h,\delta\f{H}^k_h)_{\f{L}^2_{\mu_h^k}(\Omega)}
+\left((\mu_h^k-\mu_h^{k-1})\delta\f{H}^{k-1}_h,\delta\f{H}^k_h\right)_{\f{L}^2(\Omega)},
$$
respectively. By adding the two equalities, employing the monotonicity of $\f{b}$ and summing over $k$, we obtain
\begin{align*}
&\sum_{k=1}^m
(\delta\f{E}^k_h-\delta\f{E}^{k-1}_h,\delta\f{E}^k_h)_{\f{L}^2_{\varepsilon(t_k)}(\Omega)}
+\sum_{k=1}^m
(\delta\f{H}^k_h-\delta\f{H}^{k-1}_h,\delta\f{H}^k_h)_{\f{L}^2_{\mu_h^k}(\Omega)}\\
&\leq
\sum_{k=1}^m(\f{f}_h^k-\f{f}_h^{k-1},\delta\f{E}_h^k)_{\f{L}^2(\Omega)}
-\sum_{k=1}^m(\f{J}(\cdot,\f{E}_h^k,\f{H}_h^k,t_k)
-\f{J}(\cdot,\f{E}_h^{k-1},\f{H}_h^{k-1},t_{k-1}),\delta\f{E}_h^k)_{\f{L}^2(\Omega)}\\
&\quad-\sum_{k=1}^m
\left((\varepsilon(t_k)-\varepsilon(t_{k-1}))
\delta\f{E}_h^{k-1},\delta\f{E}_h^k\right)_{\f{L}^2(\Omega)}\\
&\quad-\sum_{k=1}^m
\left((\mu_h^k-\mu_h^{k-1})
\delta\f{H}_h^{k-1},\delta\f{H}_h^k\right)_{\f{L}^2(\Omega)}
-\sum_{k=1}^m
(\mu_{t,h}^k\f{H}_h^k-\mu_{t,h}^{k-1}\f{H}_h^{k-1},
\delta\f{H}_h^k)_{\f{L}^2(\Omega)}.
\end{align*}
Applying Lemma \ref{lemma:a} to the left side yields
\begin{align*}
&\norm{\delta\f{E}^m_h}_{\f{L}^2_{\varepsilon(t_m)}(\Omega)}^2
+\norm{\delta\f{H}^m_h}_{\f{L}^2_{\mu_h^m}(\Omega)}^2\\
&\leq
\left[\frac{LT}{\underline{\varepsilon}N}+1\right]
\norm{\delta\f{E}^0_h}_{\f{L}^2_{\varepsilon(0)}(\Omega)}^2
+\left[
\frac{T}{\underline{\mu}N}
\norm{\partial_t\mu}_{L^\infty(0,T;L^\infty(\Omega)^{3\times3})}
+1\right]
\norm{\delta\f{H}^0_h}_{\f{L}^2_{\mu_h^0}(\Omega)}^2\\
&\quad+\sum_{k=1}^{m-1}\frac{LT}{\underline{\varepsilon}N}
\norm{\delta\f{E}^k_h}_{\f{L}^2_{\varepsilon(t_k)}(\Omega)}^2
+\sum_{k=1}^{m-1}
\frac{T}{\underline{\mu}N}
\norm{\partial_t\mu}_{L^\infty(0,T;L^\infty(\Omega)^{3\times3})}
\norm{\delta\f{H}^k_h}_{\f{L}^2_{\mu_h^k}(\Omega)}^2\\
&\quad+\underbrace{2\sum_{k=1}^m
(\f{f}_h^k-\f{f}_h^{k-1},\delta\f{E}_h^k)_{\f{L}^2(\Omega)}}_{\text{I}}\\
&\quad+\underbrace{2\sum_{k=1}^m
\Bigl(
\f{J}(\cdot,\f{E}_h^{k-1},\f{H}_h^{k-1},t_{k-1})
-\f{J}(\cdot,\f{E}_h^k,\f{H}_h^k,t_k)
+\bigl(\varepsilon(t_{k-1})-\varepsilon(t_k)\bigr)
\delta\f{E}_h^{k-1},
\delta\f{E}_h^k
\Bigr)_{\f{L}^2(\Omega)}}_{\text{II}}\\
&\quad-\underbrace{2\sum_{k=1}^m
\Bigl(
(\mu_h^k-\mu_h^{k-1})\delta\f{H}_h^{k-1}
+\mu_{t,h}^k\f{H}_h^k-\mu_{t,h}^{k-1}\f{H}_h^{k-1},
\delta\f{H}_h^k
\Bigr)_{\f{L}^2(\Omega)}}_{\text{III}}.
\end{align*}
By \hyperref[ass:A2]{(A2)} and Young's inequality, we have
\begin{align*}
\text{I}
&\leq
2\sum_{k=1}^m
\underline{\varepsilon}^{-\frac{1}{2}}
\norm{\f{f}_h^k-\f{f}_h^{k-1}}_{\f{L}^2(\Omega)}
\norm{\delta\f{E}_h^k}_{\f{L}^2_{\varepsilon(t_k)}(\Omega)}\\
&\leq
\sum_{k=1}^m
\underline{\varepsilon}^{-1}N
\norm{\f{f}(t_k)-\f{f}(t_{k-1})}_{\f{L}^2(\Omega)}^2
+\sum_{k=1}^m\frac{1}{N}
\norm{\delta\f{E}_h^k}_{\f{L}^2_{\varepsilon(t_k)}(\Omega)}^2\\
&\leq
\frac{T}{\underline{\varepsilon}}
\norm{\partial_t\f{f}}_{L^2(0,T;\f{L}^2(\Omega))}^2
+\sum_{k=1}^m\frac{1}{N}
\norm{\delta\f{E}_h^k}_{\f{L}^2_{\varepsilon(t_k)}(\Omega)}^2.
\end{align*}
Furthermore, \hyperref[ass:A1]{(A1)} and \hyperref[ass:A4]{(A4)} imply that
\begin{equation}\label{eq:es}
\begin{aligned}
\text{II}\leq 2L\tau\sum_{k=1}^m
&\Bigl(
1+\norm{\f{E}_h^k}_{\f{L}^2(\Omega)}
+\norm{\f{E}_h^{k-1}}_{\f{L}^2(\Omega)}
+\norm{\delta\f{E}_h^k}_{\f{L}^2(\Omega)}
+\norm{\delta\f{E}_h^{k-1}}_{\f{L}^2(\Omega)}\\
&+\norm{\f{H}_h^k}_{\f{L}^2(\Omega)}
+\norm{\f{H}_h^{k-1}}_{\f{L}^2(\Omega)}
+\norm{\delta\f{H}_h^k}_{\f{L}^2(\Omega)}
\Bigr)
\norm{\delta\f{E}_h^k}_{\f{L}^2(\Omega)}.
\end{aligned}
\end{equation}
From the first identity in
\begin{equation}\label{eq:delta-sum}
\f{E}_h^k
=
\f{E}_h^0+\sum_{i=1}^k\tau\delta\f{E}_h^i,
\qquad
\f{H}_h^k
=
\f{H}_h^0+\sum_{i=1}^k\tau\delta\f{H}_h^i,
\end{equation}
Young's inequality and the Cauchy--Schwarz inequality imply that
\begin{align*}
\sum_{k=1}^m\tau
\norm{\f{E}_h^k}_{\f{L}^2(\Omega)}
\norm{\delta\f{E}_h^k}_{\f{L}^2(\Omega)}
&\leq
\sum_{k=1}^m\tau
\norm{\f{E}_h^0}_{\f{L}^2(\Omega)}
\norm{\delta\f{E}_h^k}_{\f{L}^2(\Omega)}
+\tau^2\sum_{k=1}^m\sum_{i=1}^m
\norm{\delta\f{E}_h^i}_{\f{L}^2(\Omega)}
\norm{\delta\f{E}_h^k}_{\f{L}^2(\Omega)}\\
&\leq
\frac{T}{2}\norm{\f{E}_h^0}_{\f{L}^2(\Omega)}^2
+\sum_{k=1}^m
\frac{(1+2T)T}{2N}
\norm{\delta\f{E}_h^k}_{\f{L}^2(\Omega)}^2.
\end{align*}
The other terms in the right hand side of \eqref{eq:es} involving only
$\f{E}_h^k$, $\f{E}_h^{k-1}$, $\delta\f{E}_h^k$ and
$\delta\f{E}_h^{k-1}$ can be treated in the same way. To estimate the remaining terms in II, we use the second identity in \eqref{eq:delta-sum}. By the Cauchy--Schwarz inequality,
\begin{align*}
\norm{\f{H}_h^k}_{\f{L}^2(\Omega)}^2
&\leq
2\norm{\f{H}_h^0}_{\f{L}^2(\Omega)}^2
+2\tau^2
\left(
\sum_{i=1}^k
\norm{\delta\f{H}_h^i}_{\f{L}^2(\Omega)}
\right)^2\\
&\leq
2\norm{\f{H}_h^0}_{\f{L}^2(\Omega)}^2
+2T\tau\sum_{i=1}^k
\norm{\delta\f{H}_h^i}_{\f{L}^2(\Omega)}^2.
\end{align*}
Applying the same estimate to $\f{H}_h^{k-1}$ and summing over $k$, we obtain
\begin{align*}
&\tau\sum_{k=1}^m
\left(
\norm{\f{H}_h^k}_{\f{L}^2(\Omega)}^2
+\norm{\f{H}_h^{k-1}}_{\f{L}^2(\Omega)}^2
\right)\\
&\leq
4T\norm{\f{H}_h^0}_{\f{L}^2(\Omega)}^2
+4T^2\tau\sum_{k=1}^m
\norm{\delta\f{H}_h^k}_{\f{L}^2(\Omega)}^2.
\end{align*}
Thus, Young's inequality yields for the remaining terms in II
\begin{align*}
&2L\tau\sum_{k=1}^m
\left(
\norm{\f{H}_h^k}_{\f{L}^2(\Omega)}
+\norm{\f{H}_h^{k-1}}_{\f{L}^2(\Omega)}
+\norm{\delta\f{H}_h^k}_{\f{L}^2(\Omega)}
\right)
\norm{\delta\f{E}_h^k}_{\f{L}^2(\Omega)}\\
&\leq
4LT\norm{\f{H}_h^0}_{\f{L}^2(\Omega)}^2
+\frac{L(4T^2+1)}{\underline{\mu}}\tau
\sum_{k=1}^m
\norm{\delta\f{H}_h^k}_{\f{L}^2_{\mu_h^k}(\Omega)}^2
+\frac{3L}{\underline{\varepsilon}}\tau
\sum_{k=1}^m
\norm{\delta\f{E}_h^k}_{\f{L}^2_{\varepsilon(t_k)}(\Omega)}^2.
\end{align*}
Next, we consider III. Let us write
$$
\mu_{t,h}^k\f{H}_h^k-\mu_{t,h}^{k-1}\f{H}_h^{k-1}
=
\tau\mu_{t,h}^k\delta\f{H}_h^k
+\left(\mu_{t,h}^k-\mu_{t,h}^{k-1}\right)\f{H}_h^{k-1}.
$$
Thus, \hyperref[ass:A1]{(A1)}, \eqref{eq:qhest2} and Young's inequality give
\begin{align*}
\lvert\text{III}\rvert
&\leq
\tau\norm{\partial_t\mu}_{L^\infty(0,T;L^\infty(\Omega)^{3\times3})}
\sum_{k=1}^m
\left(
\norm{\delta\f{H}_h^{k-1}}_{\f{L}^2(\Omega)}^2
+3\norm{\delta\f{H}_h^k}_{\f{L}^2(\Omega)}^2
\right)\\
&\quad+
\tau\norm{\partial_{tt}\mu}_{L^\infty(0,T;L^\infty(\Omega)^{3\times3})}
\sum_{k=1}^m
\left(
\norm{\f{H}_h^{k-1}}_{\f{L}^2(\Omega)}^2
+\norm{\delta\f{H}_h^k}_{\f{L}^2(\Omega)}^2
\right).
\end{align*}
Using the second identity in \eqref{eq:delta-sum}, the same arguments as above and a shift of the summation index, we further obtain
\begin{align*}
\lvert\text{III}\rvert
&\leq
\tau\norm{\partial_t\mu}_{L^\infty(0,T;L^\infty(\Omega)^{3\times3})}
\norm{\delta\f{H}_h^0}_{\f{L}^2(\Omega)}^2\\
&\quad+
2T\norm{\partial_{tt}\mu}_{L^\infty(0,T;L^\infty(\Omega)^{3\times3})}
\norm{\f{H}_h^0}_{\f{L}^2(\Omega)}^2\\
&\quad+
\frac{1}{\underline{\mu}}
\left[
4\norm{\partial_t\mu}_{L^\infty(0,T;L^\infty(\Omega)^{3\times3})}
+(2T^2+1)
\norm{\partial_{tt}\mu}_{L^\infty(0,T;L^\infty(\Omega)^{3\times3})}
\right]
\tau\sum_{k=1}^m
\norm{\delta\f{H}_h^k}_{\f{L}^2_{\mu_h^k}(\Omega)}^2.
\end{align*}
By the definition of $\delta\f{H}_h^0$, \hyperref[ass:A1]{(A1)}, Lemma \ref{lemma:initial-bounded}, \eqref{eq:qhest2} and \eqref{eq:Pi_h proj}, all initial terms appearing above are uniformly bounded. Collecting the estimates for I--III, we see that there exist constants $\alpha,\beta>0$, not depending on $N$ and $h$, such that
\begin{align*}
&\frac{1}{2}
\norm{\delta\f{E}_h^m}_{\f{L}^2_{\varepsilon(t_m)}(\Omega)}^2
+\frac{1}{2}
\norm{\delta\f{H}_h^m}_{\f{L}^2_{\mu_h^m}(\Omega)}^2\\
&\leq
\alpha+
\sum_{k=1}^m\frac{\beta}{2N}
\left(
\frac{1}{2}
\norm{\delta\f{E}_h^k}_{\f{L}^2_{\varepsilon(t_k)}(\Omega)}^2
+\frac{1}{2}
\norm{\delta\f{H}_h^k}_{\f{L}^2_{\mu_h^k}(\Omega)}^2
\right).
\end{align*}
Since $N>\beta$, the term corresponding to $k=m$ on the right hand side can be absorbed into the left hand side. The discrete Gronwall inequality then yields
\begin{align*}
\frac{1}{2}
\norm{\delta\f{E}_h^m}_{\f{L}^2_{\varepsilon(t_m)}(\Omega)}^2
+\frac{1}{2}
\norm{\delta\f{H}_h^m}_{\f{L}^2_{\mu_h^m}(\Omega)}^2
\leq
2\alpha\exp(\beta)
\end{align*}
for all $m\in\{1,\dots,N\}$. Using \eqref{eq:delta-sum}, we also obtain uniform bounds for $\f{E}_h^m$ and $\f{H}_h^m$ in $\f{L}^2(\Omega)$. Moreover, the second equality in (P$_{N,h}$) implies
$$
\curl\,\f{E}_h^m
=
-\mu_h^m\delta\f{H}_h^m-\mu_{t,h}^m\f{H}_h^m,
$$
and hence, by \hyperref[ass:A1]{(A1)}, \eqref{eq:qhest2} and the quadratic-form estimate for $\mu_h^m$,
$$
\norm{\curl\,\f{E}_h^m}_{\f{L}^2(\Omega)}
\leq
\overline{\mu}\norm{\delta\f{H}_h^m}_{\f{L}^2(\Omega)}
+
\norm{\partial_t\mu}_{L^\infty(0,T;L^\infty(\Omega)^{3\times3})}
\norm{\f{H}_h^m}_{\f{L}^2(\Omega)},
$$
which shows that $\norm{\curl\,\f{E}_h^m}_{\f{L}^2(\Omega)}$ is uniformly bounded as well. Finally, by \hyperref[ass:A5]{(A5)}, we have
\begin{align*}
\int_{\partial \Omega}
\f{b}(\cdot,\boldsymbol{\nu}\times \f{E}_h^m)
\cdot\boldsymbol{\nu}\times \f{E}_h^m\,dS
&\geq
\underline b \int_{\lvert\boldsymbol{\nu}\times\f{E}_h^m\rvert\geq K}
\lvert\boldsymbol{\nu}\times\f{E}_h^m\rvert^2\,dS
\end{align*}
and since the left-hand side can be uniformly bounded in terms of $\alpha$ and $\beta$ by (P$_{N,h}$), \hyperref[ass:A4]{(A4)} and the previously derived estimates, the conclusion follows.
\end{proof}

\end{document}